\documentclass[12pt, reqno]{amsart}

 \usepackage{graphicx}
\usepackage{tikz}
\usepackage{tcolorbox}
\usepackage{subcaption}

\usetikzlibrary{shapes, positioning,patterns}
\usetikzlibrary{trees}
\usetikzlibrary{fit}
\usepackage{float}
\usepackage{ulem}
\usepackage[justification=centering]{caption}
\usepackage[all]{xy}
\usepackage{amsmath}
\usepackage{amsfonts}
\usepackage{amssymb}
\usepackage[utf8]{inputenc}
\usepackage{cite}
\usepackage{color}

\usepackage{epic}

\usepackage{lmodern}

\numberwithin{equation}{section}

\newtheorem{thm}{Theorem}[section]
\newtheorem{defi}[thm]{Definition}
\newtheorem{lem}[thm]{Lemma}
\newtheorem{prop}[thm]{Proposition}

\newtheorem{pf}{proof}[section]
\newtheorem{rem}[thm]{Remark}

\newtheorem{prob}[thm]{Problem}

\def\qed{\hfill \rule{4pt}{7pt}}

\allowdisplaybreaks[4]

\begin{document}

\title[Modular   Nahm Sums for symmetrizable matrices]{Modular
Nahm sums for symmetrizable matrices\\[5pt]  of indices $({2,\ldots, 2},1)$ and $({1,\ldots, 1},2)$}

\author[J.Q.D. Du]{Julia Q.D. Du}
\address[ Julia Q.D. Du]{School of Mathematical Sciences, Hebei Key Laboratory of Computational Mathematics and Applications,  Hebei Research Center of the Basic discipline Pure Mathematics, Hebei Normal University, Shijiazhuang 050024, P.R. China}
\email{qddu@hebtu.edu.cn}

\author[K.Q. Ji]{Kathy Q. Ji}
\address[Kathy Q. Ji]{Center for Applied Mathematics and KL-AAGDM, Tianjin University, Tianjin 300072, P.R. China}
\email{kathyji@tju.edu.cn}

\author[E.Y.Y. Shen]{Erin Y.Y. Shen}
\address[Erin Y.Y. Shen]{School of Mathematics, Hohai University, Nanjing 210098,  P. R. China}
\email{shenyy@hhu.edu.cn}

\author[C.X.Y. Xu]{Clara X.Y. Xu}
\address[Clara X.Y. Xu]{Center for Applied Mathematics and KL-AAGDM, Tianjin University, Tianjin 300072, P.R. China}
\email{claraxu@tju.edu.cn}

 \date{\today}
\begin{abstract}
In this paper, we present three families of modular   Nahm sums for symmetrizable matrices with arbitrary rank
$r\geq 2$  of indices $({2,\ldots, 2},1)$ and $({1,\ldots, 1},2)$. Specifically, the cases corresponding to
 $r = 2$ and $r = 3$ of these   families
have been previously demonstrated by Mizuno, Warnaar, and B. Wang--L. Wang.
 Building upon these three families,   we construct two vector-valued automorphic forms,
one of which is a vector-valued modular function when $r$ is odd.
\end{abstract}

\maketitle

\section{Introduction}

A Nahm sum or Nahm series was defined by Nahm \cite{Nahm-1995,Nahm-1994,Nahm-2007} as the following class of \(q\)-hypergeometric series:
    \begin{equation}
    f_{A,\boldsymbol{b},c}(q):=\sum_{\boldsymbol{n}=(n_1,\ldots,n_r)^T\in\mathbb{N}^r}\frac{q^{\frac{1}{2}\boldsymbol{n}^TA\boldsymbol{n}+\boldsymbol{n}^T\boldsymbol{b}+{c}}}{(q;q)_{n_1}\cdots(q;q)_{n_r}},
    \end{equation}
    where \(r\geq1\) is a positive integer, \(A \in\mathbb{Q}^{r\times r} \) is a real positive definite symmetric matrix,   \(\boldsymbol{b}\in\mathbb{Q}^r\) is a vector and \(c\in\mathbb{Q}\) is a scalar.
Here and throughout the paper,  we  adopt the standard notation on $q$-series \cite{Andrews-1976, Gasper-2004}:
\[
(a;q)_\infty =\prod_{n=0}^\infty (1-aq^n),
\]
\[ (a;q)_n=\frac{(a;q)_\infty}{(aq^n;q)_\infty},\\[3pt]
\]
\[(a_1,a_2,\ldots, a_k;q)_n=(a_1;q)_n(a_2;q)_n\cdots(a_k;q)_n,\]
and
\[(a_1,a_2,\ldots, a_k;q)_\infty=(a_1;q)_\infty(a_2;q)_\infty\cdots(a_k;q)_\infty.
\]

Nahm \cite{Nahm-2007} proposed the following famous  problem: find all $(A,\boldsymbol{b},c)$  with rational entries such that \(f_{A,\boldsymbol{b},c}(q)\) is a modular function, where $q=e^{2\pi i\tau}$
and $\tau\in \mathbb{H} := \{\tau \in \mathbb{C} : \text{Im}\ \tau > 0\}$. Such a triple $(A,\boldsymbol{b},c)$ is called a modular triple.    Recall that  a modular function is certain meromorphic function defined on $\mathbb{H}$ satisfying modular transformation and certain holomorphic/meromorphic conditions (see Definition \ref{defi-modu}).
Nahm also \cite{Nahm-2007} made a conjecture which provides a
necessary and sufficient condition about the matrix \( A \) so that it is the matrix part of some modular triples.  The conjecture  is formulated in terms of the Bloch group and a system of polynomial equations induced by \( A \), see Zagier \cite[p. 43]{Zagier-2007} for more details.

In 2007, Zagier \cite{Zagier-2007} made an important progress towards Nahm's problem. He confirmed Nahm's conjecture for $r = 1$ by showing
that there are exactly seven modular triples.
For $r\geq 2$, Nahm's conjecture is known to be false in general.
Vlasenko and Zwegers \cite{Vlasenko-Zwegers-2011}
provided a counterexample when $r = 2$,
that is, for $A=\left(~3/4~~-1/4 \atop -1/4 ~~~3/4 \right)$,
there exists a $\boldsymbol{b}$ and $c$ such that $f_{A,\boldsymbol{b},c}(q)$ is modular, but
not all solutions of Nahm's equation give a torsion element in the Bloch group.
For  \( r = 2 \) and \( 3 \), Zagier \cite{Zagier-2007} found many possible modular triples, which were later affirmed by Cao--Rosengren--Wang \cite{Cao-Rosengren-Wang-2024},
Cherednik--Feigin \cite{Cherednik-Feigin-2013},
Vlasenko--Zwegers \cite{Vlasenko-Zwegers-2011}, Wang \cite{Wang-2024a, Wang-2024b}
and Zagier \cite{Zagier-2007}.

According to physical predictions, Zagier \cite{Zagier-2007} expects
that if a Nahm sum associated with a matrix $A$ is modular,
then the collection of all modular Nahm sums for $A$ spans
a vector space that is invariant under the action of $SL_2(\mathbb{Z})$
in the bosonic case, or at least under $\Gamma(2)$ in the fermionic
case. Note that in the language of the physicists,
 for a sum-product identity,
the sum side is called a ``fermionic'' representation,
the product side is called a ``bosonic'' representation, see \cite{Sill-2004}.
Furthermore, Zagier  provided explicit transformation formulas for the vector-valued function arising from the Rogers--Ramanujan identities under the action of
$SL_2(\mathbb{Z})$. Building on this, Milas and Wang \cite{Milas-Wang-2024}
established analogous transformation laws for vector-valued functions composed of generalized tadpole Nahm sums.

Recently,  Mizuno \cite{Mizuno-2025} considered generalized Nahm sums associated with symmetrizable matrices, which take the form:
\begin{equation}\label{defi:GNahm}
\widetilde{f}_{A,\boldsymbol{b},{c}, \boldsymbol{d}}(q):=\sum_{\boldsymbol{n}=(n_1,\ldots,n_r)^T\in\mathbb{N}^r}\frac{q^{\frac{1}{2}\boldsymbol{n}^TAD\boldsymbol{n}+\boldsymbol{n}^T\boldsymbol{b}+c}}{(q^{d_1};q^{d_1})_{n_1}\cdots(q^{d_r};q^{d_r})_{n_r}},
\end{equation}
where  \(\boldsymbol{d}=(d_1,\ldots,d_r)\in\mathbb{Z}^r_{>0},\) \(\boldsymbol{b}\in\mathbb{Q}^r\) is a vector and \(c\in\mathbb{Q}\) is a scalar. Denote by  \(D:=\text{diag} (d_1,\ldots,d_r)_{r\times r}\) the  $r\times r$ diagonal matrix with diagonal entries $d_1,\ldots, d_r$. A matrix
    \(A\in\mathbb{Q}^{r\times r}\) is said to be  a symmetrizable matrix with the symmetrizer \(D=\text{diag} (d_1,\ldots,d_r)_{r\times r}\) if \(AD\) is symmetric positive definite.

Following the notation in  \cite{Mizuno-2025} and \cite{Wang-Wang-2025-1}, we call  \(\widetilde{f}_{A,\boldsymbol{b}, {c}, \boldsymbol{d}}(q)\) in \eqref{defi:GNahm} a  Nahm sum for the symmetrizable matrix $A$ of index $(d_1,\ldots,d_r)$.  When \(\widetilde{f}_{A,\boldsymbol{b}, {c}, \boldsymbol{d}}(q)\) is a modular function,  the quadruple \((A,\boldsymbol{b},c,\boldsymbol{d})\) is  called a  modular quadruple or   such a triple \((A,\boldsymbol{b},c)\)   is called a  modular triple of index $(d_1,\ldots,d_r)$.

Generalized Nahm sums of this kind are common in partition identities and the study of affine Lie algebras.  For example, one of Capparelli's partition identities \cite{Capparelli-1993} asserts that
\begin{align}\label{Cap-ide}
\sum_{n_1, n_2 \geq 0} \frac{q^{2n_1^2 + 6n_1n_2 + 6n_2^2}}{(q; q)_{n_1} (q^3; q^3)_{n_2}} = (-q^2, -q^3, -q^4, -q^6; q^6)_{\infty},
\end{align}
which was first discovered by Capparelli in purely combinatorial form
via the theory of affine Lie algebras. The double-sum representation on the left-hand side of \eqref{Cap-ide} was discovered by
 Kanade--Russell \cite{Kanade-Russell-2019}
and Kur\c{s}ung\"{o}z \cite{Kursungoz-2019}
independently.
Notice that \eqref{Cap-ide} reveals that  the generalized Nahm sum \( \widetilde{f}_{A,\boldsymbol{b},c,\boldsymbol{d}}(q) \) associated with \( A = \begin{pmatrix} 4 & 2 \\ 6 & 4 \end{pmatrix} \), \( \boldsymbol{b} = (0, 0)^{\mathrm{T}} \), \( c = -1/24 \) and \( \boldsymbol{d} = (1, 3) \) is modular.

Mizuno  \cite{Mizuno-2025}  found that the theory of Nahm sums for symmetric matrices can be analogously applied to  Nahm sums for symmetrizable matrices. For example,  following the strategy of Zagier \cite{Zagier-2007}, Mizuno  \cite{Mizuno-2025}  proposed potential  modular triples of  indices
$(1,2)$, $(1,3)$ and $(1,4)$, and  modular triples of indices  \((1, 1, 2)\) and \((2, 2, 1)\),  which have been  confirmed by
B. Wang and L. Wang \cite{Wang-Wang-2024, Wang-Wang-2025-1, Wang-Wang-2025-2} recently.

Mizuno \cite[Eqs. (45), (54)]{Mizuno-2025}
conjectured two vector-valued transformation formulas involving generalized Nahm sums
for symmetrizable matrices. In particular,   Mizuno observed  that there seem to be ``Langlands dual'' pairs of modular Nahm
sums  on the rank 3 case.  He conjectured   two Nahm sums associated with
\[A=\begin{pmatrix}
    1&0&1/2\\
    0&2&1\\
    1&2&2
\end{pmatrix} \quad \text{and} \quad A^{\vee}=\begin{pmatrix}
    1&0&1\\
    0&2&2\\
    1/2&1&2
\end{pmatrix}\]
 which are related by taking the transpose,
are related by the modular $S$-transformation.
 Mizuno's conjecture was subsequently proved by B. Wang--L. Wang \cite{Wang-Wang-2025-2, Wang-Wang-2025-1}.

While Nahm's problem has been studied in numerous examples
for rank $r\leq 3$,
a solution for the general case remains elusive. In this paper, we present three families of modular   Nahm sums for symmetrizable matrices with arbitrary rank $r\geq 2$  of indices $({2,\ldots, 2},1)$ and $({1,\ldots, 1},2)$. Based on these three families of generalized Nahm sums,  we construct two vector-valued automorphic forms, one of which is a vector-valued
modular function if $r$ is odd ( see Section \ref{sect:trans}   for the relevant definitions).

\begin{thm} \label{main1}
For  $r\geq 2$ and $0\leq j\leq r$, the Nahm sums \(\widetilde{f}_{A,\boldsymbol{b}_j, {c}_j, \boldsymbol{d}}(q^{32r-8})\)   for the symmetrizable matrix $A$ of index $\boldsymbol{d}=({2,\ldots, 2}, 1)$ are modular functions for $\Gamma_1(128(4r-1)^2)$, where

\begin{align*}
   &A=\begin{pmatrix}
    1 & 0 & 0 & 0 & \cdots & 0 & 0 & 1\\
    0 & 2 & 2 & 2 & \cdots & 2 & 2 & 2\\
    0 & 2 & 4 & 4 & \cdots & 4 & 4 & 4 \\
    0 & 2 & 4 & 6 & \cdots & 6 & 6 & 6 \\
    \vdots & \vdots & \vdots & \vdots & \ddots & \vdots & \vdots & \vdots\\
    0 & 2 & 4 & 6 & \cdots & 2(r-3) & 2(r-3)& 2(r-3)\\
    0 & 2 & 4 & 6 & \cdots & 2(r-3) & 2(r-2) & 2(r-2)\\
    \frac{1}{2} & 1 & 2 & 3 & \cdots & r-3& r-2 & r-1
\end{pmatrix}_{r\times r}, \\[10pt]
&\boldsymbol{b}_0=(0,\ldots,0)^T_{1\times r}, \quad c_0=\frac{5-4 r}{32 r-8},\\
&\text{and }  \text{for } 1\leq j\leq r,\\
&\boldsymbol{b}_j=(1,0,\ldots,0,2,4,\ldots,2(r-1-j ),r-j )^T_{1\times r},    \quad {c}_j=\frac{(4 r-4j-1)^2}{32 r-8}.
\end{align*}
\end{thm}

\begin{thm} \label{main2}
For  $r\geq 2$ and $0\leq j\leq r$, the Nahm sums \(\widetilde{f}_{B,\boldsymbol{b}_j, {c}_j, \boldsymbol{d}}(q^{32r-24})\)   for the symmetrizable matrix $B$ of index $\boldsymbol{d}=({2,\ldots, 2}, 1)$ are modular functions for $\Gamma_1(64(4r-3)^2)$, where

\begin{align*}
   &B=\begin{pmatrix}
    1 & 0 & 0 & 0 & \cdots & 0 & 0 & 1\\
    0 & 2 & 2 & 2 & \cdots & 2 & 2 & 2\\
    0 & 2 & 4 & 4 & \cdots & 4 & 4 & 4\\
    0 & 2 & 4 & 6 & \cdots & 6 & 6 & 6\\
    \vdots & \vdots & \vdots & \vdots & \ddots & \vdots & \vdots & \vdots\\
    0 & 2 & 4 & 6 & \cdots & 2(r-3) & 2 (r-3) & 2(r-3)\\
    0 & 2 & 4 & 6 & \cdots & 2(r-3) & 2(r-2) & 2(r-2)\\
    \frac{1}{2} & 1 & 2 & 3 & \cdots & r-3 & r-2 & r-\frac{1}{2}
\end{pmatrix}_{r\times r},
\\[10pt]
&\boldsymbol{b}_0=\left(0,\ldots,0,-\frac{1}{2}\right)^T_{1\times r}, \quad c_0=\frac{7-4 r}{32 r-24},\\
&\text{and }  \text{for } 1\leq j\leq r,\\
&\boldsymbol{b}_j=\left(1,0,\ldots,0,2,4,\ldots,2(r-1-j ),r-j -1\right)^T_{1\times r},    \quad {c}_j=\frac{(4 r-4j-3)^2}{32 r-24}.
\end{align*}
\end{thm}

\begin{thm}\label{main3}
For  $r\geq 2$ and $0\leq j\leq r$, the Nahm sums \(\widetilde{f}_{A^{\vee},\boldsymbol{b}_j, {c}_j, \boldsymbol{d}^{\vee}}(q^{32r-8})\)   for the symmetrizable matrix $A^{\vee}$ of index
$\boldsymbol{d}^{\vee}=(1,\ldots,1,2)$ are modular functions for $\Gamma_1(128(4r-1)^2)$, where
\begin{align*}
   &A^{\vee}=\begin{pmatrix}
    1 & 0 & 0 & 0 & \cdots & 0 & 0 & \frac{1}{2}\\
    0 & 2 & 2 & 2 & \cdots & 2 & 2 & 1\\
    0 & 2 & 4 & 4 & \cdots & 4 & 4 & 2\\
    0 & 2 & 4 & 6 & \cdots & 6 & 6 & 3\\
    \vdots & \vdots & \vdots & \vdots & \ddots & \vdots & \vdots & \vdots \\
    0 & 2 & 4 & 6 & \cdots & 2(r-3) & 2(r-3) & r-3\\
    0 & 2 & 4 & 6 & \cdots & 2(r-3) & 2(r-2) & r-2\\
    1 & 2 & 4 & 6 & \cdots & 2(r-3) & 2(r-2) & r-1
\end{pmatrix}_{r\times r},\\[10pt]
&\textbf{b}_0=(1,1,2,3,\ldots,r-1)^T_{1\times r},\quad c_0=\frac{8r^2-14r+5}{32r-8},\\
&\text{and for }1\leq j\leq r,\\
&\boldsymbol{b}_j=(0,\ldots,0,1,2,\ldots,r-j)^T_{1\times r},
\quad c_j=\frac{8(r-j)^2+4j-6r+1}{32r-8}.
\end{align*}
\end{thm}

We now provide several remarks on Theorems   \ref{main1}--\ref{main3}.

\begin{rem}\label{rem-iden}

{\rm (1)} The symmetrizable matrix  $A$   in Theorem \ref{main1}  is the transpose of  the symmetrizable matrix $A^{\vee}$ in Theorem \ref{main3}.

{\rm (2)}
The cases $r=2$ and $r=3$ in Theorem \ref{main1} and Theorem \ref{main2}
were conjectured by Mizuno (see \cite[Table 1, Table 2, Table 3]{Mizuno-2025}).
The case $r=2$ in Theorem \ref{main1} was confirmed by B. Wang--L. Wang \cite{Wang-Wang-2025-1}.
For $r=2$ in Theorem \ref{main2},
the case $b_0$ was established by Warnaar \cite[(5.14)]{Warnaar-2003-2}, as observed by Mizuno,
while the remaining cases were proved by Mizuno \cite{Mizuno-2025}.
The cases $r=3$ in both Theorem \ref{main1}
and Theorem \ref{main2} were confirmed by B. Wang--L. Wang \cite{Wang-Wang-2025-2}.

{\rm (3)} The matrices   $A$ and $B$ in Theorem \ref{main1} and Theorem \ref{main2}  are  symmetrizable matrices with the symmetrizer \({D}={\rm diag}(2,\ldots,2,1)_{r\times r}\)   since  the positive definiteness of the symmetric matrices
 $AD$ and $BD$  given in \eqref{eq-1.3} and \eqref{eq-1.4}    are easily verified.
 \begin{align}
      &AD=\begin{pmatrix}
    2 & 0 & 0 & 0 &  \cdots & 0 & 0 & 1\\
    0 & 4 & 4 & 4 & \cdots & 4 & 4 & 2\\
    0 & 4 & 8 & 8 & \cdots & 8 & 8 & 4\\
    0 & 4 & 8 & 12 & \cdots & 12 & 12 & 6\\
    \vdots & \vdots & \vdots & \vdots & \ddots & \vdots & \vdots\\
    0 & 4 & 8 & 12 & \cdots &4(r-3) & 4(r-3) & 2(r-3)\\
    0 & 4 & 8 & 12 & \cdots & 4(r-3) & 4(r-2) & 2(r-2)\\
    1 & 2 & 4 & 6 & \cdots & 2(r-3) & 2(r-2) & r-1
\end{pmatrix}_{r\times r},\label{eq-1.3}\\[20pt]
&BD=\begin{pmatrix}
    2 & 0 & 0 & 0 & \cdots & 0 & 0 & 1\\
    0 & 4 & 4 & 4 & \cdots & 4 & 4 & 2\\
    0 & 4 & 8 & 8 & \cdots & 8 & 8 & 4\\
    0 & 4 & 8 & 12 & \cdots & 12 & 12 & 6\\
    \vdots & \vdots & \vdots & \vdots & \ddots & \vdots & \vdots & \vdots\\
    0 & 4 & 8 & 12 & \cdots & 4(r-3) & 4(r-3) & 2(r-3)\\
    0 & 4 & 8 & 12 & \cdots & 4(r-3) & 4(r-2) & 2(r-2)\\
    1 & 2 & 4 & 6 & \cdots & 2(r-3) & 2(r-2) & r-\frac{1}{2}
    \end{pmatrix}_{r\times r}.\label{eq-1.4}
 \end{align}

{\rm (4)} Similarly, the matrix $A^{\vee}$ in Theorem \ref{main3} is a symmetrizable matrix with the symmetrizer \({D}^{\vee}={\rm diag}(1,\ldots,1,2)_{r\times r}\) since
$A^{\vee}{D}^{\vee}$   as shown in \eqref{eq-1.5}  is a symmetric positive definite matrix.
\begin{align}
       &A^{\vee}{D}^{\vee}=\begin{pmatrix}
    1 & 0 & 0 & 0 & \cdots & 0 & 0 & 1\\
    0 & 2 & 2 & 2 & \cdots & 2 & 2 & 2\\
    0 & 2 & 4 & 4 & \cdots & 4 & 4 & 4\\
    0 & 2 & 4 & 6 & \cdots & 6 & 6 & 6\\
    \vdots & \vdots & \vdots & \vdots & \ddots & \vdots & \vdots & \vdots \\
    0 & 2 & 4 & 6 & \cdots & 2(r-3) & 2(r-3) & 2(r-3)\\
    0 & 2 & 4 & 6 & \cdots & 2(r-3) & 2(r-2) & 2(r-2)\\
    1 & 2 & 4 & 6 & \cdots & 2(r-3) & 2(r-2) & 2(r-1)
\end{pmatrix}_{r\times r}.\label{eq-1.5}
\end{align}

\end{rem}

To prove Theorems \ref{main1},  \ref{main2} and \ref{main3},  we   establish the
Rogers--Ramanujan type identities stated in Theorem \ref{main1a}.
Each multi-sum appearing in Theorem \ref{main1a} corresponds
to the generalized Nahm sums $\widetilde{f}_{A,\boldsymbol{b},c,\boldsymbol{d}}(q)$,
 specified in Theorems  \ref{main1}, \ref{main2} and \ref{main3},
whereas  each product term can be expressed as a generalized eta-quotient
or a sum of generalized eta-quotients, whose modularity can be determined using the criterion established by Robins \cite[Theorem~3]{Robins-1994}. It should be noted that the identities \eqref{eq-1.11} and \eqref{eq-1.12} in Theorem \ref{main1a}  are equivalent to those due to B. Wang--L. Wang \cite[Corollary 4.2 and  Corollary 4.3]{Wang-Wang-2024} and we include them in Theorem \ref{main1a}  for completeness.

\begin{thm} \label{main1a}
For  $r\geq 2$,
\begin{align}
    (1) &\sum_{\boldsymbol{n}=(n_1,\ldots,n_r)^T\in\mathbb{N}^r}\frac{q^{\frac{1}{2}\boldsymbol{n}^TAD\boldsymbol{n}+\boldsymbol{n}^T\boldsymbol{b}_0}}{(q^2;q^2)_{n_1}(q^2;q^2)_{n_2}\cdots(q^2;q^2)_{n_{r-1}}(q;q)_{n_r}}\notag\\[10pt]
        &\quad =\frac{(q^{8r-4},q^{8r},q^{16r-4};q^{16r-4})_\infty}{(q,q^3,q^4;q^4)_\infty}
        +q^{\frac{r-1}{2}}\frac{(-q,q^{4r-2},-q^{4r-1};-q^{4r-1})_\infty}{(q^2,q^2,q^4;q^4)_\infty},\label{eq-1.7}  \\[10pt]
&\text{where }AD \text{ is given by \eqref{eq-1.3} and }\boldsymbol{b}_0=(0,0,\cdots,0,0)^T_{1\times r}.\notag\\[20pt]
(2) & \text{ For }  1\leq j\leq r, \notag \\[20pt]
&\sum_{\boldsymbol{n}=(n_1,\ldots,n_r)^T\in\mathbb{N}^r}\frac{q^{\frac{1}{2}\boldsymbol{n}^TAD\boldsymbol{n}+\boldsymbol{n}^T\boldsymbol{b}_j}}{(q^2;q^2)_{n_1}\cdots(q^2;q^2)_{n_{r-1}}(q;q)_{n_r}}\label{eq-1.8} \\
        &=\frac{(q^{2j},-q^{4r-2j-1},-q^{4r-1};-q^{4r-1})_\infty}{(q^2,q^2,q^4;q^4)_\infty}
        +q^{\frac{3r-2j-1}{2}}\frac{(q^{4j},q^{16r-4j -4},q^{16r-4};q^{16r-4})_\infty}{(q,q^3,q^4;q^4)_\infty},\notag \\[10pt]
&\text{where }AD \text{ is given by \eqref{eq-1.3} and }\boldsymbol{b}_j=(1,0,\cdots,0,2,4,\cdots,2(r-1-j ),r-j )^T_{1\times r}.\notag
\end{align}
\begin{align}
(3) &\sum_{\boldsymbol{n}=(n_1,\ldots,n_r)^T\in\mathbb{N}^r}\frac{q^{\frac{1}{2}\boldsymbol{n}^TBD\boldsymbol{n}+\boldsymbol{n}^T\boldsymbol{b}_0}}{(q^2;q^2)_{n_1}(q^2;q^2)_{n_2}\cdots(q^2;q^2)_{n_{r-1}}(q;q)_{n_r}}\notag\\[10pt]
        &=\frac{(q^{8r-8},q^{8r-4},q^{16r-12};q^{16r-12})_\infty}{(q,q^3,q^4;q^4)_\infty}
        +q^{\frac{2r-3}{4}}\frac{(-q,q^{4r-4},-q^{4r-3};-q^{4r-3})_\infty}{(q^2,q^2,q^4;q^4)_\infty},\label{eq-1.9}
         \end{align}
\begin{align}
&\text{where }BD \text{ is given by \eqref{eq-1.4} and }\boldsymbol{b}_0=(0,0,\cdots,0,-\frac{1}{2})^T_{1\times r}.\notag\\[20pt]
(4)& \text{ For }  1\leq j\leq r, \notag \\[20pt]
&\sum_{\boldsymbol{n}=(n_1,\ldots,n_r)\in\mathbb{N}^r}\frac{q^{\frac{1}{2}\boldsymbol{n}^TBD\boldsymbol{n}+\boldsymbol{n}^T\boldsymbol{b}_j}}{(q^2;q^2)_{n_1}(q^2;q^2)_{n_2}\cdots(q^2;q^2)_{n_{r-1}}(q;q)_{n_r}}\label{eq-1.10}\\[10pt]
        &=\frac{(q^{2j},-q^{4r-3-2j},-q^{4r-3};-q^{4r-3})_\infty}{(q^2,q^2,q^4;q^4)_\infty}
        +q^{\frac{6r-4j -5}{4}}\frac{(q^{4j},q^{16r-4j -12},q^{16r-12};q^{16r-12})_\infty}{(q,q^3,q^4;q^4)_\infty},\notag\\[20pt]
&\text{where }BD \text{ is given by \eqref{eq-1.4} and }\boldsymbol{b}_j=(1,0,\cdots,0,2,4,\cdots,2(r-j -1),r-j -1)^T_{1\times r}.\notag
\end{align}
\begin{align}
(5)&{\sum_{{\boldsymbol n}=(n_1,\ldots,n_r)^T\in\mathbb{N}^r}\frac{q^{\frac{1}{2}\boldsymbol{n}^TA^{\vee}{D}^{\vee} \boldsymbol{n}+\boldsymbol{n}^T\boldsymbol{b}_0}}{(q;q)_{n_1}(q;q)_{n_2}\cdots(q;q)_{n_{r-1}}(q^2;q^2)_{n_r}}=\frac{(q^\frac{1}{2},q^{2r-1},q^{2r-\frac{1}{2}};q^{2r-\frac{1}{2}})_\infty}{(q^\frac{1}{2},q^\frac{3}{2},q^2;q^2)_\infty}},\notag\\\label{eq-1.11}&\\[8pt]
&\text{where }
A^{\vee}{D}^{\vee}  \text{ is given by } \eqref{eq-1.5} \text{ and } \textbf{b}_0=(1,1,2,3,\ldots,r-1)^T_{1\times r}.\notag\\[8pt]
(6)&\text{ For }1\leq j\leq r,\notag\\[8pt]
&\sum_{\boldsymbol{n}=(n_1,\ldots,n_r)^T\in\mathbb{N}^r}\frac{q^{\frac{1}{2}\boldsymbol{n}^TA^{\vee}{D}^{\vee}\boldsymbol{n}+\boldsymbol{n}^T\boldsymbol{b}_j}}{(q;q)_{n_1}\cdots(q;q)_{n_{r-1}}(q^2;q^2)_{n_r}}=\frac{(q^j,q^{2r-j-\frac{1}{2}},q^{2r-\frac{1}{2}};q^{2r-\frac{1}{2}})_\infty}{(q^\frac{1}{2},q^\frac{3}{2},  q^2;q^2)_\infty},\label{eq-1.12}\\[8pt]
&\text{where }
A^{\vee}{D}^{\vee} \text{ is given by } \eqref{eq-1.5} \text{ and } \boldsymbol{b}_j=(0,\ldots,0,1,2,\ldots,r-j)^T_{1\times r}.\notag
\end{align}
\end{thm}

 Based on these three families of generalized Nahm sums,  we  introduce the following two vector-valued functions $\boldsymbol{G}_{4r-2}(\tau)$ and $\boldsymbol{H}_{4r-4}(\tau)$:
 \begin{align*}
\boldsymbol{G}_{4r-2}(\tau)&=(g_1(\tau), g_2(\tau), \ldots, g_{2r-1}(\tau),
g^{\vee}_1(\tau), g^{\vee}_2(\tau), \ldots, g^{\vee}_{2r-1}(\tau))^T
\end{align*}
and
\begin{align*}
\boldsymbol{H}_{4r-4}(\tau)=(h_{1}(\tau), h_2(\tau),
\ldots, h_{2r-2}(\tau),
h_1^{\vee}(\tau), h^{\vee}_2(\tau),\ldots, h_{2r-2}^{\vee}(\tau))^T,
\end{align*}
where
\begin{align}
g_j(\tau)&=q^{\frac{(4r-4j+1)^2}{32r-8}}\left(q^{\frac{2j-r-1}{2}}
\frac{(q^{8r-4j},q^{8r+4j-4},q^{16r-4};q^{16r-4})_\infty}{(q,q^3,q^4;q^4)_\infty}\right.\notag\\[5pt]
&\quad \left.+\frac{(-q^{2j-1},q^{4r-2j},-q^{4r-1};-q^{4r-1})_\infty}{(q^2,q^2,q^4;q^4)_\infty}\right),\label{gv-def}\\[5pt]
g^{\vee}_j(2\tau)&=\frac{1}{2}q^{\frac{2j^2-2j-2r+1}{16r-4}}
\left(\frac{(q^{2r-j},q^{2r+j-1},q^{4r-1};q^{4r-1})_\infty}{(q,q^3,q^4;q^4)_\infty}\right.\notag\\[10pt]
&\quad \left. +(-1)^{\frac{j(j-1)}{2}}\frac{((-q)^{2r-j},(-q)^{2r+j-1},-q^{4r-1};-q^{4r-1})_\infty}
{(-q,-q^3,q^4;q^4)_\infty}\right),\label{g-def}
\end{align}
\begin{align}
h_j(\tau)&=q^{\frac{(4 j-4 r+1)^2}{32 r-24}}\left(
q^{\frac{4 j-2 r-1}{4}}
\frac{(q^{4(2r-1-j)}, q^{8r+4j-8},q^{16r-12}; q^{16r-12})_\infty}{(q,q^3,q^4;q^4)_\infty}\right.\notag\\[10pt]
&\quad \left.
+
\frac{(q^{2(2r-j-1)},-q^{2j-1},-q^{4r-3};-q^{4r-3})_\infty}{(q^2,q^2,q^4;q^4)_\infty}\right),\label{hv-def}\\
h^{\vee}_j(2\tau)&=\frac{1}{2}q^{\frac{j^2-j-r+1}{8r-6}}
\left(\frac{(q^{2r-j-1},q^{2r+j-2},q^{4r-3};q^{4r-3})_\infty}
{(q,q^3,q^4;q^4)_\infty}\right.\notag\\[5pt]
&\quad \left.+
\frac{((-q)^{2r-j-1},(-q)^{2r+j-2},-q^{4r-3};-q^{4r-3})_\infty}
{(-q,-q^3,q^4;q^4)_\infty}\right).\label{h-def}
\end{align}

We then establish the following results.

\begin{thm}\label{vmf-G-AF}
For $r\geq 2$,  $\boldsymbol{G}_{4r-2}(\tau)$ is a vector-valued automorphic form
of some multiplier
for the group generated by $\left(1~~2\atop 0~~1\right)$
and $\left(1~~0\atop 4~~1\right)$.
In particular, for  odd integer $r\geq 3$,
$\boldsymbol{G}_{4r-2}(\tau)$ is a vector-valued modular function
of some multiplier with respect to
the congruence subgroup $\Gamma_0(2)$.
\end{thm}

\begin{thm}\label{vmf-H-AF}
For $r\geq 2$,  $\boldsymbol{H}_{4r-4}(\tau)$ is a vector-valued automorphic form
of some multiplier
for the group generated by $\left(1~~4\atop 0~~1\right)$
and $\left(1~~0\atop 8~~1\right)$.
\end{thm}

Several remarks are in order.

\begin{rem}
$(1)$~
In order to prove Theorem \ref{vmf-G-AF} and Theorem \ref{vmf-H-AF},
we establish modular transformation formulas on
``Langlands dual'' pair $\boldsymbol{g}_{2r-1}(\tau)$
and $\boldsymbol{g}^{\vee}_{2r-1}(\tau)$,
and $\boldsymbol{h}_{2r-2}(\tau)$
and $\boldsymbol{h}^{\vee}_{2r-2}(\tau)$ (see Theorem \ref{vmf-G} and Theorem \ref{vmf-H}), where
\begin{align*}
&\boldsymbol{g}_{2r-1}(\tau)=(g_1(\tau), g_2(\tau), \ldots, g_{2r-1}(\tau))^T,
\quad
\boldsymbol{g}_{2r-1}^{\vee}(\tau)=(g^{\vee}_1(\tau), g^{\vee}_2(\tau), \ldots, g^{\vee}_{2r-1}(\tau))^T,
\end{align*}
and
\begin{align*}
&\boldsymbol{h}_{2r-2}(\tau)=(h_1(\tau), h_2(\tau), \ldots, h_{2r-2}(\tau))^T,
\quad
\boldsymbol{h}_{2r-2}^{\vee}(\tau)=(h^{\vee}_1(\tau), h^{\vee}_2(\tau), \ldots, h^{\vee}_{2r-2}(\tau))^T.
\end{align*}
Notice that the case $r=3$ of Theorem \ref{vmf-G}
was conjectured by Mizuno \cite[Eq. (54)]{Mizuno-2025}), and subsequently confirmed by
B. Wang--L. Wang \cite{Wang-Wang-2025-2}.

$(2)$~
It follows from Theorem \ref{main1a} (1) and (2) that
\begin{align*}
g_1(\tau)&= \widetilde{f}_{A,\boldsymbol{b}_{0},{c}_{0},\boldsymbol{d}}(\tau),
\end{align*}
and, for \(r\leq j\leq2r-1\), we have
\begin{align*}
g_j(\tau)&= \widetilde{f}_{A,\boldsymbol{b}_{2r-j},{c}_{2r-j},\boldsymbol{d}}(\tau),
\end{align*}
where $(A,\boldsymbol{b}_j,{c}_j)$ is the  triple of index
$\boldsymbol{d}=({2,\ldots, 2}, 1)$ given in Theorem  \ref{main1}.

Moreover,
 by Theorem \ref{main1a} (5)  and (6), we have
\begin{align*}
g^{\vee}_{2r-1}(2\tau)=\frac{1}{2}\left(\widetilde{f}_{A^{\vee},2\boldsymbol{b}_{0},2c_{0},2\boldsymbol{d}^{\vee}}(\tau)+(-1)^{r-1}\widetilde{f}_{A^{\vee},2\boldsymbol{b}_{0},2c_{0},2\boldsymbol{d}^{\vee}}\big(\tau+\frac{1}{2}\big)\right),
\end{align*}
and, for \(2\leq j\leq 2r-2\) and \(j\) is even, we have
\[g^{\vee}_j(2\tau)=\frac{1}{2}\left(\widetilde{f}_{A^{\vee},2\boldsymbol{b}_{r-\frac{j}{2}},2c_{r-\frac{j}{2}},2\boldsymbol{d}^{\vee}}(\tau)+(-1)^\frac{j(j-1)}{2}\widetilde{f}_{A^{\vee},2\boldsymbol{b}_{r-\frac{j}{2}},2c_{r-\frac{j}{2}},2\boldsymbol{d}^{\vee}}\big(\tau+\frac{1}{2}\big)\right),\]
where \((A^{\vee},\boldsymbol{b}_i,c_i)\) is the triple of index
\(\boldsymbol{d}^{\vee}=(1,\ldots,1,2)\) given in Theorem \ref{main3}.

$(3)$~By Theorem \ref{main1a} (3) and (4), we obtain
\begin{align*}
h_1(\tau)&= \widetilde{f}_{{B},\boldsymbol{b}_{0},{c}_{0},\boldsymbol{d}}(\tau),
\end{align*}
and for $r-1\leq j\leq 2r-2$,
\begin{align*}
h_j(\tau)&= \widetilde{f}_{{B},\boldsymbol{b}_{2r-j-1},{c}_{2r-j-1},\boldsymbol{d}}(\tau),
\end{align*}
where $({B},\boldsymbol{b}_j,{c}_j)$ is the  triple of index $\boldsymbol{d}=({2,\ldots, 2}, 1)$
given in Theorem \ref{main2}.
\end{rem}

{ When $r\geq 3$ and $j=2$, we find  that
$g_2(\tau)$ can be expressed as the sum of
two  Nahm sums for the symmetrizable matrix $A$ of index $\boldsymbol{d}=({2,\ldots, 2}, 1)$, as shown in the following proposition.
\begin{prop} We have
\begin{align*}
g_2(\tau)&= \widetilde{f}_{A,\boldsymbol{b}_{r+1},{c}_{r+1},\boldsymbol{d}}(\tau)
+\widetilde{f}_{{A},\boldsymbol{\widetilde{b}}_{r+1},\widetilde{c}_{r+1},\boldsymbol{d}}(\tau),%\label{eq-g-j-2}
\end{align*}
 where
$A$ is the symmetrizable matrix with the symmetrizer \(D={\rm diag} ({2,\ldots, 2}, 1)_{r\times r}\)    given in Theorem \ref{main1},
\begin{align*}
&\boldsymbol{b}_{r+1}=(0,\ldots,0,1)^T_{1\times r}, \quad
c_{r+1}=\frac{37-4r}{32r-8},\\
&\boldsymbol{\widetilde{b}}_{r+1}=(1,2,4,\ldots, 2(r-2),r)^T_{1\times r},\quad
\widetilde{c}_{r+1}=\frac{(4r-7)^2}{32r-8}.
\end{align*}
\end{prop}
This relation  is implied by the following identity.
\begin{thm}\label{Nahm-SumC}
For $r\geq 3$,
\begin{align}
&\sum_{\boldsymbol{n}=(n_1,\ldots,n_r)^T\in\mathbb{N}^r}\frac{q^{\frac{1}{2}\boldsymbol{n}^T{A}D\boldsymbol{n}
+\boldsymbol{n}^T\boldsymbol{b}_{r+1}}
+q^{\frac{1}{2}\boldsymbol{n}^T{A}D\boldsymbol{n}
+\boldsymbol{n}^T\boldsymbol{\widetilde{b}}_{r+1}+\frac{r-3}{2}}}
{(q^2;q^2)_{n_1}(q^2;q^2)_{n_2}\cdots(q^2;q^2)_{n_{r-1}}(q;q)_{n_r}}\notag\\[5pt]
&=
\frac{(q^{8r+4},q^{8r-8},q^{16r-4};q^{16r-4})_\infty}
{(q, q^3, q^4 ;q^4)_\infty}
+q^{\frac{r-3}{2}}
\frac{(-q^3,q^{4r-4},-q^{4r-1}; -q^{4r-1})_\infty}{(q^2,q^2,q^4;q^4)_\infty},\label{SumC-eq}
\end{align}
where $AD$ is given by \eqref{eq-1.3},
$\boldsymbol{b}_{r+1}=(0,\ldots,0,1)^T_{1\times r}$, and
$\boldsymbol{\widetilde{b}}_{r+1}=(1,2,4,\ldots, 2(r-2),r)^T_{1\times r}$.
\end{thm}

Notice that the case $r=3$ of Theorem \ref{Nahm-SumC} has been
given by B. Wang--L. Wang \cite[Eq. (1.13)]{Wang-Wang-2025-2}.

}

To express $h^{\vee}_j(2\tau)$ as a combination of
generalized Nahm sums of index \((1,\ldots,1,2)\), we employ the notion of  partial Nahm sums introduced in \cite{Wang-Zeng-2025}.
For the   Nahm sum associated with symmetrizable matrix $A$:
\[\widetilde{f}_{A,\boldsymbol{b},{c},\boldsymbol{d}}(q):=\sum_{\boldsymbol{n}=(n_1,\ldots,n_r)^T\in \mathbb{N}^r}\frac{q^{\frac{1}{2}\boldsymbol{n}^TAD\boldsymbol{n}+\boldsymbol{n}^T\boldsymbol{b}+c}}{(q^{d_1};q^{d_1})_{n_1}\cdots(q^{d_r};q^{d_r})_{n_r}},\]
its partial Nahm sum is defined as
\[\widetilde{f}_{A,\boldsymbol{b},{c},\boldsymbol{d}, \sigma}(q):=\sum_{\boldsymbol{n}=(n_1,\ldots,n_r)\in\mathbb{N}^r \atop n_r\equiv \sigma \pmod{2}}\frac{q^{\frac{1}{2}\boldsymbol{n}^TAD\boldsymbol{n}+\boldsymbol{n}^T\boldsymbol{b}+c}}{(q^{d_1};q^{d_1})_{n_1}\cdots(q^{d_r};q^{d_r})_{n_r}},\]
where $\sigma=0,1$.

It turns out that $h^{\vee}_j(2\tau)$  can be expressed  as a combination of
the generalized Nahm sums $\widetilde{f}_{B^{\vee},\boldsymbol{b}_j,{c}_j,\boldsymbol{d}^{\vee}}(q)$ of index \(\boldsymbol{d}^{\vee}=(1,\ldots,1,2)\), where \((B^{\vee},\,\boldsymbol{b}_j,\,c_j)\) are given by
\begin{align}
&B^{\vee}=\begin{pmatrix}
            1 & 0 & 0 & 0 & \cdots & 0 & 0 & \frac{1}{2}\\
            0 & 2 & 2 & 2 & \cdots & 2 & 2 & 1\\
            0 & 2 & 4 & 4 & \cdots & 4 & 4 & 2\\
            0 & 2 & 4 & 6 & \cdots & 6 & 6 & 3\\
            \vdots & \vdots & \vdots & \vdots & \ddots & \vdots & \vdots & \vdots\\
            0 & 2 & 4 & 6 & \cdots & 2(r-3) & 2(r-3) & r-3\\
            0 & 2 & 4 & 6 & \cdots & 2(r-3) & 2(r-2) & r-2\\
            1 & 2 & 4 & 6 & \cdots & 2(r-3) & 2(r-2) & r-\frac{1}{2}
        \end{pmatrix}_{r\times r}, \label{matrix:Bvee}\\[10pt]
        &\boldsymbol{b}_0=(1,1,2,3,\ldots,r-1)^T_{1\times r},\quad c_0=\frac{4 r^2-11 r+7}{16 r-12},\\
        &\text{and for }1\leq j\leq r, \nonumber \\
        &\boldsymbol{b}_j=(0,\ldots,0,1,2,\ldots,r-j-1,r-j-1)^T_{1\times r},\quad c_j=\frac{4(r-j)^2+6j-7r+3}{16r-12}.
\end{align}
It is straightforward to verify that $B^{\vee}D^{\vee}$
as shown in \eqref{eq-1.6} is a symmetric positive definite matrix.
\begin{align}
        &B^{\vee}D^{\vee}=\begin{pmatrix}
            1 & 0 & 0 & 0 & \cdots & 0 & 0 & 1\\
            0 & 2 & 2 & 2 & \cdots & 2 & 2 & 2\\
            0 & 2 & 4 & 4 & \cdots & 4 & 4 & 4\\
            0 & 2 & 4 & 6 & \cdots & 6 & 6 & 6\\
            \vdots & \vdots & \vdots & \vdots & \ddots & \vdots & \vdots & \vdots\\
            0 & 2 & 4 & 6 & \cdots & 2(r-3) & 2(r-3) & 2(r-3)\\
            0 & 2 & 4 & 6 & \cdots & 2(r-3) & 2(r-2) & 2(r-2)\\
            1 & 2 & 4 & 6 & \cdots & 2(r-3) & 2(r-2) & 2r-1
        \end{pmatrix}_{r\times r}.\label{eq-1.6}
    \end{align}

Note that the symmetrizable matrix  $B^{\vee}$ given  in \eqref{matrix:Bvee}  is the transpose of  the symmetrizable matrix $B$ in Theorem \ref{main2}.

\begin{prop} We have
\begin{align*}
h^{\vee}_{2r-2}(2\tau)&=\frac{1}{2}\left(\widetilde{f}_{B^{\vee},2\boldsymbol{b}_{0},2c_{0},2\boldsymbol{d}^{\vee},0}(\tau)-\widetilde{f}_{B^{\vee},2\boldsymbol{b}_{0},2c_{0},2\boldsymbol{d}^{\vee},1}(\tau)\right.\\[5pt]
&\quad \left.+\widetilde{f}_{B^{\vee},2\boldsymbol{b}_{0},2c_{0},2\boldsymbol{d}^{\vee},0}\big(\tau+\frac{1}{2}\big)
-\widetilde{f}_{B^{\vee},2\boldsymbol{b}_{0},2c_{0},2\boldsymbol{d}^{\vee},1}\big(\tau+\frac{1}{2}\big)\right),
\end{align*}
and for \(1\leq j\leq 2r-3\) and \(j\) is odd, we have
\begin{align*}
h^{\vee}_j(2\tau)&=\frac{1}{2}\left(\widetilde{f}_{B^{\vee},2\boldsymbol{b}_{r-\frac{j+1}{2}},2c_{r-\frac{j+1}{2}},2\boldsymbol{d}^{\vee},0}(\tau)-
\widetilde{f}_{B^{\vee},2\boldsymbol{b}_{r-\frac{j+1}{2}},2c_{r-\frac{j+1}{2}},2\boldsymbol{d}^{\vee},1}(\tau)\right.\\[5pt]
&\quad \left.+\widetilde{f}_{B^{\vee},2\boldsymbol{b}_{r-\frac{j+1}{2}},2c_{r-\frac{j+1}{2}},2\boldsymbol{d}^{\vee},0}\big(\tau+\frac{1}{2}\big)
-\widetilde{f}_{B^{\vee},2\boldsymbol{b}_{r-\frac{j+1}{2}},2c_{r-\frac{j+1}{2}},2\boldsymbol{d}^{\vee},1}\big(\tau+\frac{1}{2}\big)\right).
\end{align*}
\end{prop}
This relation follows directly from the following result:

 \begin{thm}\label{main1b}
We have
\begin{align}
(1)&\sum_{\boldsymbol{n}=(n_1,\ldots,n_r)^T\in\mathbb{N}^r}
\frac{(-1)^{n_r}q^{\frac{1}{2}
\boldsymbol{n}^TB^{\vee}D^{\vee}\boldsymbol{n}+\boldsymbol{n}^T\boldsymbol{b}_0}}{(q;q)_{n_1}\cdots(q;q)_{n_{r-1}}(q^2;q^2)_{n_r}}=\frac{(q^{\frac{1}{2}},q^{2r-2},q^{2r-\frac{3}{2}};q^{2r-\frac{3}{2}})_\infty}{(q^{\frac{1}{2}},q^{\frac{3}{2}},q^2;q^2)_\infty},\label{eq-1.13}\\[8pt]
&\text{ where } B^{\vee}D^{\vee} \text{ is given by } \eqref{eq-1.6} \text{ and } \textbf{b}_0=(1,1,2,3,\ldots,r-1)^T_{1\times r},\notag\\[8pt]
    (2)&\text{for }1\leq j\leq r,\notag\\
        &\sum_{\boldsymbol{n}=(n_1,\ldots,n_r)^T\in\mathbb{N}^r}
\frac{(-1)^{n_r}q^{\frac{1}{2}
\boldsymbol{n}^TB^{\vee}D^{\vee}\boldsymbol{n}+\boldsymbol{n}^T\boldsymbol{b}_j}}{(q;q)_{n_1}\cdots(q;q)_{n_{r-1}}(q^2;q^2)_{n_r}}
=\frac{(q^{j},q^{2r-j-\frac{3}{2}},q^{2r-\frac{3}{2}};q^{2r-\frac{3}{2}})_\infty}{(q^{\frac{1}{2}},q^{\frac{3}{2}},q^2;q^2)_\infty},\label{eq-1.14}&\\[8pt]
&\text{where }B^{\vee}D^{\vee} \text{ is given by } \eqref{eq-1.6} \text{ and }\boldsymbol{b}_j=(0,\ldots,0,1,2,\ldots,r-j-1,r-j-1)^T_{1\times r}.\notag
    \end{align}
\end{thm}

It remains to address the following problem:
\begin{prob} \label{proba}
$(a)$ For $r\geq 4$ and
$3\leq j\leq r-1$,
can $g_j(\tau)$ be expressed as a combination of
 Nahm sums   for the symmetrizable matrix $A$ given in Theorem  \ref{main1}  of index
$({2,\ldots, 2}, 1)$?

$(b)$ For $r\geq 2$ and $1\leq j\leq 2r-3$ and $j$ is odd,
can $g^{\vee}_j(2\tau)$ be expressed as a combination of
  Nahm sums for the symmetrizable matrix $A^{\vee}$ given in Theorem  \ref{main3}  of index \((1,\ldots,1,2)\)?

$(c)$ For $r\geq 4$ and $2\leq j\leq r-2$,
can $h_j(\tau)$ be expressed as a combination of
  Nahm sums for the symmetrizable matrix $B$ given in Theorem  \ref{main2}
of index $({2,\ldots, 2}, 1)$?

$(d)$ For $r\geq 3$ and $2\leq j\leq 2r-4$ and $j$ is even,
can $h^{\vee}_j(2\tau)$ be expressed as a combination of   Nahm sums for the symmetrizable matrix $B^{\vee}$ given in   \eqref{matrix:Bvee}
of index \((1,\ldots,1,2)\)?
\end{prob}

The paper is organized as follows. In Section 2, we review some relevant known results on the Bailey machinery, and then utilize this machinery to provide proofs of
three
families of multi-sum  Rogers--Ramanujan type identities (see {Theorem }\ref{thm:multisuma}, {Theorem }\ref{thm:multisumb}
and Theorem \ref{thm-SumC}). Section 3 is devoted to proving {Theorem}  \ref{main1a}, Theorem \ref{Nahm-SumC} and Theorem \ref{main1b},
using the results of Theorems \ref{thm:multisuma}, \ref{thm:multisumb} and \ref{thm-SumC}
in  Section 2.
In Section 4, we show that the generalized Nahm sums
specified in {Theorems}  \ref{main1}, \ref{main2}, \ref{main3}
are modular functions. Section 5 is devoted to
proving Theorem \ref{vmf-G-AF} and Theorem \ref{vmf-H-AF}
by establishing the modular transformation formulas
on two ``Langlands dual'' pairs.

\section{Multi-sum Rogers--Ramanujan type identities}

In this section,
we aim to present the following
{  four} families of multi-sum Rogers--Ramanujan type identities
which serve as key ingredients in the proof of  {Theorem} \ref{main1a}, Theorem \ref{Nahm-SumC} and Theorem \ref{main1b}.

\begin{thm} \label{thm:multisuma}
For $r\geq 2$,

\begin{enumerate}
    \item[$(a)$]
\begin{align}
    &\sum_{N_1\geq\cdots\geq N_{r-1}\geq0}\frac{q^{N_1^2+\cdots+N_{r-1}^2}}{(q;q)_{N_1-N_2}\cdots(q;q)_{N_{r-2}-N_{r-1}}(q;q)_{N_{r-1}}(q;q^2)_{N_{r-1}}}\nonumber \\[5pt]
        &=\frac{(q^{4r-2},q^{4r},q^{8r-2};q^{8r-2})_\infty}{(q;q)_\infty}, \label{Nahmsumsa}
\end{align}

\item[$(b)$]
for  $1\leq j\leq r$,
\begin{align}
&\sum_{N_1\geq\cdots\geq N_{r-1}\geq0}\frac{q^{N_1^2+\cdots+N_{r-1}^2+N_1+\cdots+N_{j-1}+2N_j+\cdots+2N_{r-1}}}{(q;q)_{N_1-N_2}\cdots(q;q)_{N_{r-2}-N_{r-1}}(q;q)_{N_{r-1}}(q^3;q^2)_{N_{r-1}}}\nonumber \\[5pt]
&= \frac{(q^{2j},q^{8r-2j-2},q^{8r-2};q^{8r-2})_\infty}{(q^2;q)_\infty}, \label{Nahmsumsd}
\end{align}

\item[$(c)$]
\begin{align}
&\sum_{N_1\geq\cdots\geq N_{r-1}\geq0}\frac{q^{N_1^2+\cdots+N_{r-1}^2+{N_{r-1}\choose 2}}}{(q;q)_{N_1-N_2}\cdots(q;q)_{N_{r-2}-N_{r-1}}(q;q)_{N_{r-1}}(q;q^2)_{N_{r-1}}}\nonumber \\[5pt]
&=\frac{(q^{4r-4},q^{4r-2},q^{8r-6};q^{8r-6})_\infty}{(q;q)_\infty},\label{Nahmsumsb}
\end{align}

\item[$(d)$]
for  $1\leq j\leq r$,
\begin{align}
&\sum_{N_1\geq\cdots\geq N_{r-1}\geq0}\frac{q^{N_1^2+\cdots+N_{r-1}^2+N_1+\cdots+N_{j-1}+2N_j+\cdots+2N_{r-1}+{N_{r-1}\choose 2}}}{(q;q)_{N_1-N_2}\cdots(q;q)_{N_{r-2}-N_{r-1}}(q;q)_{N_{r-1}}(q^3;q^2)_{N_{r-1}}}\nonumber \\[5pt]
&= \frac{(q^{2j},q^{8r-2j-6},q^{8r-6};q^{8r-6})_\infty}{(q^2;q)_\infty}. \label{Nahmsumse}
\end{align}
\end{enumerate}
\end{thm}

\begin{thm}\label{thm:multisumb}
For  $r\geq 2$,

\begin{enumerate}
\item[$(a)$]
\begin{align}
&\sum_{N_1\geq\cdots\geq N_{r-1}\geq0}\frac{(-1)^{N_{r-1}}q^{2(N_1^2+\cdots+N_{r-1}^2+N_1+\cdots+N_{r-1})+N_{r-1}^2}}{(q^2;q^2)_{N_1-N_2}\cdots(q^2;q^2)_{N_{r-2}-N_{r-1}}(q^4;q^4)_{N_{r-1}}(-q^3;q^2)_{N_{r-1}}}\nonumber \\[10pt]
&=\frac{(q,q^{4r-4},q^{4r-3};q^{4r-3})_\infty}{(1-q)(q^4;q^2)_\infty},\label{Nahmsumsh}
\end{align}

\item[$(b)$]
for  $1\leq j\leq r$,
\begin{align}
&\sum_{N_1\geq\cdots\geq N_{r-1}\geq0}\frac{(-1)^{N_{r-1}}q^{2(N_1^2+\cdots+N_{r-1}^2+N_j+\cdots+N_{r-1})+N^2_{r-1}-2N_{r-1}}}{(q^2;q^2)_{N_1-N_2}\cdots(q^2;q^2)_{N_{r-2}-N_{r-1}}(q^4;q^4)_{N_{r-1}}(-q;q^2)_{N_{r-1}}}\notag\\[5pt]
&=\frac{(q^{2j},q^{4r-2j-3},q^{4r-3};q^{4r-3})_\infty}{(q^2;q^2)_\infty}. \label{Nahmsumsj}
\end{align}
\end{enumerate}
\end{thm}

{
\begin{thm}\label{thm-SumC}
For  $r\geq 3$,
\begin{align}
&(a) \quad \sum_{N_1\geq N_2\geq \cdots \geq N_{r-1}\geq 0}
\frac{q^{N_1^2+N_2^2+N_3^2+\cdots +N_{r-1}^2+N_{r-1}}(1+q^{-1}-q^{N_{r-1}-1})}{(q;q)_{N_1-N_2}\cdots (q;q)_{N_{r-2}-N_{r-1}}(q;q)_{N_{r-1}}(q;q^2)_{N_{r-1}}}\notag\\[5pt]
&\quad =\frac{(q^{4r+2},q^{4r-4},q^{8r-2};q^{8r-2})_\infty}
{(q;q)_\infty},\label{SumC-1}\\[5pt]
&(b) \sum_{N_1\geq N_2\geq \cdots \geq N_{r-1}\geq 0}
\frac{q^{N_1^2+N_2^2+N_3^2+\cdots +N_{r-1}^2+N_1+N_2+\cdots N_{r-2}+2N_{r-1}}(1+q+q^{N_{r-1}+\frac{1}{2}})}
{(q;q)_{n-N_2}(q;q)_{N_2-N_3}\cdots (q;q)_{N_{r-2}-N_{r-1}}(-q^{1/2};q)_{N_{r-1}+1}(q^2;q^2)_{N_{r-1}}}\notag\\[5pt]
&\quad =\frac{(q^{\frac{3}{2}}, q^{2r-2}, q^{2r-\frac{1}{2}};q^{2r-\frac{1}{2}})_\infty}{(q;q)_\infty}.\label{SumC-2}
\end{align}
\end{thm}
}

The proofs of Theorems \ref{thm:multisuma}, \ref{thm:multisumb}, \ref{thm-SumC} rely on  the Bailey machinery. To this end,  we review some necessary background on the Bailey machinery in Subsection \ref{Bailey}. We proceed to lay out the derivations of each
identity in Theorem \ref{thm:multisuma}, Theorem \ref{thm:multisumb} and Thereom \ref{thm-SumC} in Subsection \ref{subsec2.2}, Subsection \ref{subsec2.3}
and Subsection \ref{subsec-thmSC-pr}, respectively.

\subsection{Preliminaries} \label{Bailey} We first recall the definition of Bailey pair.
A pair of sequences \((\alpha_n(a;q),\beta_n(a;q))\) is called a Bailey pair relative to \(a\) if for all \(n\geq0,\)
\begin{align}
    \beta_n(a;q)=\sum_{r=0}^n\frac{\alpha_r(a;q)}{(q;q)_{n-r}(aq;q)_{n+r}}.\label{eq-2.9}
\end{align}

Note that taking the limit as $n\rightarrow \infty$ yields the following expression, subject to appropriate convergence conditions:
\begin{equation}
\lim_{n\rightarrow \infty} \beta_{n}(a;q)=\frac{1}{(q;q)_\infty(aq;q)_\infty} \sum_{r=0}^\infty \alpha_{r}(a;q).\label{Bailey-limit}
\end{equation}
In practice, the above sum  can often be rewritten as an infinite product via an application of the Jacobi triple product  identity \cite[Theorem 2.8]{Andrews-1976},
    \begin{align}
    (q^2,zq,q/z;q^2)_\infty&=\sum_{n=-\infty}^\infty(-1)^nq^{n^2}z^n.\label{eq-2.18}
    \end{align}

Once we have a Bailey pair, the following Bailey's lemma provides  a method to generate a new Bailey pair.

\begin{lem}[Bailey's Lemma] \label{BaileyLemm}Suppose that \((\alpha_n(a;q),\beta_n(a;q))\) is a Bailey pair relative to \(a.\) Then \((\alpha_n'(a;q),\beta_n'(a;q))\) is also a Bailey pair relative to \(a,\) where
\begin{equation}\label{eq-2.10}
\begin{split}
&\alpha_n'(a;q)=\frac{(\rho_1,\rho_2;q)_n(aq/\rho_1\rho_2)^n}{(aq/\rho_1,aq/\rho_2;q)_n}\alpha_n(a;q),\\
    &\beta_n'(a;q)=\sum_{r=0}^n\frac{(\rho_1,\rho_2;q)_r(aq/\rho_1\rho_2;q)_{n-r}(aq/\rho_1\rho_2)^r}{(aq/\rho_1,aq/\rho_2;q)_n(q;q)_{n-r}}\beta_r(a;q).
\end{split}
\end{equation}
\end{lem}

Upon taking \(\rho_1 \rightarrow \infty, \rho_2 \rightarrow \infty,\) we obtain

\begin{lem}{\!\!\rm \cite[Eq. (S1)]{Bressoud-2000}} \label{lemma-S1}
    If \((\alpha_n(a;q),\beta_n(a;q))\) is a Bailey pair relative to \(a\), then \((\alpha_n'(a;q),\beta_n'(a;q))\) is also a Bailey pair relative to \(a\), where
    \begin{align}        \alpha_n'(a;q)=a^nq^{n^2}\alpha_n(a;q),\quad\beta_n'(a;q)=\sum_{r=0}^n\frac{a^rq^{r^2}}{(q;q)_{n-r}}\beta_r(a;q). \label{eq-S1}
    \end{align}
\end{lem}

%Applying \eqref{eq-S1} into \eqref{Bailey-limit},
%Letting $n, \rho_1, \rho_2\rightarrow\infty$ in \eqref{eq-2.10},
%we have

%\begin{lem}\!\!{\rm \cite[Eq. (1.2.8)]{McLaughlin-2008}}\label{lem-B-2}
%If $(\alpha_n(a;q), \beta(a;q))$ is a Bailey pair relative to $a$,
%then
%\begin{align*}
%\sum_{n=0}^\infty a^nq^{n^2}\beta_n(a;q)
%=\frac{1}{(aq;q)_\infty}\sum_{n=0}^\infty
%a^{n}q^{n^2}\alpha_n(a;q).
%\end{align*}
%\end{lem}}

The following Lemma  provides a method for generating a Bailey pair related to \(a/q\) from a Bailey pair relative to \(a.\)

\begin{lem}{\!\rm \cite[Lemma 3.1]{Lovejoy-2022}}\label{lem-1}
If \((\alpha_n(a;q),\beta_n(a;q))\) is a Bailey pair relative to \(a\),  then \((\alpha_n'(a/q;q),\beta_n'(a/q;q))\) is  a Bailey pair relative to \(a/q\), where
\begin{equation}\label{eq-2.13}
\begin{split}
&\alpha_n'(a/q;q)=(1-a)\left(\frac{q^n\alpha_n(a;q)}{1-aq^{2n}}-\frac{q^{n-1}\alpha_{n-1}(a;q)}{1-aq^{2n-2}}\right),\\
&\beta_n'(a/q;q)=q^n\beta_n(a;q).
\end{split}
\end{equation}
\end{lem}
% Conversely, the following lemma give a way to generate a Bailey pair related to \(aq\) from a Bailey pair related to \(a.\)
% \begin{lem}\cite[Lemma 2.4]{Lovejoy-2022}
%     If \((\alpha_n(a;q),\beta_n(a;q))\) is a Bailey pair related to \(a,\) then
%     \begin{align}
%         \alpha_n'(aq;q)=&\frac{(1-aq^{2n+1})(aq/b;q)_n(-b)^nq^{n(n-1)/2}}{(1-aq)(bq;q)_n}\notag\\
%         &\quad\times\sum_{r=0}^n\frac{(b;q)_r}{(aq/b;q)_r}(-b)^{-r}q^{-r(r-1)/2}\alpha_r(a;q),\notag\\
%         \beta_n'(aq;q)=&\frac{(b;q)_n}{(bq;q)_n}\beta_n(a;q)\label{eq-2.14}
%     \end{align}
%     is a Bailey pair related to \(aq.\)
% \end{lem}
% When \(b\rightarrow\infty,\) \eqref{eq-2.14} becomes
% \begin{lem}\cite[Eq. (2.14)]{Wang-Wang-2025-2}\label{lem-2}
% If \((\alpha_n(a;q),\beta_n(a;q))\) is a Bailey pair related to \(a,\) then
% \begin{align}
%         &\alpha_n'(aq;q)=\frac{1-aq^{2n+1}}{1-aq}q^{-n}\sum_{r=0}^n\alpha_r(a;q),\notag\\
%         &\beta_n'(aq;q)=q^{-n}\beta_n(a;q)\label{eq-2.15}
%     \end{align}
%     is a Bailey pair related to \(aq.\)
%     \end{lem}
%     When \(b=0,\) \eqref{eq-2.14} becomes
% \begin{lem}\cite[Eq. (2.15]{Wang-Wang-2025-2}\label{lem-3}
%     If \((\alpha_n(a;q),\beta_n(a;q))\) is a Bailey pair related to \(a\), then
%     \begin{align}
%         &\alpha_n'(aq;q)=\frac{(1-aq^{2n+1})a^nq^{n^2}}{1-aq}\sum_{r=0}^na^{-r}q^{-r^2}\alpha_r(a;q),\notag\\
%         &\beta_n'(aq;q)=\beta_n(a;q).\label{eq-2.16}
%     \end{align}
% \end{lem}

\subsection{Proof of Theorem \ref{thm:multisuma} } \label{subsec2.2}

We begin by recalling the following Bailey pairs from Group    C of Slater's list \cite[p. 469]{Slater-1951}, which are required in the proof of  Theorem \ref{thm:multisuma}.

In what follows, we   set  \(\alpha_0(a;q)=1\) and   adopt the convention that \(\alpha_n(a;q)=0\) for \(n<0\) unless specified otherwise.

\begin{align*}
{\rm (C1)} \quad &\alpha_{2n}(1;q)=(-1)^nq^{3n^2}(q^n+q^{-n}),\quad\alpha_{2n+1}(1;q)=0,\\[5pt]
    &\beta_n(1;q)=\frac{1}{(q;q)_n(q;q^2)_n};\\[10pt]
 {\rm (C3)} \quad  &\alpha_{2n}(q;q)=(-1)^nq^{3n^2+n},\quad\alpha_{2n+1}(q;q)=(-1)^{n+1}q^{3n^2+5n+2},\\[5pt]
    &\beta_n(q;q)=\frac{1}{(q;q)_n(q^3;q^2)_n};\\[10pt]
    {\rm (C4)}\quad &\alpha_{2n}(q;q)=(-1)^nq^{3n^2+3n},\quad\alpha_{2n+1}(q;q)=(-1)^{n+1}q^{3n^2+3n},\\[5pt]
    &\beta_n(q;q)=\frac{q^n}{(q;q)_n(q^3;q^2)_n};\\[10pt]
     {\rm (C4^*)}\quad&\alpha_{2n}(q^2;q)=(-1)^nq^{3n^2+n}\frac{1-q^{4n+2}}{1-q^2},\quad\alpha_{2n+1}(q^2;q)=0,\\[5pt]
    &\beta_n(q^2;q)=\frac{1}{(q;q)_n(q^3;q^2)_n};\\[10pt]
    {\rm (C5)}\quad&\alpha_{2n}(1;q)=(-1)^nq^{n^2}(q^n+q^{-n}),\quad\alpha_{2n+1}(1;q)=0,\\[5pt]
    &\beta_n(1;q)=\frac{q^{\frac{1}{2}(n^2-n)}}{(q;q)_n(q;q^2)_n};\\[10pt]
       {\rm (C6)}\quad&\alpha_{2n}(q;q)=(-1)^nq^{n^2-n},\quad\alpha_{2n+1}(q;q)=(-1)^{n+1}q^{n^2+3n+2},\\[5pt]
    &\beta_n(q;q)=\frac{q^{\frac{1}{2}(n^2-n)}}{(q;q)_n(q^3;q^2)_n};\\[10pt]
    {\rm (C7)}\quad&\alpha_{2n}(q;q)=(-1)^nq^{n^2+n},\quad\alpha_{2n+1}(q;q)=(-1)^{n+1}q^{n^2+n},\\[5pt]
    &\beta_n(q;q)=\frac{q^{\frac{1}{2}(n^2+n)}}{(q;q)_n(q^3;q^2)_n};\\[10pt]
       {\rm (C7^*)}\quad&\alpha_{2n}'(q^2;q)=(-1)^nq^{n^2-n}\frac{1-q^{4n+2}}{1-q^2},\quad\alpha_{2n+1}'(q^2;q)=0,\notag\\[5pt]
    &\beta_n'(q^2;q)=\frac{q^{\frac{1}{2}(n^2-n)}}{(q;q)_n(q^3;q^2)_n}.
\end{align*}
 Note that ${\rm (C4^*)}$ and ${\rm (C7^*)}$ were first stated  by B. Wang--L. Wang  \cite[p. 13, p. 14]{Wang-Wang-2025-2}, who derived them by applying  base-change method  due to Lovejoy \cite[Lemma 2.4]{Lovejoy-2014} to  the Bailey pairs (C4) and (C7), respectively.

We now give a proof of  Theorem \ref{thm:multisuma}. We prove each of the four identities in the theorem one by one.

\noindent  {\bf  $(a)$ Proof of  \eqref{Nahmsumsa}.}  Iterate Lemma \ref{lemma-S1} (with $a=1$)  \(r-1\) times on the Bailey pair (C1) to get a new Bailey pair $(\alpha^{(r-1)}_{n}(1;q), \beta^{(r-1)}_{n}(1;q))$ relative to $1$, where
    \begin{align*}
        &\alpha^{(r-1)}_{2n}(1;q)=(-1)^nq^{(4r-1)n^2}(q^n+q^{-n}),\quad\alpha_{2n+1}^{(r-1)}(1;q)=0,\notag\\[5pt]
        &\beta_n^{(r-1)}(1;q)=\sum_{n\geq N_1\geq\cdots\geq N_{r-1}\geq0}\frac{q^{N_1^2+\cdots+N_{r-1}^2}}{(q;q)_{n-N_1}(q;q)_{N_1-N_2}\cdots(q;q)_{N_{r-2}-N_{r-1}}(q;q)_{N_{r-1}}(q;q^2)_{N_{r-1}}}.
    \end{align*}
It follows from \eqref{Bailey-limit} that
    \begin{align*}
         &\sum_{N_1\geq\cdots\geq N_{r-1}\geq0}\frac{q^{N_1^2+\cdots+N_{r-1}^2}}{(q;q)_{N_1-N_2}\cdots(q;q)_{N_{r-2}-N_{r-1}}(q;q)_{N_{r-1}}(q;q^2)_{N_{r-1}}}\\[5pt]
        &=\frac{1}{(q;q)_\infty}\left(1+\sum_{n=1}^\infty(-1)^nq^{(4r-1)n^2}(q^n+q^{-n})\right)\notag\\[5pt]
        &=\frac{1}{(q;q)_\infty}\sum_{n=-\infty}^\infty(-1)^nq^{(4r-1)n^2+n}\notag \\[5pt]
        &\overset{\eqref{eq-2.18}}{=}\frac{(q^{4r},q^{4r-2},q^{8r-2};q^{8r-2})_\infty}{(q;q)_\infty}.
    \end{align*}

\noindent  {\bf $(b)$ Proof of  \eqref{Nahmsumsd}.}  We consider the
following two cases:

Case 1: If $j=r$, then applying Lemma \ref{lemma-S1} (with $a=q$) \(r-1\) times to   the Bailey pair (C3), we obtain a new Bailey pair:
    \begin{align*}
        &\alpha_{2n}^{(r-1)}(q;q)=(-1)^nq^{(4r-1)n^2+(2r-1)n},\quad\alpha_{2n+1}^{(r-1)}(q;q)=(-1)^{n+1}q^{(4r-1)n^2+(6r-1)n+2r},\\[5pt]
        &\beta_n^{(r-1)}(q;q)=\sum_{n\geq N_1\geq\cdots\geq N_{r-1}\geq0}\frac{q^{N_1^2+\cdots+N_{r-1}^2+N_1+\cdots+N_{r-1}}}{(q;q)_{n-N_1}(q;q)_{N_1-N_2}\cdots(q;q)_{N_{r-2}-N_{r-1}}(q;q)_{N_{r-1}}(q^3;q^2)_{N_{r-1}}}.
    \end{align*}
By \eqref{Bailey-limit},  we derive
    \begin{align*}
         &\sum_{N_1\geq\cdots\geq N_{r-1}\geq0}\frac{q^{N_1^2+\cdots+N_{r-1}^2+N_1+\cdots+N_{r-1}}}{(q;q)_{N_1-N_2}\cdots(q;q)_{N_{r-2}-N_{r-1}}(q;q)_{N_{r-1}}(q^3;q^2)_{N_{r-1}}}\\[5pt]
        &=\frac{1}{(q^2;q)_\infty}\left(\sum_{n=0}^\infty(-1)^nq^{(4r-1)n^2+(2r-1)n}+ \sum_{n=0}^\infty (-1)^{n+1}q^{(4r-1)n^2+(6r-1)n+2r} \right)\notag\\[5pt]
        &=\frac{1}{(q^2;q)_\infty}\sum_{n=-\infty}^\infty(-1)^nq^{(4r-1)n^2+(2r-1)n}\notag \\[5pt]
        &\overset{\eqref{eq-2.18}}{=}\frac{(q^{6r-2},q^{2r},q^{8r-2};q^{8r-2})_\infty}{(q^2;q)_\infty}.
    \end{align*}

Case 2: If $1\leq j <r$, then  iterating Lemma \ref{lemma-S1} (with $a=q^2$) \(r-1-j \) times on the  Bailey pair (C4\(^*\)) yields a new Bailey pair:
\begin{align}
    \alpha_{2n}^{(r-1-j )}(q^2;q)&=(-1)^nq^{(4r-4j -1)n^2+(4r-4j -3)n}\frac{1-q^{4n+2}}{1-q^2},\quad\alpha_{2n+1}^{(r-j -1)}(q^2;q)=0,\notag\\[5pt]
    \beta_n^{(r-1-j )}(q^2;q)&=\sum_{n\geq N_{j+1}\geq\cdots\geq N_{r-1}\geq0}\frac{q^{N_{j+1}^2+\cdots+N_{r-1}^2+2N_{j+1}+\cdots+2N_{r-1}}}{(q;q)_{n-N_{j+1}}(q;q)_{N_{j+1}-N_{j+2}}\cdots(q;q)_{N_{r-2}-N_{r-1}}}\notag\\
    &\hskip 2cm\times\frac{1}{(q;q)_{N_{r-1}}(q^3;q^2)_{N_{r-1}}}.\label{eq-2.19}
\end{align}
Applying Lemma \ref{lem-1} (with $a=q^2$) to the Bailey pair \eqref{eq-2.19}, we get a new Bailey pair relative to $q$:
\begin{align}
    \alpha^{(r-j )}_{2n}(q;q)&=(-1)^nq^{(4r-4j -1)n^2+(4r-4j -1)n},\alpha^{(r-j )}_{2n+1}(q;q)=(-1)^{n+1}q^{(4r-4j -1)n^2+(4r-4j -1)n},\notag\\[5pt]
    \beta^{(r-j )}_n(q;q)&=\sum_{n\geq N_{j+1}\geq\cdots\geq N_{r-1}\geq0}\frac{q^{N_{j+1}^2+\cdots+N_{r-1}^2+n+2N_{j+1}+\cdots+2N_{r-1}}}{(q;q)_{n-N_{j+1}}(q;q)_{N_{j+1}-N_{i+2}}\cdots(q;q)_{N_{r-2}-N_{r-1}}}\notag\\
     &\hskip 2cm \times\frac{1}{(q;q)_{N_{r-1}}(q^3;q^2)_{N_{r-1}}}.\label{eq-2.20}
\end{align}
Iterate Lemma \ref{lemma-S1} (with $a=q$)  \(j\) times on \eqref{eq-2.20} to yield
\begin{align*}
    \alpha_{2n}^{(r)}(q;q)&=(-1)^nq^{(4r-1)n^2+(4r-2j-1)n},~\alpha_{2n+1}^{(r)}(q;q)=(-1)^{n+1}q^{(4r-1)n^2+(4r+2j-1)n+2j},\notag\\[5pt]
    \beta_n^{(r)}(q;q)&=\sum_{n\geq N_1\geq\cdots\geq N_{r-1}\geq0}\frac{q^{N_1^2+\cdots+N_{r-1}^2+N_1+\cdots+N_{j-1}+2N_j+\cdots+2N_{r-1}}}{(q;q)_{n-N_1}(q;q)_{N_1-N_2}\cdots(q;q)_{N_{r-2}-N_{r-1}}}\notag\\[5pt]
    &\hskip 2cm \times\frac{1}{(q;q)_{N_{r-1}}(q^3;q^2)_{N_{r-1}}}.
\end{align*}
By means of \eqref{Bailey-limit}, we get
\begin{align*}
& \sum_{N_1\geq\cdots\geq N_{r-1}\geq0}\frac{q^{N_1^2+\cdots+N_{r-1}^2+N_1+\cdots+N_{j-1}+2N_j+\cdots+2N_{r-1}}}{(q;q)_{N_1-N_2}(q;q)_{N_2-N_3}\cdots(q;q)_{N_{r-2}-N_{r-1}}(q;q)_{N_{r-1}}(q^3;q^2)_{N_{r-1}}}\nonumber \\[5pt]
    &=\frac{1}{(q^2;q)_\infty}\left(\sum_{n=0}^\infty(-1)^nq^{(4r-1)n^2+(4r-2j-1)n}+\sum_{n=0}^\infty  (-1)^{n+1}q^{(4r-1)n^2+(4r+2j-1)n+2j} \right)\notag\\[5pt]
    &=\frac{1}{(q^2;q)_\infty}\times \sum_{n=-\infty}^\infty(-1)^nq^{(4r-1)n^2+(4r-2j-1)n}\notag\\[5pt]
    &\overset{\eqref{eq-2.18}}{=}\frac{(q^{2j},q^{8r-2j-2},q^{8r-2};q^{8r-2})_\infty}{(q^2;q)_\infty}.
\end{align*}

\noindent  {\bf $(c)$ Proof of  \eqref{Nahmsumsb}.}
Iterating Lemma \ref{lemma-S1} (with $a=1$) \(r-1\) times on the Bailey pair (C5) yields a new Bailey pair $(\alpha^{(r-1)}_{n}(1;q), \beta^{(r-1)}_{n}(1;q))$ relative to $1$, where
      \begin{align*}
        &\alpha^{(r-1)}_{2n}(1;q)=(-1)^nq^{(4r-3)n^2}(q^n+q^{-n}),\quad\alpha_{2n+1}^{(r-1)}(1;q)=0,\notag\\[5pt]
        &\beta_n^{(r-1)}(1;q)=\sum_{n\geq N_1\geq\cdots\geq N_{r-1}\geq0}\frac{q^{N_1^2+\cdots+N_{r-1}^2+{N_{r-1}\choose 2}}}{(q;q)_{n-N_1}(q;q)_{N_1-N_2}\cdots(q;q)_{N_{r-2}-N_{r-1}}(q;q)_{N_{r-1}}(q;q^2)_{N_{r-1}}}.
    \end{align*}
In view of \eqref{Bailey-limit},  we get
    \begin{align*}
         &\sum_{N_1\geq\cdots\geq N_{r-1}\geq0}\frac{q^{N_1^2+\cdots+N_{r-1}^2+{N_{r-1}\choose 2}}}{(q;q)_{N_1-N_2}\cdots(q;q)_{N_{r-2}-N_{r-1}}(q;q)_{N_{r-1}}(q;q^2)_{N_{r-1}}}\\[5pt]
        &=\frac{1}{(q;q)_\infty}\left(1+\sum_{n=1}^\infty(-1)^nq^{(4r-3)n^2}(q^n+q^{-n})\right)\notag\\[5pt]
        &=\frac{1}{(q;q)_\infty}\sum_{n=-\infty}^\infty(-1)^nq^{(4r-3)n^2+n}\notag \\[5pt]
        &\overset{\eqref{eq-2.18}}{=}\frac{(q^{4r-4},q^{4r-2},q^{8r-6};q^{8r-6})_\infty}{(q;q)_\infty}.
    \end{align*}

\noindent  {\bf  $(d)$ Proof of  \eqref{Nahmsumse}.}
We consider the following two cases:

Case 1: If $j=r$, then iterating Lemma \ref{lemma-S1} (with $a=q$) \(r-1\) times on the Bailey pair (C6)  yields
    \begin{align*}
        &\alpha_{2n}^{(r-1)}(q;q)=(-1)^nq^{(4r-3)n^2+(2r-3)n},\quad\alpha_{2n+1}^{(r-1)}(q;q)=(-1)^{n+1}q^{(4r-3)n^2+(6r-3)n+2r},\\[5pt]
        &\beta_n^{(r-1)}(q;q)=\sum_{n\geq N_1\geq\cdots\geq N_{r-1}\geq0}\frac{q^{N_1^2+\cdots+N_{r-1}^2+N_1+\cdots+N_{r-1}+{N_{r-1}\choose 2}}}{(q;q)_{n-N_1}(q;q)_{N_1-N_2}\cdots(q;q)_{N_{r-2}-N_{r-1}}(q;q)_{N_{r-1}}(q^3;q^2)_{N_{r-1}}}.
    \end{align*}
From \eqref{Bailey-limit}, we obtain
    \begin{align*}
         &\sum_{N_1\geq\cdots\geq N_{r-1}\geq0}\frac{q^{N_1^2+\cdots+N_{r-1}^2+N_1+\cdots+N_{r-1}+{N_{r-1}\choose 2}}}{(q;q)_{N_1-N_2}\cdots(q;q)_{N_{r-2}-N_{r-1}}(q;q)_{N_{r-1}}(q^3;q^2)_{N_{r-1}}}\\[5pt]
        &=\frac{1}{(q^2;q)_\infty}\left(\sum_{n=0}^\infty(-1)^nq^{(4r-3)n^2+(2r-3)n}+ \sum_{n=0}^\infty (-1)^{n+1}q^{(4r-3)n^2+(6r-3)n+2r} \right)\notag\\[5pt]
        &=\frac{1}{(q^2;q)_\infty}\sum_{n=-\infty}^\infty(-1)^nq^{(4r-3)n^2+(2r-3)n}\notag \\[5pt]
        &\overset{\eqref{eq-2.18}}{=}\frac{(q^{6r-6},q^{2r},q^{8r-6};q^{8r-6})_\infty}{(q^2;q)_\infty}.
    \end{align*}

Case 2: If $1\leq j <r$, then  applying  Lemma \ref{lemma-S1} (with $a=q^2$) \(r-1-j \) times to the  Bailey pair (C7\(^*\)) yields the following Bailey pair:
\begin{align}
    \alpha_{2n}^{(r-1-j )}(q^2;q)&=(-1)^nq^{(4r-4j -3)n^2+(4r-4j -5)n}\frac{1-q^{4n+2}}{1-q^2},\quad\alpha_{2n+1}^{(r-j -1)}(q^2;q)=0,\notag\\[5pt]
    \beta_n^{(r-1-j )}(q^2;q)&=\sum_{n\geq N_{j+1}\geq\cdots\geq N_{r-1}\geq0}\frac{q^{N_{j+1}^2+\cdots+N_{r-1}^2+2N_{j+1}+\cdots+2N_{r-1}+{N_{r-1}\choose 2}}}{(q;q)_{n-N_{j+1}}(q;q)_{N_{j+1}-N_{j+2}}\cdots(q;q)_{N_{r-2}-N_{r-1}}}\notag\\
    &\hskip 2cm \times\frac{1}{(q;q)_{N_{r-1}}(q^3;q^2)_{N_{r-1}}}.\label{eq-2.30db}
\end{align}
Employing Lemma \ref{lem-1} (with $a=q^2$) to the Bailey pair \eqref{eq-2.30db}, we get
\begin{align}
    \alpha^{(r-j )}_{2n}(q;q)&=(-1)^nq^{(4r-4j -3)n^2+(4r-4j -3)n},\quad\alpha^{(r-j )}_{2n+1}(q;q)=(-1)^{n+1}q^{(4r-4j -3)n^2+(4r-4j -3)n},\notag\\[5pt]
    \beta^{(r-j )}_n(q;q)&=\sum_{n\geq N_{j+1}\geq\cdots\geq N_{r-1}\geq0}\frac{q^{n+N_{j+1}^2+\cdots+N_{r-1}^2+2N_{j+1}+\cdots+2N_{r-1}+{N_{r-1}\choose 2}}}{(q;q)_{n-N_{j+1}}(q;q)_{N_{j+1}-N_{i+2}}\cdots(q;q)_{N_{r-2}-N_{r-1}}}\notag\\
    &\hskip 2cm \times\frac{1}{(q;q)_{N_{r-1}}(q^3;q^2)_{N_{r-1}}}.\label{eq-2.22}
\end{align}
Iterate  Lemma \ref{lemma-S1} (with $a=q$)  \(j\) times on \eqref{eq-2.22} to generate the following Bailey pair:
\begin{align*}
    \alpha_{2n}^{(r)}(q;q)&=(-1)^nq^{(4r-3)n^2+(4r-2j-3)n},\quad\alpha_{2n+1}^{(r)}(q;q)=(-1)^{n+1}q^{(4r-3)n^2+(4r+2j-3)n+2j},\notag\\[5pt]
    \beta_n^{(r)}(q;q)&=\sum_{n\geq N_1\geq\cdots\geq N_{r-1}\geq0}\frac{q^{N_1^2+\cdots+N_{r-1}^2+N_1+\cdots+N_{j-1}+2N_j+\cdots+2N_{r-1}+{N_{r-1}\choose 2}}}{(q;q)_{n-N_1}(q;q)_{N_1-N_2}\cdots(q;q)_{N_{r-2}-N_{r-1}}}\notag\\[5pt]
    &\hskip 2cm \times\frac{1}{(q;q)_{N_{r-1}}(q^3;q^2)_{N_{r-1}}}.
\end{align*}
Using \eqref{Bailey-limit}, we deduce
\begin{align*}
& \sum_{N_1\geq\cdots\geq N_{r-1}\geq0}\frac{q^{N_1^2+\cdots+N_{r-1}^2+N_1+\cdots+N_{j-1}+2N_j+\cdots+2N_{r-1}+{N_{r-1}\choose 2}}}{(q;q)_{N_1-N_2}(q;q)_{N_2-N_3}\cdots(q;q)_{N_{r-2}-N_{r-1}}}\notag\\
    &\quad\times\frac{1}{(q;q)_{N_{r-1}}(q^3;q^2)_{N_{r-1}}} \nonumber \\[5pt]
    &=\frac{1}{(q^2;q)_\infty}\left(\sum_{n=0}^\infty(-1)^nq^{(4r-3)n^2+(4r-2j-3)n}+\sum_{n=0}^\infty  (-1)^{n+1}q^{(4r-3)n^2+(4r+2j-3)n+2j} \right)\notag\\[5pt]
    &=\frac{1}{(q^2;q)_\infty}\times \sum_{n=-\infty}^\infty(-1)^nq^{(4r-3)n^2+(4r-2j-3)n}\notag\\[5pt]
    &\overset{\eqref{eq-2.18}}{=}\frac{(q^{2j},q^{8r-2j-6},q^{8r-6};q^{8r-6})_\infty}{(q^2;q)_\infty}.
\end{align*}
This completes the proof of Theorem \ref{thm:multisuma}.

\subsection{Proof of Theorem \ref{thm:multisumb}}  \label{subsec2.3}

We first recall some Bailey pairs from  Group G of Slater's list \cite[p. 469]{Slater-1951}, which are necessary in the proof of  Theorem \ref{thm:multisumb}.
\begin{align*}
    {\rm (G1)}\quad&\alpha_n(1;q)=(-1)^nq^{\frac{1}{2}n^2+\frac{1}{2}{n\choose 2}}(1+q^{n/2}),\quad\beta_n(1;q)=\frac{1}{(-q^{\frac{1}{2}};q)_n(q^2;q^2)_n};\\[10pt]
    {\rm (G1^*)}\quad&\alpha_n(q;q)=(-1)^nq^{\frac{3}{2}{n+1\choose 2}}\frac{q^{-n}-q^{n+1}}{1-q},\quad\beta_n(q;q)=\frac{1}{(-q^{\frac{1}{2}};q)_n(q^2;q^2)_n};\\[10pt]
    {\rm (G2)}\quad&\alpha_n(q;q)=(-1)^nq^{\frac{3}{2}{n+1\choose2}}\frac{q^{-\frac{n}{2}}-q^{\frac{n+1}{2}}}{1-q^{\frac{1}{2}}},\quad\beta_n(q;q)=\frac{1}{(-q^{\frac{3}{2}};q)_n(q^2;q^2)_n};\\[10pt]
    {\rm (G4)}\quad&\alpha_n(1;q)=(-1)^nq^{\frac{1}{2}{n\choose 2}}(1+q^{\frac{1}{2}n}),\quad\beta_n(1;q)=\frac{(-1)^nq^{\frac{1}{2}n^2}}{(-q^{\frac{1}{2}};q)_n(q^2;q^2)_n};\\[10pt]
    {\rm (G4^*)}\quad&\alpha_n(1;q)=(-1)^nq^{\frac{1}{2}{n\choose 2}-\frac{1}{2}n}(1+q^{\frac{3}{2}n}),\quad\beta_n(1;q)=\frac{(-1)^nq^{\frac{1}{2}n^2-n}}{(-q^{\frac{1}{2}};q)_n(q^2;q^2)_n};\\[10pt]
    {\rm (G4^{**})}\quad &\alpha_n'(q;q)=(-1)^n\frac{1-q^{2n+1}}{1-q}q^{\frac{1}{4}n^2-\frac{3}{4}n},\quad\beta_n'(q;q)=\frac{(-1)^nq^{\frac{1}{2}n^2-n}}{(-q^{\frac{1}{2}};q)_n(q^2;q^2)_n}; \\[10pt]
    {\rm (G5)}\quad&\alpha_n(q;q)=(-1)^n\frac{q^{\frac{1}{2}{n\choose 2}}(1-q^{n+\frac{1}{2}})}{1-q^{\frac{1}{2}}},\quad\beta_n(q;q)=\frac{(-1)^nq^{\frac{1}{2}n^2}}{(-q^{\frac{3}{2}};q)_n(q^2;q^2)_n}.\\[5pt]
\end{align*}

 The Bailey pair (G1\(^*\)) was established by B. Wang--L. Wang  \cite[Lemma 2.5, Eq. (2.13)]{Wang-Wang-2024} via using a generalization of a limiting case of Jackson's basic analogue of Dougall's theorem,  as originally derived by Slater \cite[Eq. (4.2)]{Slater-1951}. In fact,   (G1\(^*\)) can also be derived by   applying the  base-change method due to Lovejoy  \cite[Lemma 2.4]{Lovejoy-2022} to  the Bailey pair (G1). We note that (G4\(^*\)) refers to the unlabeled Bailey pair corresponding to (G4) in \cite{Slater-1951}. Meanwhile, the Bailey pair that we term  (G4\(^{**}\))  first appeared in the work of B. Wang--L. Wang \cite[Eq. (2.15)]{Wang-Wang-2025-2}, who obtained it by applying base-change method due to Lovejoy \cite[Lemma 2.4]{Lovejoy-2022} to the Bailey pair (G4\(^*\)).
A typo for (G5) in \cite{Slater-1951} and \cite{Wang-Wang-2025-2} has been
corrected here.

We are now ready to prove Theorem  \ref{thm:multisumb},  establishing each of its two identities in turn.

\noindent {\bf $(a)$ Proof of  \eqref{Nahmsumsh}.}
Applying Lemma \ref{lemma-S1} (with \(a=q\)) \(r-1\) times to the Bailey pair (G5), we derive
    \begin{align*}
        &\alpha_n^{(r-1)}(q;q)=(-1)^n\frac{q^{(r-\frac{3}{4})n^2+(r-\frac{5}{4})n}(1-q^{n+\frac{1}{2}})}{1-q^{\frac{1}{2}}},\\
        &\beta_n^{(r-1)}(q;q)=\sum_{n\geq N_1\geq\cdots\geq N_{r-1}\geq0}\frac{(-1)^{N_{r-1}}q^{N_1^2+\cdots+N_{r-1}^2+N_1+\cdots+N_{r-1}+\frac{1}{2}N_{r-1}^2}}{(q;q)_{n-N_1}(q;q)_{N_1-N_2}\cdots(q;q)_{N_{r-2}-N_{r-1}}(-q^{\frac{3}{2}};q)_{N_{r-1}}(q^2;q^2)_{N_{r-1}}}.
    \end{align*}
It follows from \eqref{Bailey-limit} that
    \begin{align*}
        &\sum_{N_1\geq\cdots\geq N_{r-1}\geq0}\frac{(-1)^{N_{r-1}}q^{2(N_1^2+\cdots+N_{r-1}^2+N_1+\cdots+N_{r-1})+N_{r-1}^2}}{(q^2;q^2)_{N_1-N_2}\cdots(q^2;q^2)_{N_{r-2}-N_{r-1}}(-q^3;q^2)_{N_{r-1}}(q^4;q^4)_{N_{r-1}}}\\
        &=\frac{1}{(1-q)(q^4;q^2)_\infty}\sum_{n=0}^\infty (-1)^nq^{(2r-3/2)n^2+(2r-5/2)n}(1-q^{2n+1})\\[5pt]
        &\overset{\eqref{eq-2.18}}{=}\frac{(q,q^{4r-4},q^{4r-3};q^{4r-3})_\infty}{(1-q)(q^4;q^2)_\infty}.
    \end{align*}

\noindent  {\bf  $(b)$ Proof of  \eqref{Nahmsumsj}.}
There are two cases to be considered.

Case 1: If \(j=r,\) then applying Lemma \ref{lemma-S1} (with \(a=1\)) \(r-1\) times to the Bailey pair (G4\(^*\)), we obtain
    \begin{align*}
        &\alpha_n^{(r-1)}(1;q)=(-1)^nq^{(r-3/4)n^2-3n/4}(1+q^{3n/2}),\\
        &\beta_n^{(r-1)}(1;q)=\sum_{n\geq N_1\geq\cdots\geq N_{r-1}}\frac{(-1)^{N_{r-1}}q^{N_1^2+\cdots+N_{r-1}^2+\frac{1}{2}N_{r-1}^2-N_{r-1}}}{(q;q)_{n-N_1}(q;q)_{N_1-N_2}\cdots(q;q)_{N_{r-2}-N_{r-1}}(-q^{1/2};q)_{N_{r-1}}(q^2;q^2)_{N_{r-1}}}.
    \end{align*}
By \eqref{Bailey-limit}, we get
    \begin{align*}
        &\sum_{N_1\geq\cdots\geq N_{r-1}\geq0}\frac{(-1)^{N_{r-1}}q^{2(N_1^2+\cdots+N_{r-1}^2)+N_{r-1}^2-2N_{r-1}}}{(q^2;q^2)_{N_1-N_2}\cdots(q^2;q^2)_{N_{r-2}-N_{r-1}}(-q;q^2)_{N_{r-1}}(q^4;q^4)_{N_{r-1}}}\\
        &=\frac{1}{(q^2;q^2)_\infty}\left(1+\sum_{n=1}^\infty(-1)^nq^{(2r-3/2)n^2-3n/2}(1+q^{3n})\right)\\[5pt]
        &\overset{\eqref{eq-2.18}}{=}\frac{(q^{2r-3},q^{2r},q^{4r-3};q^{4r-3})_\infty}{(q^2;q^2)_\infty}.
    \end{align*}

    Case 2: If \(1\leq j <r,\) then applying Lemma \ref{lemma-S1} \(r-1-j\) (with \(a=q\)) to the Bailey pair (G4\({^{**}}\)),  we obtain
        \begin{align}
            \alpha_n^{(r-1-j )}(q;q)&=(-1)^nq^{(r-j -\frac{3}{4})n^2+(r-j -\frac{7}{4})n}\frac{1-q^{2n+1}}{1-q},\notag\\
            \beta_n^{(r-1-j )}(q;q)&=\sum_{n\geq N_{j+1}\geq\cdots\geq N_{r-1}\geq0}\frac{(-1)^{N_{r-1}}q^{N_{j+1}^2+\cdots+N_{r-1}^2+N_{j+1}+\cdots+N_{r-1}+\frac{1}{2}N_{r-1}^2-N_{r-1}}}{(q;q)_{n-N_{j+1}}(q;q)_{N_{j+1}-N_{i+2}}\cdots(q;q)_{N_{r-2}-N_{r-1}}}\notag\\
            &\quad\times\frac{1}{(-q^{1/2};q)_{N_{r-1}}(q^2;q^2)_{N_{r-1}}}.\label{eq-2.25}
        \end{align}
        Applying Lemma \ref{lem-1} (with \(a=q\)) to the Bailey pair \eqref{eq-2.25}, we get
        \begin{align}
            \alpha_n^{(r-j)}(1;q)&=(-1)^nq^{(r-j -\frac{3}{4})n^2+(r-j -\frac{3}{4})n}(1+q^{(2j-2r+\frac{3}{2})n}),\notag\\
            \beta_n^{(r-j)}(1;q)&=\sum_{n\geq N_{j+1}\geq\cdots\geq N_{r-1}\geq0}\frac{(-1)^{N_{r-1}}q^{N_{j+1}^2+\cdots+N_{r-1}^2+N_{j+1}+\cdots+N_{r-1}+n+\frac{1}{2}N_{r-1}^2-N_{r-1}}}{(q;q)_{n-N_{j+1}}(q;q)_{N_{j+1}-N_{i+2}}\cdots(q;q)_{N_{r-2}-N_{r-1}}}\notag\\
            &\quad\times\frac{1}{(-q^{1/2};q)_{N_{r-1}}(q^2;q^2)_{N_{r-1}}}.\label{eq-2.26}
        \end{align}
        Applying Lemma \ref{lemma-S1} (with \(a=1\))  \(j\) times to \eqref{eq-2.26}, we obtain the Bailey pair:
        \begin{align*}
            &\alpha_n^{(r)}(1;q)=(-1)^nq^{(r-\frac{3}{4})n^2+(r-j -\frac{3}{4})n}(1+q^{(2j-2r+\frac{3}{2})n}),\\
            &\beta_n^{(r)}(1;q)=\sum_{n\geq N_1\geq\cdots\geq N_{r-1}\geq0}\frac{(-1)^{N_{r-1}}q^{N_1^2+\cdots+N_{r-1}^2+N_j+\cdots+N_{r-1}+\frac{1}{2}N_{r-1}^2-N_{r-1}}}{(q;q)_{n-N_1}(q;q)_{N_1-N_2}\cdots(q;q)_{N_{r-2}-N_{r-1}}(-q^{1/2};q)_{N_{r-1}}(q^2;q^2)_{N_{r-1}}}.
        \end{align*}
Due to \eqref{Bailey-limit}, we deduce that
        \begin{align*}
            &\sum_{ N_1\geq\cdots\geq N_{r-1}\geq0}\frac{(-1)^{N_{r-1}}q^{2(N_1^2+\cdots+N_{r-1}^2+N_j+\cdots+N_{r-1})+N_{r-1}^2-2N_{r-1}}}{(q^2;q^2)_{N_1-N_2}\cdots(q^2;q^2)_{N_{r-2}-N_{r-1}}(-q;q^2)_{N_{r-1}}(q^4;q^4)_{N_{r-1}}}\\[5pt]
            &=\frac{1}{(q^2;q^2)_\infty}\left(\sum_{n=1}^\infty(-1)^nq^{(2r-\frac{3}{2})n^2+(2r-2j-\frac{3}{2})n}(1+q^{(4j-4r+3)n})+1\right)\\[5pt]
            &\overset{\eqref{eq-2.18}}{=}\frac{(q^{2j},q^{4r-2j-3},q^{4r-3};q^{4r-3})_\infty}{(q^2;q^2)_\infty}.
        \end{align*}
This finishes the proof of Theorem  \ref{thm:multisumb}.

{
\subsection{Proof of  Theorem \ref{thm-SumC}}\label{subsec-thmSC-pr} Here we require the following Bailey pairs, established by B. Wang--L. Wang \cite[Lemma 2.5 and Lemma 2.8]{Wang-Wang-2025-2}:
\begin{equation}\label{W-Bailey}
\begin{split}
&\alpha_0(1;q)=1, \alpha_{2n}(1;q)=(-1)^nq^{3n^2}(q^{3n}+q^{-3n}),
\alpha_{2n+1}(1;q)=0,\\
&\beta_n(1;q)=\frac{q^n(1+q^{-1})-q^{2n-1}}{(q;q)_n(q;q^2)_n},
\end{split}
\end{equation}
and
\begin{equation}\label{W-Bailey-2}
\begin{split}
&\alpha_{n}(q;q)=(-1)^nq^{\frac{3}{4}(n^2-n)}\frac{1-q^{3n+\frac{3}{2}}}{1-q^{\frac{1}{2}}},\\
&\beta_n(q;q)=\frac{q^n(1+q+q^{n+\frac{1}{2}})}{(-q^{3/2};q)_n(q;q^2)_n}.
\end{split}
\end{equation}

We now prove each of the   identities \eqref{SumC-1} and \eqref{SumC-2} in Theorem \ref{thm-SumC}  one by one.

$(a)$ {\bf Proof of  \eqref{SumC-1}.}
Iterating Lemma \ref{lemma-S1} (with $a=1$) $r-2$ times
on the Bailey pair \eqref{W-Bailey} yields
\begin{equation}\label{neq-BP-w1}
\begin{split}
&\alpha_0^{(r-2)}(1;q)=1,\
\alpha_{2n}^{(r-2)}(1;q)=(-1)^n q^{(4r-5)n^2}(q^{3n}+q^{-3n}),\
\alpha_{2n+1}^{(r-2)}(1;q)=0,\\
&\beta^{(r-2)}_n(1;q)=\sum_{n\geq N_2\geq \cdots N_{r-1}\geq 0}
\frac{q^{N_2^2+N_3^2+\cdots +N_{r-1}^2+N_{r-1}}(1+q^{-1}-q^{N_{r-1}-1})}{(q;q)_{n-N_2}\cdots (q;q)_{N_{r-2}-N_{r-1}}(q;q)_{N_{r-1}}(q;q^2)_{N_{r-1}}}.
\end{split}
\end{equation}
By \eqref{Bailey-limit}, we have
\begin{align*}
&\quad \sum_{N_1\geq N_2\geq \cdots N_{r-1}\geq 0}
\frac{q^{N_1^2+N_2^2+N_3^2+\cdots +N_{r-1}^2+N_{r-1}}(1+q^{-1}-q^{N_{r-1}-1})}{(q;q)_{N_1-N_2}\cdots (q;q)_{N_{r-2}-N_{r-1}}(q;q)_{N_{r-1}}(q;q^2)_{N_{r-1}}}\\[5pt]
&=\frac{1}{(q;q)_\infty} \left(\sum_{n=1}^\infty (-1)^nq^{(4r-1)n^2}(q^{3n}+q^{-3n})+1\right)\\[5pt]
&=\frac{1}{(q;q)_\infty} \sum_{n=-\infty}^\infty
(-1)^nq^{(4r-1)n^2+3n}\\[5pt]
&\overset{\eqref{eq-2.18}}{=}\frac{(q^{4r+2},q^{4r-4},q^{8r-2};q^{8r-2})_\infty}
{(q;q)_\infty}.
\end{align*}

$(b)$ {\bf Proof of  \eqref{SumC-2}.}
Iterating Lemma \ref{lemma-S1} (with $a=q$) $r-2$ times
on the Bailey pair \eqref{W-Bailey-2} gives
\begin{align*}
\alpha_n^{(r-2)}(q;q)&=(-1)^nq^{(r-\frac{5}{4})n^2+(r-\frac{11}{4})n}
\frac{1-q^{3n+\frac{3}{2}}}{1-q^{\frac{1}{2}}},\\[5pt]
\beta_n^{(r-2)}(q;q)&=\sum_{n\geq N_2\geq \cdots \geq N_{r-1}\geq 0}
\frac{q^{N_2^2+N_3^2+\cdots +N_{r-1}^2+N_2+\cdots N_{r-2}+2N_{r-1}}(1+q+q^{N_{r-1}+\frac{1}{2}})}
{(q;q)_{n-N_2}(q;q)_{N_2-N_3}\cdots (q;q)_{N_{r-2}-N_{r-1}}(-q^{3/2};q)_{N_{r-1}}(q^2;q^2)_{N_{r-1}}}.
\end{align*}
It follows from \eqref{Bailey-limit} that
\begin{align*}
&\sum_{N_1\geq N_2\geq \cdots \geq N_{r-1}\geq 0}
\frac{q^{N_1^2+N_2^2+N_3^2+\cdots +N_{r-1}^2+N_1+N_2+\cdots N_{r-2}+2N_{r-1}}(1+q+q^{N_{r-1}+\frac{1}{2}})}
{(q;q)_{n-N_2}(q;q)_{N_2-N_3}\cdots (q;q)_{N_{r-2}-N_{r-1}}(-q^{1/2};q)_{N_{r-1}+1}(q^2;q^2)_{N_{r-1}}}\\[5pt]
&=\frac{1}{(1+q^{\frac{1}{2}})(q^2;q)_\infty}
\sum_{n=0}^\infty (-1)^nq^{(r-\frac{1}{4})n^2+(r-\frac{7}{4})n}
\frac{1-q^{3n+\frac{3}{2}}}{1-q^{\frac{1}{2}}}\\[5pt]
&=\frac{1}{(q;q)_\infty}\sum_{n=-\infty}^\infty (-1)^n
q^{(r-\frac{1}{4})n^2+(r-\frac{7}{4})n}\\[5pt]
&\overset{\eqref{eq-2.18}}{=}\frac{(q^{\frac{3}{2}}, q^{2r-2}, q^{2r-\frac{1}{2}};q^{2r-\frac{1}{2}})_\infty}{(q;q)_\infty}.
\end{align*}
This completes the proof of Theorem \ref{thm-SumC}.

}

\section{Proofs of {Theorems}  \ref{main1a}, \ref{Nahm-SumC}, \ref{main1b}}

Armed with identities in Theorems \ref{thm:multisuma}, \ref{thm:multisumb} and \ref{thm-SumC}, we are ready to prove {Theorems} \ref{main1a}, \ref{Nahm-SumC} and \ref{main1b}. Throughout the entire article, we let \(N_i=n_{i+1}+\cdots+n_r\) for \(1\leq i\leq r-1\) and \(N_i=0\) for \(i>r-1\).

\subsection{Proof of  Theorem \ref{main1a}}
We prove each of the six identities in the theorem one by one.

{\bf $(1)$ Proof of  \eqref{eq-1.7}.} It relies on \eqref{Nahmsumsa} from Theorem \ref{thm:multisuma}
and the following multi-sum Rogers--Ramanujan identity from
\cite[Eq. (1.16)]{Wang-Wang-2024}:
\begin{align}
 &\sum_{N_1\geq\cdots\geq N_{r-1}\geq0}\frac{q^{2(N_1^2+\cdots+N_{r-1}^2+N_1+\cdots+N_{r-1})}}{(q^2;q^2)_{N_1-N_2}\cdots(q^2;q^2)_{N_{r-2}-N_{r-1}}(q^4;q^4)_{N_{r-1}}(-q^3;q^2)_{N_{r-1}}}\nonumber \\[10pt]
    &=\frac{(q,q^{4r-2},q^{4r-1};q^{4r-1})_\infty}{(1-q)(q^4;q^2)_\infty}.\label{Nahmsumsg}
\end{align}

Let  $AD$ be the symmetric matrix given by \eqref{eq-1.3}. Then we have
\begin{align} \label{symexpan}
 &\frac{1}{2}\boldsymbol{n}^TAD\boldsymbol{n} \nonumber \\[5pt]
=&n_1^2+\sum_{j=2}^{r-1}\left(2(j-1)n_j^2\right)+\frac{r-1}{2}n_r^2+2\sum_{2\leq j<k\leq r-1}\left(2(j-1)n_jn_k\right)+n_1n_r\nonumber  \\[5pt]
&+2\sum_{j=2}^{r-1}\left((j-1)n_jn_r\right) \nonumber  \\[5pt]
=&{n_1}^2+n_1n_r+2\left(\left(n_2+\cdots+n_{r-1}+\frac{n_r}{2}\right)^2+\cdots+\left(n_{r-1}+\frac{n_r}{2}\right)^2+\frac{n_r^2}{4}\right).
  \end{align}
Since   $\boldsymbol{b}_0=(0,0,\cdots,0,0)^T$,   the multiple summation of \eqref{eq-1.7} simplifies to:
\begin{align*}
F:&=\sum_{{n}=(n_1,\ldots,n_r)^T\in\mathbb{N}^r}\frac{q^{\frac{1}{2}\boldsymbol{n}^TAD\boldsymbol{n}+\boldsymbol{n}^T\boldsymbol{b}_0}}{(q^2;q^2)_{n_1}(q^2;q^2)_{n_2}\cdots(q^2;q^2)_{n_{r-1}}(q;q)_{n_r}}\notag\\[10pt]
&=\sum_{n_1,n_2,\ldots,n_r\geq0}\frac{q^{{n_1}^2+n_1n_r+2\left(\left(n_2+\cdots+\frac{n_r}{2}\right)^2+\cdots+\left(n_{r-1}+\frac{n_{r}}{2}\right)^2+\frac{n_{r}^2}{4}\right)}}{(q^2;q^2)_{n_1}(q^2;q^2)_{n_2}\cdots(q^2;q^2)_{n_{r-1}}(q;q)_{n_{r}}}.
\end{align*}
For \(\sigma\in\{0,1\},\) let
    \begin{align}
        F_\sigma&=\sum_{n_r\equiv\sigma\pmod2\atop n_1,n_2,\ldots,n_r\geq0}\frac{q^{{n_1}^2+n_1n_{r}+2\left(\left(n_2+\cdots+n_{r-1}+\frac{n_r}{2}\right)^2+\cdots+\left(n_{r-1}+\frac{n_r}{2}\right)^2+\frac{n_r^2}{4}\right)}}{(q^2;q^2)_{n_1}(q^2;q^2)_{n_2}\cdots(q^2;q^2)_{n_{r-1}}(q;q)_{n_r}}\notag\\
&=\sum_{n_1,\ldots,n_r\geq0}\frac{q^{{n_1}^2+n_1(2n_r+\sigma)+2\left(\left(n_2+\cdots+n_{r-1}+\frac{2n_r+\sigma}{2}\right)^2+\cdots+
        \left(n_{r-1}+\frac{2n_r+\sigma}{2}\right)^2+\frac{(2n_r+\sigma)^2}{4}\right)}}{(q^2;q^2)_{n_1}(q^2;q^2)_{n_2}\cdots(q^2;q^2)_{n_{r-1}}(q;q)_{2n_r+\sigma}}. \label{eq-aaaa}
    \end{align}
 It is straightforward that
  \begin{align}
        F=F_0+F_1.\label{eq-3.2}
    \end{align}
We proceed to simplify $F_0$ and $F_1$ respectively.  To this end, we invoke Euler's
$q$-exponential identity \cite[Corollary 2.2]
    {Andrews-1976}:
    \begin{align}
        \sum_{n=0}^\infty\frac{q^{n\choose 2}z^n}{(q;q)_n}=(-z;q)_\infty.\label{eq-3.3}
    \end{align}
 First, for $F_0$, we find that
    \begin{align}        F_0&=\sum_{n_1,\ldots,n_r\geq0}\frac{q^{{n_1}^2+2n_1n_r+2((n_2+\cdots+n_{r-1}+n_{r})^2+\cdots+(n_{r-1}+n_r)^2+n_r^2)}}{(q^2;q^2)_{n_1}(q^2;q^2)_{n_2}\cdots(q^2;q^2)_{n_{r-1}}(q;q)_{2n_r}}\notag\\[5pt]
        &=\sum_{N_1\geq\cdots\geq N_{r-1}\geq0}\frac{q^{2(N_1^2+N_2^2+\cdots+N_{r-1}^2)}}{(q^2;q^2)_{N_1-N_2}\cdots(q^2;q^2)_{N_{r-2}-N_{r-1}}(q;q)_{2N_{r-1}}}\sum_{n_1=0}^\infty\frac{q^{{n_1}^2+2n_1N_{r-1}}}{(q^2;q^2)_{n_1}}\notag\\[5pt]
        &\overset{\eqref{eq-3.3}}{=}\sum_{N_1\geq\cdots\geq N_{r-1}\geq0}\frac{q^{2(N_1^2+\cdots+N_{r-1}^2)}(-q^{2N_{r-1}+1};q^2)_\infty}{(q^2;q^2)_{N_1-N_2}\cdots(q^2;q^2)_{N_{r-2}-N_{r-1}}(q;q)_{2N_{r-1}}}\notag\\[5pt]
        &=\sum_{N_1\geq\cdots\geq N_{r-1}\geq0}\frac{(-q;q^2)_\infty q^{2(N_1^2+\cdots+N_{r-1}^2)}}{(q^2;q^2)_{N_1-N_2}\cdots(q^2;q^2)_{N_{r-2}-N_{r-1}}(q^2;q^2)_{N_{r-1}}(q^2;q^4)_{N_{r-1}}}.
        \label{eq-3.4}
    \end{align}

Similarly, for $F_1$, we obtain
    \begin{align}
        \small F_1&=\sum_{n_1,\ldots,n_r\geq0}\frac{q^{{n_1}^2+n_1(2n_r+1)+2\left(\left(n_2+\cdots+n_{r-1}+\frac{2n_r+1}{2}\right)^2+\cdots+\left(n_{r-1}+\frac{2n_r+1}{2}\right)^2+\frac{(2n_r+1)^2}{4}\right)}}{(q^2;q^2)_{n_1}(q^2;q^2)_{n_2}\cdots(q^2;q^2)_{n_{r-1}}(q;q)_{2n_r+1}}\notag\\[5pt]
    &=\sum_{n_1,N_1,\ldots,N_{r-1}\geq0}\frac{q^{{n_1}^2+n_1(2N_{r-1}+1)+2\left(\left(N_1+\frac{1}{2}\right)^2+\cdots+\left(N_{r-1}+\frac{1}{2}\right)^2\right)}}{(q^2;q^2)_{n_1}(q^2;q^2)_{N_1-N_2}\cdots(q^2;q^2)_{N_{r-2}-N_{r-1}}(q;q)_{2N_{r-1}+1}}\notag\\[5pt]
        &=q^\frac{r-1}{2}\sum_{N_1\geq\cdots\geq N_{r-1}\geq0}\frac{q^{2(N_1^2+\cdots+N_{r-1}^2+N_1+\cdots+N_{r-1})}}{(q^2;q^2)_{N_1-N_2}\cdots(q^2;q^2)_{N_{r-2}-N_{r-1}}(q;q)_{2N_{r-1}+1}}\sum_{n_1=0}^\infty\frac{q^{n_1^2+n_1(2N_{r-1}+1)}}{(q^2;q^2)_{n_1}}\notag\\[5pt]
        &\overset{\eqref{eq-3.3}}{=}q^\frac{r-1}{2}\sum_{N_1\geq\cdots\geq N_{r-1}\geq0}\frac{q^{2(N_1^2+\cdots+N_{r-1}^2+N_1+\cdots+N_{r-1})}(-q^{2N_{r-1}+2};q^2)_\infty}{(q^2;q^2)_{N_1-N_2}\cdots(q^2;q^2)_{N_{r-2}-N_{r-1}}(q;q)_{2N_{r-1}+1}}\notag\\[5pt]
        &=\frac{q^\frac{r-1}{2}(-q^2;q^2)_\infty}{(1-q)}\sum_{N_1\geq\cdots\geq N_{r-1}\geq0}\frac{q^{2(N_1^2+\cdots+N_{r-1}^2+N_1+\cdots+N_{r-1})}}{(q^2;q^2)_{N_1-N_2}\cdots(q^2;q^2)_{N_{r-2}-N_{r-1}}(q^4;q^4)_{N_{r-1}}(q^3;q^2)_{N_{r-1}}}.\label{eq-3.5}
    \end{align}
 Next, substituting \eqref{Nahmsumsa} with $q$ replaced by $q^2$ into
 \eqref{eq-3.4}, we have
    \begin{align}
        F_0
        &=\frac{(-q;q^2)_\infty}{(q^2;q^2)_\infty}(q^{8r-4},q^{8r},q^{16r-4};q^{16r-4})_\infty\notag \\[5pt]
        &=\frac{(q^{8r-4},q^{8r},q^{16r-4};q^{16r-4})_\infty}{(q,q^3,q^4;q^4)_\infty}.\label{eq-3.6}
    \end{align}
 Likewise, substituting \eqref{Nahmsumsg} with \(q\) replaced by \(-q\) into \eqref{eq-3.5}, we find
    \begin{align}
        F_1&=q^\frac{r-1}{2}\frac{(-q^2;q^2)_\infty}{1-q}\frac{(-q,q^{4r-2},-q^{4r-1};-q^{4r-1})_\infty}{(1+q)(q^4;q^2)_\infty}\notag\\
        &=q^\frac{r-1}{2}\frac{(-q,q^{4r-2},-q^{4r-1};-q^{4r-1})_\infty}{(q^2,q^2,q^4;q^4)_\infty}.\label{eq-3.7}
    \end{align}
 Finally, adding \eqref{eq-3.6} and \eqref{eq-3.7}, by \eqref{eq-3.2}, we obtain \eqref{eq-1.7}.

 %%%%%%%%%%%%%

$(2)$ {\bf Proof of  \eqref{eq-1.8}.}  It is based on  \eqref{Nahmsumsd}  from Theorem \ref{thm:multisuma} and the following multi-sum Rogers--Ramanujan identity from \cite[Eq. (1.14)]{Wang-Wang-2024}:
\begin{align}
&\sum_{N_1\geq\cdots\geq N_{r-1}\geq0}\frac{q^{2(N_1^2+\cdots+N_{r-1}^2+N_j+\cdots+N_{r-1})}}{(q^2;q^2)_{N_1-N_2}\cdots(q^2;q^2)_{N_{r-2}-N_{r-1}}(q^4;q^4)_{N_{r-1}}(-q;q^2)_{N_{r-1}}}\notag\\[5pt]
&=\frac{(q^{2j},q^{4r-2j-1},q^{4r-1};q^{4r-1})_\infty}{(q^2;q^2)_\infty}. \label{Nahmsumsi}
\end{align}

Let \(AD\) be the symmetric matrix defined by \eqref{eq-1.3}, and let \(\boldsymbol{b}_j = (1,0,\ldots,0,2,4,\cdots,2(r-1-j ),r-j )^T\) for \(1\leq j\leq r\). By virtue of   \eqref{symexpan}, we deduce that
\begin{align*}
& \frac{1}{2}\boldsymbol{n}^\mathrm{T}AD\boldsymbol{n}+\boldsymbol{n}^\mathrm{T}\boldsymbol{b}_j\\[5pt]
=&n_1^2+n_1n_r+2\left(\left(n_2+\cdots+n_{r-1}+\frac{n_r}{2}\right)^2+\cdots+\left(n_{r-1}+\frac{n_r}{2}\right)^2+\frac{n_r^2}{4}\right)\\    &+n_1+2\left(\left(n_{j+1}+\cdots+n_{r-1}+\frac{n_r}{2}\right)+\cdots+\left(n_{r-1}+\frac{n_r}{2}\right)+\frac{n_r}{2}\right).
\end{align*}
Hence the multiple summation of \eqref{eq-1.8} can be simplified to:
\begin{align*}
    H&=\sum_{\boldsymbol{n}=(n_1,\ldots,n_r)^T\in\mathbb{N}^r}\frac{q^{\frac{1}{2}\boldsymbol{n}^\mathrm{T}AD\boldsymbol{n}+\boldsymbol{n}^\mathrm{T}\boldsymbol{b}_j}}{(q^2;q^2)_{n_1}(q^2;q^2)_{n_2}\cdots(q^2;q^2)_{n_{r-1}}(q;q)_{n_r}}\\[10pt]
    &=\sum_{n_2,\ldots,n_r\geq0}\frac{q^{2\left(\left(n_2+\cdots+n_{r-1}+\frac{n_r}{2}\right)^2+\cdots+\left(n_{r-1}+\frac{n_r}{2}\right)^2+\frac{1}{4}n_r^2\right)}}{(q^2;q^2)_{n_2}\cdots(q^2;q^2)_{n_{r-1}}(q;q)_{n_r}}\\[10pt]
    &\quad\times\sum_{n_1=0}^\infty\frac{q^{n_1^2+n_1n_r+n_1+2\left(n_{j+1}+\cdots+n_{r-1}+\frac{n_r}{2}+\cdots+n_{r-1}+\frac{n_r}{2}+\frac{n_r}{2}\right)}}{(q^2;q^2)_{n_1}}.
\end{align*}
For \(\sigma\in\{0,1\}\), we define \(H_\sigma\) corresponding to \(n_r\equiv\sigma\pmod{2}\):
\begin{align*}
    H_\sigma&=\sum_{n_2,\ldots,n_r\geq0}\frac{q^{2\left(\left(n_2+\cdots+n_{r-1}+\frac{2n_r+\sigma}{2}\right)^2+\cdots+\left(n_{r-1}+\frac{2n_r+\sigma}{2}\right)^2+\frac{(2n_r+\sigma)^2}{4}\right)}}{(q^2;q^2)_{n_2}\cdots(q^2;q^2)_{n_{r-1}}(q;q)_{2n_r+\sigma}}\\[10pt]
&\quad\times\sum_{n_1=0}^\infty\frac{q^{n_1^2+n_1(2n_r+\sigma)+n_1
+2\left((n_{j+1}+\cdots+n_{r-1}+\frac{2n_r+\sigma}{2})+\cdots+(n_{r-1}+\frac{2n_r+\sigma}{2})+\frac{2n_r+\sigma}{2}\right)}}{(q^2;q^2)_{n_1}}.
\end{align*}
It is immediate that
\begin{align}
    H=H_0+H_1.\label{eq-3.13}
\end{align}
Next, we simplify \(H_0\) as follows:
\begin{align}
H_0&=\sum_{n_2,\ldots,n_r\geq0}\frac{q^{2((n_2+\cdots+n_r)^2+\cdots+(n_{r-1}+n_r)^2+n_r^2)}}{(q^2;q^2)_{n_2}\cdots(q^2;q^2)_{n_{r-1}}(q;q)_{2n_r}}\notag\\[5pt]   &\quad\times\sum_{n_1=0}^\infty\frac{q^{n_1^2+2n_1n_r+n_1+2((n_{j+1}+\cdots+n_r)+\cdots+(n_{r-1}+n_r)+n_r)}}{(q^2;q^2)_{n_1}}\notag\\[5pt]
    &=\sum_{N_1\geq\cdots\geq N_{r-1}\geq0}\frac{q^{2(N_1^2+\cdots+N_{r-1}^2+N_j+\cdots+N_{r-1})}}{(q^2;q^2)_{N_1-N_2}\cdots(q^2;q^2)_{N_{r-2}-N_{r-1}}(q;q)_{2N_{r-1}}}\sum_{n_1=0}^\infty\frac{q^{n_1^2+2n_1N_{r-1}+n_1}}{(q^2;q^2)_{n_1}}\notag\\[5pt]
    &\overset{\eqref{eq-3.3}}{=}\sum_{N_1\geq\cdots\geq N_{r-1}\geq0}\frac{q^{2(N_1^2+\cdots+N_{r-1}^2+N_j+\cdots+N_{r-1})}(-q^{2N_{r-1}+2};q^2)_\infty}{(q^2;q^2)_{N_1-N_2}\cdots(q^2;q^2)_{N_{r-2}-N_{r-1}}(q;q)_{2N_{r-1}}}\notag\\[5pt]
    &=\sum_{N_1\geq\cdots\geq N_{r-1}\geq0}\frac{q^{2(N_1^2+\cdots+N_{r-1}^2+N_j+\cdots+N_{r-1})}(-q^2;q^2)_\infty}{(q^2;q^2)_{N_1-N_2}\cdots(q^2;q^2)_{N_{r-2}-N_{r-1}}(q^4;q^4)_{N_{r-1}}(q;q^2)_{N_{r-1}}}.\label{eq-3.14}
\end{align}
Similarly, the expression for \(H_1\) is derived as
\begin{align}    H_1&=\sum_{n_2,\ldots,n_r\geq0}\frac{q^{2\left(\left(n_2+\cdots+n_{r-1}+\frac{2n_r+1}{2}\right)^2+\cdots+\left(n_{r-1}+\frac{2n_r+1}{2}\right)^2+\frac{(2n_r+1)^2}{4}\right)}}{(q^2;q^2)_{n_2}\cdots(q^2;q^2)_{n_{r-1}}(q;q)_{2n_r+1}}\notag\\[5pt]     &\quad\times\sum_{n_1=0}^\infty\frac{q^{n_1^2+n_1(2n_r+1)+n_1+2\left(\left(n_{j+1}+\cdots+n_{r-1}+\frac{2n_r+1}{2}\right)+\cdots+\left(n_{r-1}+\frac{2n_r+1}{2}\right)+\frac{2n_r+1}{2}\right)}}{(q^2;q^2)_{n_1}}\notag\\[5pt]
    &=\sum_{N_1\geq\cdots\geq N_{r-1}\geq0}\frac{q^{2\left(\left(N_1+\frac{1}{2}\right)^2+\cdots+\left(N_{r-1}+\frac{1}{2}\right)^2+N_j+\cdots+N_{r-1}+\frac{r-j }{2}\right)}}{(q^2;q^2)_{N_1-N_2}\cdots(q^2;q^2)_{N_{r-2}-N_{r-1}}(q;q)_{2N_{r-1}+1}}\sum_{n_1=0}^\infty\frac{q^{n_1^2+n_1(2N_{r-1}+1)+n_1}}{(q^2;q^2)_{n_1}}\notag\\[5pt]
    &\overset{\eqref{eq-3.3}}{=}q^\frac{3r-2j-1}{2}\sum_{N_1\geq\cdots\geq N_{r-1}\geq0}\frac{q^{2\left(N_1^2+\cdots+N_{r-1}^2+N_1+\cdots+N_{j-1}+2N_j+\cdots+2N_{r-1}\right)}(-q^{2N_{r-1}+3};q^2)_\infty}{(q^2;q^2)_{N_1-N_2}\cdots(q^2;q^2)_{N_{r-2}-N_{r-1}}(q;q)_{2N_{r-1}+1}}\notag\\[5pt]
    &=\frac{q^\frac{3r-2j-1}{2}(-q^3;q^2)_\infty}{(1-q)}\sum_{N_1\geq\cdots\geq N_{r-1}\geq0}\frac{q^{2(N_1^2+\cdots+N_{r-1}^2+N_1+\cdots+N_{j-1}+2N_j+\cdots+2N_{r-1})}}{(q^2;q^2)_{N_1-N_2}\cdots(q^2;q^2)_{N_{r-2}-N_{r-1}}(q^2;q^2)_{N_{r-1}}(q^6;q^4)_{N_{r-1}}}.\label{eq-3.15}
\end{align}
Substituting the identity \eqref{Nahmsumsi} (for \(1\leq j\leq r\)) with \(q\mapsto -q\) into \eqref{eq-3.14}, we obtain
\begin{align}
    H_0&=\frac{(-q^2;q^2)_\infty(q^{2j},-q^{4r-2j-1},-q^{4r-1};-q^{4r-1})_\infty}{(q^2;q^2)_\infty}\notag\\
    &=\frac{(q^{2j},-q^{4r-2j-1},-q^{4r-1};-q^{4r-1})_\infty}{(q^2,q^2,q^4;q^4)_\infty}.\label{eq-3.16}
\end{align}
Likewise, substituting \eqref{Nahmsumsd} (for \(1\leq j\leq r\)) with \(q\mapsto q^2\) into \eqref{eq-3.15}, we get
\begin{align}
    H_1
    &=q^\frac{3r-2j-1}{2}\frac{(-q^3;q^2)_\infty(q^{4j},q^{16r-4j -4},q^{16r-4};q^{16r-4})_\infty}{(1-q)(q^4;q^2)_\infty}\notag\\
    &=q^\frac{3r-2j-1}{2}\frac{(q^{4j},q^{16r-4j -4},q^{16r-4};q^{16r-4})_\infty}{(q,q^3,q^4;q^4)_\infty}.\label{eq-3.17}
\end{align}
Combining \eqref{eq-3.13}, \eqref{eq-3.16} and \eqref{eq-3.17}, we arrive at the identity \eqref{eq-1.8}.

%%%%%%%%%%%%%%%%%%%%%%%%%%

$(3)$ {\bf Proof of   \eqref{eq-1.9}.} It depends on \eqref{Nahmsumsb} from Theorem \ref{thm:multisuma}  and \eqref{Nahmsumsh} from Theorem \ref{thm:multisumb}.

Let \(BD\) be the symmetric matrix given by \eqref{eq-1.4}, and \(\boldsymbol{b}_0=(0,0,\ldots,0,-\frac{1}{2})^T\). Similar to \eqref{symexpan}, we derive that
\begin{align*}
   & \frac{1}{2}\boldsymbol{n}^\mathrm{T}BD\boldsymbol{n}+\boldsymbol{n}^\mathrm{T}\boldsymbol{b}_0\\[5pt]
=&n_1^2+n_1n_r+2\left(\left(n_2+\cdots+n_{r-1}+\frac{n_r}{2}\right)^2+\cdots+\left(n_{r-1}+\frac{n_r}{2}\right)^2+\frac{n_r^2}{4}\right)+\frac{n_r^2}{4}-\frac{n_r}{2}.
\end{align*}
Hence the multiple summation of \eqref{eq-1.9} can be written as:
\begin{align*}
    I&=\sum_{\boldsymbol{n}=(n_1,\ldots,n_r)^T\in\mathbb{N}^r}\frac{q^{\frac{1}{2}\boldsymbol{n}^\mathrm{T}BD\boldsymbol{n}+\boldsymbol{n}^\mathrm{T}\boldsymbol{b}_0}}{(q^2;q^2)_{n_1}(q^2;q^2)_{n_2}\cdots(q^2;q^2)_{n_{r-1}}(q;q)_{n_r}}\\
    &=\sum_{n_1,\ldots,n_r\geq0}\frac{q^{n_1^2+n_1n_r}q^{2\left(\left(n_2+\cdots+n_{r-1}+\frac{1}{2}n_r\right)^2+\cdots+\left(n_{r-1}+\frac{1}{2}n_r\right)^2+\frac{1}{4}n_r^2\right)+\frac{1}{4}n_r^2-\frac{1}{2}n_r}}{(q^2;q^2)_{n_1}(q^2;q^2)_{n_2}\cdots(q^2;q^2)_{n_{r-1}}(q;q)_{n_r}}.
\end{align*}
For \(\sigma\in\{0,1\}\), we set
\begin{align*}
    I_\sigma=\sum_{n_1,\ldots,n_r\geq0}&\frac{q^{2\left(\left(n_2+\cdots+n_{r-1}+\frac{1}{2}(2n_r+\sigma)\right)^2+\cdots+\left(n_{r-1}+\frac{1}{2}(2n_r+\sigma)\right)^2+\frac{1}{4}(2n_r+\sigma)^2\right)}}{(q^2;q^2)_{n_2}\cdots(q^2;q^2)_{n_{r-1}}(q;q)_{2n_r+\sigma}}\\
    &\times\frac{q^{n_1^2+n_1(2n_r+\sigma)}q^{\frac{1}{4}(2n_r+\sigma)^2-\frac{1}{2}(2n_r+\sigma)}}{(q^2;q^2)_{n_1}}.
\end{align*}
It follows that
\begin{align}
    I=I_0+I_1.\label{eq-3.18}
\end{align}
We proceed to simplify \(I_0\):
\begin{align}
    I_0&=\sum_{n_1,\ldots,n_r\geq0}\frac{q^{n_1^2+2n_1n_r}q^{2((n_2+\cdots+n_r)^2+\cdots+(n_{r-1}+n_r)^2+n_r^2)+n_r^2-n_r}}{(q^2;q^2)_{n_1}(q^2;q^2)_{n_2}\cdots(q^2;q^2)_{n_{r-1}}(q;q)_{2n_r}}\notag\\[5pt]
    &=\sum_{N_1\geq\cdots\geq N_{r-1}\geq0}\frac{q^{2(N_1^2+N_2^2+\cdots+N_{r-1}^2)+N_{r-1}^2-N_{r-1}}}{(q^2;q^2)_{N_1-N_2}\cdots(q^2;q^2)_{N_{r-2}-N_{r-1}}(q;q)_{2N_{r-1}}}\sum_{n_1=0}^\infty\frac{q^{n_1^2+2n_1N_{r-1}}}{(q^2;q^2)_{n_1}}\notag\\[5pt]
    &\overset{\eqref{eq-3.3}}{=}\sum_{N_1\geq\cdots\geq N_{r-1}\geq0}\frac{q^{2(N_1^2+\cdots+N_{r-1}^2)+N_{r-1}^2-N_{r-1}}(-q^{2N_{r-1}+1};q^2)_\infty}{(q^2;q^2)_{N_1-N_2}\cdots(q^2;q^2)_{N_{r-2}-N_{r-1}}(q;q)_{2N_{r-1}}}   \notag\\[5pt]
    &=(-q;q^2)_\infty\sum_{N_1\geq\cdots\geq N_{r-1}\geq0}\frac{q^{2(N_1^2+\cdots+N_{r-1}^2)}q^{N_{r-1}^2-N_{r-1}}}{(q^2;q^2)_{N_1-N_2}\cdots(q^2;q^2)_{N_{r-2}-N_{r-1}}(q^2;q^2)_{N_{r-1}}(q^2;q^4)_{N_{r-1}}}.\label{eq-3.19}
\end{align}

Similarly, the simplification of \(I_1\) yields
\begin{align}    I_1&=\sum_{n_1,\ldots,n_r\geq0}\frac{q^{n_1^2+n_1(2n_r+1)+2\left(\left(n_2+\cdots+n_{r-1}+\frac{2n_r+1}{2}\right)^2+\cdots+\frac{(2n_r+1)^2}{4}\right)+\frac{(2n_r+1)^2}{4}-\frac{2n_r+1}{2}}}{(q^2;q^2)_{n_1}(q^2;q^2)_{n_2}\cdots(q^2;q^2)_{n_{r-1}}(q;q)_{2n_{r}+1}}\notag\\[5pt]
    &=\sum_{N_1\geq\cdots\geq N_{r-1}\geq0}\frac{q^{2\left(\left(N_1+\frac{1}{2}\right)^2+\cdots+\left(N_{r-1}+\frac{1}{2}\right)^2\right)+\frac{(2N_{r-1}+1)^2}{4}-\frac{2N_{r-1}+1}{2}}}{(q^2;q^2)_{N_1-N_2}\cdots(q^2;q^2)_{N_{r-2}-N_{r-1}}(q;q)_{2N_{r-1}+1}}\sum_{n_1=0}^\infty\frac{q^{n_1^2+n_1(2N_{r-1}+1)}}{(q^2;q^2)_{n_1}}\notag\\[5pt]
    &\overset{\eqref{eq-3.3}}{=}q^{\frac{2r-3}{4}}\sum_{N_1\geq\cdots\geq N_{r-1}\geq0}\frac{q^{2(N_1^2+\cdots+N_{r-1}^2+N_1+\cdots+N_{r-1})+N_{r-1}^2}(-q^{2N_{r-1}+2};q^2)_\infty}{(q^2;q^2)_{N_1-N_2}\cdots(q^2;q^2)_{N_{r-2}-N_{r-1}}(q;q)_{2N_{r-1}+1}}\notag\\[5pt]
    &=q^{\frac{2r-3}{4}}\sum_{N_1\geq\cdots\geq N_{r-1}\geq0}\frac{q^{2(N_1^2+\cdots+N_{r-1}^2+N_1+\cdots+N_{r-1})}q^{N_{r-1}^2}(-q^2;q^2)_\infty}{(q^2;q^2)_{N_1-N_2}\cdots(q^2;q^2)_{N_{r-2}-N_{r-1}}(q^3;q^2)_{N_{r-1}}(q^4;q^4)_{N_{r-1}}(1-q)}.\label{eq-3.20}
\end{align}
Substituting \eqref{Nahmsumsb} with \(q\mapsto q^2\) into \eqref{eq-3.19}, we get
\begin{align}
    I_0&=\frac{(-q;q^2)_\infty(q^{8r-8},q^{8r-4},q^{16r-12};q^{16r-12})_\infty}{(q^2;q^2)_\infty}\notag\\[5pt]
    &=\frac{(q^{8r-8},q^{8r-4},q^{16r-12};q^{16r-12})_\infty}{(q,q^3,q^4;q^4)_\infty}.\label{eq-3.21}
\end{align}
Substituting \eqref{Nahmsumsh} with \(q\mapsto -q\) into \eqref{eq-3.20}, we obtain
\begin{align}
    I_1&=q^{\frac{2r-3}{4}}\frac{(-q^2;q^2)_\infty(-q,q^{4r-4},-q^{4r-3};-q^{4r-3})_\infty}{(1-q)(1+q)(q^4;q^2)_\infty}\notag\\[5pt]
    &=q^{\frac{2r-3}{4}}\frac{(-q,q^{4r-4},-q^{4r-3};-q^{4r-3})_\infty}{(q^2,q^2,q^4;q^4)_\infty}.\label{eq-3.22}
\end{align}
By combining \eqref{eq-3.18}, \eqref{eq-3.21} and \eqref{eq-3.22}, we establish the identity \eqref{eq-1.9}.

%%%%%%%%%%%%%%%%

$(4)$ {\bf Proof of  \eqref{eq-1.10}.} The argument utilizes \eqref{Nahmsumse} from Theorem \ref{thm:multisuma}  and \eqref{Nahmsumsj} from Theorem \ref{thm:multisumb}.

Let \(BD\) be the symmetric matrix given by \eqref{eq-1.4}, and let \(\boldsymbol{b}_j=(1,0,\ldots,0,2,4,\ldots,2(r-j -1),r-j -1)^T\). We deduce that
\begin{align*}
    &\frac{1}{2}\boldsymbol{n}^\mathrm{T}BD\boldsymbol{n}+\boldsymbol{n}^\mathrm{T}\boldsymbol{b}_j\\[5pt]
=&n_1^2+n_1n_r+2\left(\left(n_2+\cdots+n_{r-1}+\frac{n_r}{2}\right)^2+\cdots+\left(n_{r-1}+\frac{n_r}{2}\right)^2+\frac{n_r^2}{4}\right)+\frac{n_r^2}{4}\\[5pt]
        &+n_1-n_r+2\left(\left(n_{j+1}+\cdots+n_{r-1}+\frac{n_r}{2}\right)+\cdots+\left(n_{r-1}+\frac{n_r}{2}\right)+\frac{n_r}{2}\right).
\end{align*}
Thus, the multiple sum in    \eqref{eq-1.10}  takes the form
\begin{align*}
    K&=\sum_{\boldsymbol{n}=(n_1,\ldots,n_r)^T\in\mathbb{N}^r}\frac{q^{\frac{1}{2}\boldsymbol{n}^\mathrm{T}BD\boldsymbol{n}+\boldsymbol{n}^\mathrm{T}\boldsymbol{b}_j}}{(q^2;q^2)_{n_1}(q^2;q^2)_{n_2}\cdots(q^2;q^2)_{n_{r-1}}(q;q)_{n_r}}\\[5pt]    &=\sum_{n_1,\ldots,n_r\geq0}\frac{q^{2\left(\left(n_2+\cdots+n_{r-1}+\frac{1}{2}n_r\right)^2+\cdots+\left(n_{r-1}+\frac{1}{2}n_r\right)^2+\frac{1}{4}n_r^2\right)+\frac{1}{4}n_r^2-n_r}}{(q^2;q^2)_{n_2}\cdots(q^2;q^2)_{n_{r-1}}(q;q)_{n_r}}\\[5pt]    &\quad \quad \times\frac{q^{n_1^2+n_1n_r+n_1+2\left(\left(n_{j+1}+\cdots+n_{r-1}+\frac{1}{2}n_r\right)+\cdots+\left(n_{r-1}+\frac{1}{2}n_r\right)+\frac{1}{2}n_r\right)}}{(q^2;q^2)_{n_1}}.
\end{align*}
For \(\sigma\in\{0,1\}\), we set
\begin{align*}
    K_\sigma=\sum_{n_1,\ldots,n_r\geq0}&\frac{q^{2\left(\left(n_2+\cdots+n_{r-1}+\frac{1}{2}(2n_r+\sigma)\right)^2+\cdots+\left(n_{r-1}+\frac{1}{2}(2n_r+\sigma)\right)^2+\frac{1}{4}(2n_r+\sigma)^2\right)+\frac{1}{4}(2n_r+\sigma)^2-(2n_r+\sigma)}}{(q^2;q^2)_{n_2}\cdots(q^2;q^2)_{n_{r-1}}(q;q)_{2n_r+\sigma}}\\
    &\times\frac{q^{n_1^2+n_1(2n_r+\sigma)+n_1+2\left(\left(n_{j+1}+\cdots+n_{r-1}+\frac{1}{2}(2n_r+\sigma)\right)+\cdots+\left(n_{r-1}+\frac{1}{2}(2n_r+\sigma)\right)+\frac{1}{2}(2n_r+\sigma)\right)}}{(q^2;q^2)_{n_1}}.
\end{align*}
It is clear that
\begin{align}
    K=K_0+K_1.\label{eq-3.28}
\end{align}
Next, we simplify \(K_0\) as follows:
\begin{align}
    K_0&=\sum_{n_1,\ldots,n_r\geq0}\frac{q^{2((n_2+\cdots+n_r)^2+\cdots+(n_{r-1}+n_r)^2+n_r^2)+n_r^2-2n_r}}{(q^2;q^2)_{n_2}\cdots(q^2;q^2)_{n_{r-1}}(q;q)_{2n_r}}\notag\\[5pt]    &\quad \quad \quad \times\frac{q^{n_1^2+2n_1n_r+n_1+2((n_{j+1}+\cdots+n_r)+\cdots+(n_{r-1}+n_r)+n_r)}}{(q^2;q^2)_{n_1}}\notag\\[5pt]
    &=\sum_{N_1\geq\cdots\geq N_{r-1}\geq0}\frac{q^{2(N_1^2+\cdots+N_{r-1}^2+N_j+\cdots+N_{r-1})+N_{r-1}^2-2N_{r-1}}}{(q^2;q^2)_{N_1-N_2}\cdots(q^2;q^2)_{N_{r-2}-N_{r-1}}(q;q)_{2N_{r-1}}}\sum_{n_1=0}^\infty\frac{q^{n_1^2+2n_1N_{r-1}+n_1}}{(q^2;q^2)_{n_1}}\notag\\[5pt]
    &\overset{\eqref{eq-3.3}}{=}\sum_{N_1\geq\cdots\geq N_{r-1}\geq0}\frac{q^{2(N_1^2+\cdots+N_{r-1}^2+N_j+\cdots+N_{r-1})+N_{r-1}^2-2N_{r-1}}(-q^{2N_{r-1}+2};q^2)_\infty}{(q^2;q^2)_{N_1-N_2}\cdots(q^2;q^2)_{N_{r-2}-N_{r-1}}(q;q)_{2N_{r-1}}}\notag\\[5pt]
    &=\sum_{N_1\geq\cdots\geq N_{r-1}\geq0}\frac{q^{2(N_1^2+\cdots+N_{r-1}^2+N_j+\cdots+N_{r-1})+N_{r-1}^2-2N_{r-1}}(-q^2;q^2)_\infty}{(q^2;q^2)_{N_1-N_2}\cdots(q^2;q^2)_{N_{r-2}-N_{r-1}}(q^4;q^4)_{N_{r-1}}(q;q^2)_{N_{r-1}}}.\label{1.10-Eq-1}
\end{align}
The simplification of \(K_1\) gives
\begin{align}
    K_1&=\sum_{n_1,\ldots,n_r\geq0}\frac{q^{2\left((n_2+\cdots+n_{r-1}+\frac{2n_r+1}{2})^2+\cdots+(n_{r-1}+\frac{2n_r+1}{2})^2+\frac{(2n_r+1)^2}{4}\right)+\frac{(2n_r+1)^2}{4}-(2n_r+1)}}{(q^2;q^2)_{n_2}\cdots(q^2;q^2)_{n_{r-1}}(q;q)_{2n_r+1}}\notag\\[5pt]     &\quad\quad\times\frac{q^{n_1^2+n_1(2n_r+1)+n_1+2\left(\left(n_{j+1}+\cdots+n_{r-1}+\frac{2n_r+1}{2}\right)+\cdots+\left(n_{r-1}+\frac{2n_r+1}{2}\right)+\frac{2n_r+1}{2}\right)}}{(q^2;q^2)_{n_1}}\notag\\[5pt]
    &=\sum_{N_1\geq\cdots\geq N_{r-1}\geq0}\frac{q^{2\left(\left(N_1+\frac{1}{2}\right)^2+\cdots+\left(N_{r-1}+\frac{1}{2}\right)^2+N_j+\cdots+N_{r-1}+\frac{r-j }{2}\right)+N_{r-1}^2-N_{r-1}-\frac{3}{4}}}{(q^2;q^2)_{N_1-N_2}\cdots(q^2;q^2)_{N_{r-2}-N_{r-1}}(q;q)_{2N_{r-1}+1}}\notag\\[5pt]
    &\quad \sum_{n_1=0}^\infty\frac{q^{n_1^2+n_1(2N_{r-1}+1)+n_1}}{(q^2;q^2)_{n_1}}\notag\\[5pt]
    &\overset{\eqref{eq-3.3}}{=}(-q^{2N_{r-1}+3};q^2)_\infty\notag\\[5pt]
    &
    \times \sum_{N_1\geq\cdots\geq N_{r-1}\geq0}\frac{q^{2\left(N_1^2+\cdots+N_{r-1}^2+N_1+\cdots+N_{j-1}+2N_j+\cdots+2N_{r-1}+\frac{3r-2j-1}{4}\right)+N_{r-1}^2-N_{r-1}-\frac{3}{4}}}{(q^2;q^2)_{N_1-N_2}\cdots(q^2;q^2)_{N_{r-2}-N_{r-1}}(q;q)_{2N_{r-1}+1}}\notag\\[5pt]
    &=\frac{q^{\frac{6r-4j -5}{4}}(-q^3;q^2)_\infty}{(1-q)}\notag\\[5pt]
    &\times \sum_{N_1\geq\cdots\geq N_{r-1}\geq0}\frac{q^{2(N_1^2+\cdots+N_{r-1}^2+N_1+\cdots+N_{j-1}+2N_j+\cdots+2N_{r-1})+N_{r-1}^2-N_{r-1}}}{(q^2;q^2)_{N_1-N_2}\cdots(q^2;q^2)_{N_{r-2}-N_{r-1}}(q^2;q^2)_{N_{r-1}}(q^6;q^4)_{N_{r-1}}}.\label{1.10-Eq-2}
\end{align}
Substituting \eqref{Nahmsumsj} (for \(1\leq j \leq r\)) with \(q\mapsto -q\) into \eqref{1.10-Eq-1}, we obtain
\begin{align}
    K_0&=\frac{(-q^2;q^2)_\infty(q^{2j},-q^{4r-2j-3},-q^{4r-3};-q^{4r-3})_\infty}{(q^2;q^2)_\infty}\notag\\
    &=\frac{(q^{2j},-q^{4r-2j-3},-q^{4r-3};-q^{4r-3})_\infty}{(q^2,q^2,q^4;q^4)_\infty}.\label{eq-3.31}
\end{align}
Substituting \eqref{Nahmsumse} (for \(1\leq j \leq r\)) with \(q\mapsto q^2\) into \eqref{1.10-Eq-1}, we derive that
\begin{align}
    K_1
    &=q^{\frac{6r-4j -5}{4}}\frac{(-q^3;q^2)_\infty(q^{4j},q^{16r-4j -12},q^{16r-12};q^{16r-12})_\infty}{(1-q)(q^4;q^2)_\infty}\notag\\
    &=q^{\frac{6r-4j -5}{4}}\frac{(q^{4j},q^{16r-4j -12},q^{16r-12};q^{16r-12})_\infty}{(q,q^3,q^4;q^4)_\infty}.\label{eq-3.32}
\end{align}
Finally, substituting  \eqref{eq-3.31} and \eqref{eq-3.32} into \eqref{eq-3.28}, we establish the identity \eqref{eq-1.10}.

%%%%%%%%%%%%%%%%
\medskip

$(5)$ {\bf Proof of  \eqref{eq-1.11}.} This identity is equivalent to that  established by  B. Wang--L. Wang in \cite[Corollary 4.3]{Wang-Wang-2024}. We provide an explanation for completeness.

Let \(A^\vee D^\vee\) be the symmetric matrix defined by \eqref{eq-1.5}, and let \(\boldsymbol{b}_0=(1,1,2,3,\ldots,r-1)_{1\times r}^T\). We have
\begin{align*}
    &\frac{1}{2}\boldsymbol{n}^TA^\vee D^\vee\boldsymbol{n}+\boldsymbol{n}^T\boldsymbol{b}_0\\[5pt]
    =&\frac{1}{2}n_1^2+n_1n_{r}+\left((n_2+\cdots+n_r)^2+\cdots+(n_{r-1}+n_r)^2+n_r^2\right)\\[5pt]
    &+n_1+(n_2+\cdots+n_r)+\cdots+(n_{r-1}+n_r)+n_r.
\end{align*}
It follows that the multiple sum in  \eqref{eq-1.11} becomes
\begin{align*}
    &\sum_{\boldsymbol{n}=(n_1,\ldots,n_r)^T\in\mathbb{N}^r}\frac{q^{\frac{1}{2}\boldsymbol{n}^TA^\vee D^\vee\boldsymbol{n}+\boldsymbol{n}^T\boldsymbol{b}_0}}{(q;q)_{n_1}(q;q)_{n_2}\cdots(q;q)_{n_{r-1}}(q^2;q^2)_{n_r}}\\[10pt]
    =&\sum_{n_1,\ldots,n_r\geq0}\frac{q^{\frac{1}{2}n_1^2+n_1n_r+\left((n_2+\cdots+n_r)^2+\cdots+(n_{r-1}+n_r)^2+n_r^2\right)+n_1+(n_2+\cdots+n_r)+\cdots+(n_{r-1}+n_r)+n_r}}{(q;q)_{n_1}(q;q)_{n_2}\cdots(q;q)_{n_{r-1}}(q^2;q^2)_{n_r}}\\[10pt]
    =&\sum_{n_1,\ldots,n_r\geq0}\frac{q^{\frac{1}{2}n_1^2+n_1N_{r-1}+\left(N_1^2+\cdots+N_{r-1}^2\right)+n_1+N_1+\cdots+N_{r-1}}}{(q;q)_{n_1}(q;q)_{n_2}\cdots(q;q)_{n_{r-1}}(q^2;q^2)_{n_r}}\\[10pt]
    =&\frac{(-q^\frac{1}{2};q)_\infty(q^\frac{1}{2},q^{2r-1},q^{2r-\frac{1}{2}};q^{2r-\frac{1}{2}})_\infty}{(q;q)_\infty}\\[10pt]
    =&\frac{(q^\frac{1}{2},q^{2r-1},q^{2r-\frac{1}{2}};q^{2r-\frac{1}{2}})_\infty}{(q^\frac{1}{2},q^\frac{3}{2},q^2;q^2)_\infty},
\end{align*}
where the second last identity is due to
B. Wang--L. Wang \cite[Corollary 4.3]{Wang-Wang-2024}.

%%%%%%%%%%%%%%%%

$(6)$ {\bf Proof of \eqref{eq-1.12}.}   This identity is equivalent to the one established by B. Wang and L. Wang in \cite[Corollary 4.2]{Wang-Wang-2024}. For the sake of completeness, we provide an explanation herein.

Let \(A^\vee D^\vee\) be the symmetric matrix defined by \eqref{eq-1.5}, and let \(\boldsymbol{b}_j=(0,\ldots,0,1,2,\ldots,r-j)_{1\times r}^T\). We have
\begin{align*}
    &\frac{1}{2}\boldsymbol{n}^TA^\vee D^\vee\boldsymbol{n}+\boldsymbol{n}^T\boldsymbol{b}_j\\[5pt]
    =&\frac{1}{2}n_1^2+n_1n_{r}+\left((n_2+\cdots+n_r)^2+\cdots+(n_{r-1}+n_r)^2+n_r^2\right)\\[5pt]
    &+(n_{j+1}+\cdots+n_r)+\cdots+(n_{r-1}+n_r)+n_r.
\end{align*}
We consider the multiple summation of \eqref{eq-1.12},
\begin{align*}
    &\sum_{\boldsymbol{n}=(n_1,\ldots,n_r)^T\in\mathbb{N}^r}\frac{q^{\frac{1}{2}\boldsymbol{n}^TA^\vee D^\vee\boldsymbol{n}+\boldsymbol{n}^T\boldsymbol{b}_j}}{(q;q)_{n_1}(q;q)_{n_2}\cdots(q;q)_{n_{r-1}}(q^2;q^2)_{n_r}}\\[10pt]
    =&\sum_{n_1,\ldots,n_r\geq0}\frac{q^{\frac{1}{2}n_1^2+n_1n_{r}+\left((n_2+\cdots+n_r)^2+\cdots+(n_{r-1}+n_r)^2+n_r^2\right)+(n_{j+1}+\cdots+n_r)+\cdots+(n_{r-1}+n_r)+n_r}}{(q;q)_{n_1}(q;q)_{n_2}\cdots(q;q)_{n_{r-1}}(q^2;q^2)_{n_r}}\\[10pt]
    =&\sum_{n_1,\ldots,n_r\geq0}\frac{q^{\frac{1}{2}n_1^2+n_1N_{r-1}+\left(N_1^2+\cdots+N_{r-1}^2\right)+N_j+\cdots+N_{r-1}}}{(q;q)_{n_1}(q;q)_{n_2}\cdots(q;q)_{n_{r-1}}(q^2;q^2)_{n_r}}\\[10pt]
    =&\frac{(-q^\frac{1}{2};q)_\infty(q^j,q^{2r-j-\frac{1}{2}},q^{2r-\frac{1}{2}};q^{2r-\frac{1}{2}})_\infty}{(q;q)_\infty}\\[10pt]
    =&\frac{(q^j,q^{2r-j-\frac{1}{2}},q^{2r-\frac{1}{2}};q^{2r-\frac{1}{2}})_\infty}{(q^\frac{1}{2},q^\frac{3}{2},q^2;q^2)_\infty}.
\end{align*}
 The second last identity is due to B. Wang--L. Wang \cite[Corollary 4.2]{Wang-Wang-2024}.  This completes the proof of Theorem \ref{main1a}. \qed

{
\subsection{Proof of  Theorem \ref{Nahm-SumC}} We invoke   \eqref{SumC-1}   and  \eqref{SumC-2} in Theorem~\ref{thm-SumC} for this derivation.
\begin{proof}
Let
\begin{align*}
F=\sum_{n_1,\ldots, n_r\geq 0}
\frac{q^{\frac{1}{2}\boldsymbol{n}^T{A}D\boldsymbol{n}+\boldsymbol{b}_{r+1} \boldsymbol{n}}}
{(q^2;q^2)_{n_1}\cdots (q^2;q^2)_{n_{r-1}}(q;q)_{n_r}},
\end{align*}
and
\begin{align*}
F_\sigma=\sum_{n_1,\ldots, n_r\geq 0\atop n_r\equiv \sigma\pmod{2}}
\frac{q^{\frac{1}{2}\boldsymbol{n}^T{A}D\boldsymbol{n}+\boldsymbol{b}_{r+1} \boldsymbol{n}}}
{(q^2;q^2)_{n_1}\cdots (q^2;q^2)_{n_{r-1}}(q;q)_{n_r}},
\end{align*}
where $\sigma =0,1$.
By \eqref{eq-3.4},
\begin{align}
F_0&
=(-q;q^2)_\infty\sum_{N_1\geq\cdots\geq N_{r-1}\geq0}\frac{ q^{2(N_1^2+\cdots+N_{r-1}^2+N_{r-1})}}{(q^2;q^2)_{N_1-N_2}\cdots(q^2;q^2)_{N_{r-2}-N_{r-1}}(q^2;q^2)_{N_{r-1}}(q^2;q^4)_{N_{r-1}}}.\label{F-0}
\end{align}
From \eqref{eq-3.5}, we have
\begin{align}
F_1=\frac{q^\frac{r+1}{2}(-q^2;q^2)_\infty}{(1-q)}\sum_{N_1\geq\cdots\geq N_{r-1}\geq0}\frac{q^{2(N_1^2+\cdots+N_{r-1}^2+N_1+\cdots+N_{r-2})+4N_{r-1}}}{(q^2;q^2)_{N_1-N_2}\cdots(q^2;q^2)_{N_{r-2}-N_{r-1}}(q^4;q^4)_{N_{r-1}}(q^3;q^2)_{N_{r-1}}}.\label{F-1}
\end{align}
Let
\begin{align*}
H=\sum_{n_1,\ldots, n_r\geq 0}
\frac{q^{\frac{1}{2}\boldsymbol{n}^T{A}D\boldsymbol{n}+\boldsymbol{\widetilde{b}}_{r+1} \boldsymbol{n}}}
{(q^2;q^2)_{n_1}\cdots (q^2;q^2)_{n_{r-1}}(q;q)_{n_r}}.
\end{align*}
For $\sigma=0,1$, define
\begin{align*}
H_\sigma=\sum_{n_1,\ldots, n_r\geq 0\atop n_r\equiv \sigma\pmod{2}}
\frac{q^{\frac{1}{2}\boldsymbol{n}^T{A}D\boldsymbol{n}+\boldsymbol{\widetilde{b}}_{r+1} \boldsymbol{n}}}
{(q^2;q^2)_{n_1}\cdots (q^2;q^2)_{n_{r-1}}(q;q)_{n_r}}.
\end{align*}
Then $H=H_0+H_1$.

We have
\begin{align*}
&\quad \frac{1}{2}\boldsymbol{n}^T{A}D\boldsymbol{n}+\boldsymbol{\widetilde{b}}_{r+1} \boldsymbol{n}\\
&=n_1^2+n_1n_r+n_1+2\left((n_2+n_3+\cdots +n_{r-1}+\frac{n_r}{2})^2\right.\\
&\quad
\left.+(n_3+\cdots +n_{r-1}+\frac{n_r}{2})^2
+\cdots (n_{r-1}+\frac{n_r}{2})^2+\left(\frac{n_r}{2}\right)^2\right)\\
&\quad +2((n_2+n_3+\cdots +n_{r-1}+\frac{n_r}{2})
+(n_3+\cdots +n_{r-1}+\frac{n_r}{2})
+\cdots\\
&\quad + (n_{r-1}+\frac{n_r}{2})+\frac{n_r}{2})
+2\cdot \frac{n_r}{2}.
\end{align*}
It follows that
\begin{align}
&q^{\frac{r-3}{2}}H_0\notag\\[5pt]
&=
q^{\frac{r-3}{2}}\sum_{n_2,\ldots, n_r\geq 0}
q^{2((n_2+\cdots+n_r)^2+(n_3+\cdots+n_r)^2
+\cdots+(n_{r-1}+n_{r})^2+n_{r}^2)}\notag\\
&\quad \times
\frac{q^{2((n_2+\cdots+n_r)+(n_3+\cdots+n_r)
+\cdots+(n_{r-1}+n_{r})+n_{r})
+2n_{r}}}
{(q^2;q^2)_{n_2}\cdots (q^2;q^2)_{n_{r-1}}(q;q)_{2n_r}}
\sum_{n_1=0}^\infty \frac{q^{n_1^2+(2n_r+1)n_1}}{(q^2;q^2)_{n_1}}\notag\\
&\overset{\eqref{eq-3.3}}{=}
q^{\frac{r-3}{2}}(-q^2;q^2)_\infty
\sum_{n_2,\ldots, n_r\geq 0}
q^{2((n_2+\cdots+n_r)^2+(n_3+\cdots+n_r)^2
+\cdots+(n_{r-1}+n_{r})^2+n_{r}^2)}\notag\\
&\quad \times
\frac{q^{2((n_2+\cdots+n_r)+(n_3+\cdots+n_r)
+\cdots+(n_{r-1}+n_{r})+n_{r})
+2n_{r}}}
{(q^2;q^2)_{n_2}\cdots (q^2;q^2)_{n_{r-1}}(q;q)_{2n_r}(-q^2;q^2)_{n_r}}
\notag\\[5pt]
&=q^{\frac{r-3}{2}}(-q^2;q^2)_\infty
\sum_{n_2,\ldots, n_r\geq 0}
q^{2((n_2+\cdots+n_r)^2+(n_3+\cdots+n_r)^2
+\cdots+(n_{r-1}+n_{r})^2+n_{r}^2)}\notag\\
&\quad \times
\frac{q^{2((n_2+\cdots+n_r)+(n_3+\cdots+n_r)
+\cdots+(n_{r-1}+n_{r})+n_{r})
+2n_{r}}}
{(q^2;q^2)_{n_2}\cdots (q^2;q^2)_{n_{r-1}}(q;q^2)_{n_r}(q^4;q^4)_{n_r}}\notag\\
&=q^{\frac{r-3}{2}}(-q^2;q^2)_\infty
\notag\\
&\quad \times
\sum_{N_1\geq N_2\geq\cdots\geq N_{r-1}\geq 0}\frac{q^{2(N_1^2+N_2^2
+\cdots+N_{r-2}^2+N_{r-1}^2)+2(N_1+N_2
+\cdots+N_{r-2})
+4N_{r-1}}}
{(q^2;q^2)_{N_1-N_2}\cdots (q^2;q^2)_{N_{r-2}-N_{r-1}}(q;q^2)_{N_{r-1}}(q^4;q^4)_{N_{r-1}}},\label{H-0}
\end{align}
and
\begin{align}
&q^{\frac{r-3}{2}}H_1\notag\\[5pt]
&=
q^{\frac{r-3}{2}}\sum_{n_2,\ldots, n_r\geq 0}
q^{2((n_2+\cdots+n_r+\frac{1}{2})^2+(n_3+\cdots+n_r+\frac{1}{2})^2
+\cdots+(n_{r-1}+n_{r}+\frac{1}{2})^2+(n_{r}+\frac{1}{2})^2)}\notag\\
&\quad \times
\frac{q^{2((n_2+\cdots+n_r+\frac{1}{2})+(n_3+\cdots+n_r+\frac{1}{2})
+\cdots+(n_{r-1}+n_{r}+\frac{1}{2})+(n_{r}+\frac{1}{2}))
+2(n_{r}+\frac{1}{2})}}
{(q^2;q^2)_{n_2}\cdots (q^2;q^2)_{n_{r-1}}(q;q)_{2n_r+1}}
\sum_{n_1=0}^\infty \frac{q^{n_1^2+(2n_r+2)n_1}}{(q^2;q^2)_{n_1}}\notag\\
&\overset{\eqref{eq-3.3}}{=}
q^{\frac{r-3}{2}}(-q;q^2)_\infty
\sum_{n_2,\ldots, n_r\geq 0}
q^{2((n_2+\cdots+n_r+\frac{1}{2})^2+(n_3+\cdots+n_r+\frac{1}{2})^2
+\cdots+(n_{r-1}+n_{r}+\frac{1}{2})^2+(n_{r}+\frac{1}{2})^2)}\notag\\
&\quad \times
\frac{q^{2((n_2+\cdots+n_r+\frac{1}{2})+(n_3+\cdots+n_r+\frac{1}{2})
+\cdots+(n_{r-1}+n_{r}+\frac{1}{2})+(n_{r}+\frac{1}{2}))
+2(n_{r}+\frac{1}{2})}}
{(q^2;q^2)_{n_2}\cdots (q^2;q^2)_{n_{r-1}}(q^2;q^2)_{n_{r}}(q^2;q^4)_{n_r+1}}\notag\\[5pt]
&=(-q;q^2)_\infty\sum_{N_1\geq N_2\geq \cdots \geq N_{r-1}\geq 0}
\frac{q^{2(N_1^2+\cdots+N_{r-1}^2)+4
(N_1+\cdots N_{r-2})+6N_{r-1}+2r-2}}
{(q^2;q^2)_{N_1-N_2}\cdots (q^2;q^2)_{N_{r-2}-N_{r-1}}
(q^2;q^2)_{N_{r-1}}(q^2;q^4)_{N_{r-1}+1}}\notag\\[5pt]
&=(-q;q^2)_\infty\sum_{N_1\geq N_2\geq \cdots \geq N_{r-1}\geq 1}
\frac{q^{2(N_1^2+\cdots+N_{r-1}^2)+2N_{r-1}}(q^{-2}-q^{2N_{r-1}-2})}
{(q^2;q^2)_{N_1-N_2}\cdots (q^2;q^2)_{N_{r-2}-N_{r-1}}
(q^2;q^2)_{N_{r-1}}(q^2;q^4)_{N_{r-1}}}.\label{H-1}
\end{align}

It follows from \eqref{F-0} and \eqref{H-1} that
\begin{align*}
&\quad F_0+q^{\frac{r-3}{2}}H_1\\[5pt]
&=
(-q;q^2)_\infty\sum_{N_1\geq\cdots\geq N_{r-1}\geq0}\frac{ q^{2(N_1^2+\cdots+N_{r-1}^2+N_{r-1})}(1+q^{-2}-q^{2N_{r-1}-2})}{(q^2;q^2)_{N_1-N_2}\cdots(q^2;q^2)_{N_{r-2}-N_{r-1}}(q^2;q^2)_{N_{r-1}}(q^2;q^4)_{N_{r-1}}}.
\end{align*}
Substituting $q=q^2$ in \eqref{SumC-1}, we have
\begin{align}
F_0+q^{\frac{r-3}{2}}H_1
=\frac{(q^{8r+4},q^{8r-8},q^{16r-4};q^{16r-4})_\infty}
{(q, q^3, q^4 ;q^4)_\infty}.\label{F-H-1}
\end{align}

In view of \eqref{F-1} and \eqref{H-0}, we derive
\begin{align*}
&\quad F_1+q^{\frac{r-3}{2}}H_0\\[5pt]
&=q^{\frac{r-3}{2}}(-q^2;q^2)_\infty\\[5pt]
&\times
\sum_{N_1\geq N_2\geq\cdots \geq N_{r-1}\geq 0}\frac{q^{2(N_1^2+N_2^2
+\cdots+N_{r-2}^2+N_{r-1}^2)+2(N_1+N_2
+\cdots+N_{r-2})
+4N_{r-1}}(1+q^2-q^{2N_{r-1}+1})}
{(q^2;q^2)_{N_1-N_2}\cdots (q^2;q^2)_{N_{r-2}-N_{r-1}}(q;q^2)_{N_{r-1}+1}(q^4;q^4)_{N_{r-1}}}.
\end{align*}
Letting $q\rightarrow q^2$ and then $q\rightarrow -q$
in \eqref{SumC-2}, we have
\begin{align}
F_1+q^{\frac{r-3}{2}}H_0=q^{\frac{r-3}{2}}
\frac{(-q^3,q^{4r-4},-q^{4r-1}; -q^{4r-1})_\infty}{(q^2,q^2,q^4;q^4)_\infty}.\label{F-H-2}
\end{align}
Therefore, combining \eqref{F-H-1} and \eqref{F-H-2} yields
\begin{align*}
F+q^{\frac{r-3}{2}} H=
\frac{(q^{8r+4},q^{8r-8},q^{16r-4};q^{16r-4})_\infty}
{(q, q^3, q^4 ;q^4)_\infty}
+q^{\frac{r-3}{2}}
\frac{(-q^3,q^{4r-4},-q^{4r-1}; -q^{4r-1})_\infty}{(q^2,q^2,q^4;q^4)_\infty},
\end{align*}
which is exactly \eqref{SumC-eq}.
\end{proof}}

\subsection{Proof of  Theorem \ref{main1b}}

We prove the two identities stated in the theorem case by case.

$(1)$ {\bf Proof of  \eqref{eq-1.13}.}  This follows from \eqref{Nahmsumsh}  in Theorem \ref{thm:multisumb}.

Let \(B^\vee D^\vee\) be the symmetric matrix defined by \eqref{eq-1.6}, and \(\boldsymbol{b}_0=(1,1,2,3,\ldots,r-1)^T.\) We deduce that
\begin{align*}
    &\frac{1}{2}\boldsymbol{n}^TB^\vee D^\vee\boldsymbol{n}+\boldsymbol{n}^T\boldsymbol{b}_0\\
    =&\frac{1}{2}n_1^2+n_1n_r+n_1+(n_2+\cdots+n_r)^2+\cdots+n_r^2+(n_2+\cdots+n_r)+\cdots+n_r+\frac{1}{2}n_r^2.
\end{align*}
Then
\begin{align*}
    &\sum_{\boldsymbol{n}=(n_1,\ldots,n_r)^T\in\mathbb{N}^r}\frac{(-1)^{n_r}q^{\boldsymbol{n}^TB^\vee D^\vee\boldsymbol{n}+2\boldsymbol{n}^T\boldsymbol{b}_0}}{(q^2;q^2)_{n_1}\cdots(q^2;q^2)_{n_{r-1}}(q^4;q^4)_{n_r}}\\[5pt]
&=\sum_{n_1,\ldots,n_r\geq0}\frac{(-1)^{n_r}q^{n_1^2+2n_1n_r+2n_1+2\left((n_2+\cdots+n_r)^2+\cdots+n_r^2+(n_2+\cdots+n_r)+\cdots+n_r\right)+n_r^2}}{(q^2;q^2)_{n_1}\cdots(q^2;q^2)_{n_{r-1}}(q^4;q^4)_{n_r}}\\[5pt]
    &=\sum_{N_1\geq\cdots\geq N_{r-1}\geq0}\frac{(-1)^{N_{r-1}}q^{2\left(N_1^2+\cdots+N_{r-1}^2+N_1+\cdots+N_{r-1}\right)+N_{r-1}^2}}{(q^2;q^2)_{N_1-N_2}\cdots(q^2;q^2)_{N_{r-2}-N_{r-1}}(q^4;q^4)_{N_{r-1}}}\sum_{n_1=0}^\infty\frac{q^{n_1^2+2n_1N_{r-1}+2n_1}}{(q^2;q^2)_{n_1}}\\[5pt]
    &\overset{\eqref{eq-3.3}}{=}\sum_{N_1\geq\cdots\ge N_{r-1}\geq0}\frac{(-1)^{N_{r-1}}q^{2\left(N_1^2+\cdots+N_{r-1}^2+N_1+\cdots+N_{r-1}\right)+N_{r-1}^2}(-q^{2N_{r-1}+3};q^2)_\infty}{(q^2;q^2)_{N_1-N_2}\cdots(q^2;q^2)_{N_{r-2}-N_{r-1}}(q^4;q^4)_{N_{r-1}}}\\[5pt]
    &=(-q^3;q^2)_\infty\sum_{N_1\geq\cdots\geq N_{r-1}\geq0}\frac{(-1)^{N_{r-1}}q^{2\left(N_1^2+\cdots+N_{r-1}^2+N_1+\cdots+N_{r-1}\right)+N_{r-1}^2}}{(q^2;q^2)_{N_1-N_2}\cdots(q^2;q^2)_{N_{r-2}-N_{r-1}}(-q^3;q^2)_{N_{r-1}}(q^4;q^4)_{N_{r-1}}}\\[5pt]
    &\overset{\eqref{Nahmsumsh}}{=}\frac{(-q^3;q^2)_\infty(1+q)(q,q^{4r-4},q^{4r-3};q^{4r-3})_\infty}{(q^2;q^2)_\infty}\\[5pt]
    &=\frac{(q,q^{4r-4},q^{4r-3};q^{4r-3})_\infty}{(q,q^3,q^4;q^4)_\infty},
\end{align*}
as desired.

%%%%%%%%%%%%%%%%%%

$(2)$ {\bf Proof of  \eqref{eq-1.14}.}  This derives from \eqref{Nahmsumsj}  in Theorem \ref{thm:multisumb}.

Let \(B^\vee D^\vee\) be the symmetric matrix defined by \eqref{eq-1.6}, and let \(\boldsymbol{b}_j=(0,\ldots,0,1,2,\ldots,r-j-1,r-j-1)^T.\) We deduce that
\begin{align*}
    &\frac{1}{2}\boldsymbol{n}^TB^\vee D^\vee\boldsymbol{n}+\boldsymbol{n}^T\boldsymbol{b}_j\\
    =&\frac{1}{2}n_1^2+n_1n_r+(n_2+\cdots+n_r)^2+\cdots+n_r^2+(n_{j+1}+\cdots+n_r)+\cdots+n_r+\frac{1}{2}n_r^2-n_r.
\end{align*}
Then
\begin{align*}
    &\sum_{\boldsymbol{n}=(n_1,\ldots,n_r)^T\in\mathbb{N}^r}\frac{(-1)^{n_r}q^{\boldsymbol{n}^TB^\vee D^\vee\boldsymbol{n}+2\boldsymbol{n}^T\boldsymbol{b}_j}}{(q^2;q^2)_{n_1}\cdots(q^2;q^2)_{n_{r-1}}(q^4;q^4)_{n_r}}\\[5pt]
&=\sum_{n_1,\ldots,n_r\geq0}\frac{(-1)^{n_r}q^{n_1^2+2n_1n_r+2\left((n_2+\cdots+n_r)^2+\cdots+n_r^2+(n_{j+1}+\cdots+n_r)+\cdots+n_r\right)+n_r^2-2n_r}}{(q^2;q^2)_{n_1}\cdots(q^2;q^2)_{n_{r-1}}(q^4;q^4)_{n_r}}\\[5pt]
    &=\sum_{N_1\geq\cdots\geq N_{r-1}\geq0}\frac{(-1)^{N_{r-1}}q^{2\left(N_1^2+\cdots+N_{r-1}+N_j+\cdots+N_{r-1}\right)+N_{r-1}^2-2N_{r-1}}}{(q^2;q^2)_{N_1-N_2}\cdots(q^2;q^2)_{N_{r-2}-N_{r-1}}(q^4;q^4)_{N_{r-1}}}\sum_{n_1=0}^\infty\frac{q^{n_1^2+2n_1N_{r-1}}}{(q^2;q^2)_{n_1}}\\[5pt]
    &\overset{\eqref{eq-3.3}}{=}\sum_{N_1\geq\cdots\geq N_{r-1}\geq0}\frac{(-1)^{N_{r-1}}q^{2\left(N_1^2+\cdots+N_{r-1}^2+N_j+\cdots+N_{r-1}\right)+N_{r-1}^2-2N_{r-1}}(-q^{2N_{r-1}+1};q^2)_\infty}{(q^2;q^2)_{N_1-N_2}\cdots(q^2;q^2)_{N_{r-2}-N_{r-1}}(q^4;q^4)_{N_{r-1}}}\\[5pt]
    &=(-q;q^2)_\infty\sum_{N_1\geq\cdots\geq N_{r-1}\geq0}\frac{(-1)^{N_{r-1}}q^{2\left(N_1^2+\cdots+N_{r-1}^2+N_j+\cdots+N_{r-1}\right)+N_{r-1}^2-2N_{r-1}}}{(q^2;q^2)_{N_1-N_2}\cdots(q^2;q^2)_{N_{r-2}-N_{r-1}}(-q;q^2)_{N_{r-1}}(q^4;q^4)_{N_{r-1}}}\\[5pt]
    &\overset{\eqref{Nahmsumsj}}{=}\frac{(-q;q^2)_\infty(q^{2j},q^{4r-3-2j},q^{4r-3};q^{4r-3})_\infty}{(q^2;q^2)_\infty}\\[5pt]
    &=\frac{(q^{2j},q^{4r-3-2j},q^{4r-3};q^{4r-3})_\infty}{(q,q^3,q^4;q^4)_\infty},
\end{align*}
which is \eqref{eq-1.14}. This finishes the proof of  Theorem \ref{main1b}. \qed

\section{Proofs of Theorems \ref{main1}, \ref{main2} and \ref{main3}}

Armed with Theorem \ref{main1a}, we are now in a position to establish the modularity of the generalized Nahm sums   specified in Theorems \ref{main1}, \ref{main2} and \ref{main3}.
Let us first review some knowledge on modular functions. It turns out that  such families of the generalized Nahm sums are certain modular functions for
some congruence subgroup $\Gamma_1(N)$.

For a positive integer $N$, the congruence subgroups $\Gamma_0(N)$ and $\Gamma_1(N)$ are defined as
\begin{align*}
\Gamma_0(N) &=\left\{\begin{pmatrix}a &b\\ c &d\end{pmatrix}\in
\Gamma\colon  c\equiv 0\bmod{N}\right\},\\[5pt]
\Gamma_1(N) &=\left\{\begin{pmatrix}a &b\\ c &d\end{pmatrix}\in
\Gamma_0(N)\colon a\equiv d\equiv 1\bmod{N}\right\},
\end{align*}
where $\Gamma$ is the
full modular group given by
\begin{align*}
\Gamma=\left\{\begin{pmatrix}a &b\\ c &d\end{pmatrix}
\colon a, b, c, d\in\mathbb{Z},~\textrm{and}~ad-bc=1\right\}.
\end{align*}
Let $\gamma=\begin{pmatrix}a &b\\ c &d\end{pmatrix}\in\Gamma$ act on $\tau\in \mathbb{C}$ by the linear fractional transformation
\begin{align*}
    \gamma\tau = \frac{a\tau+b}{c\tau+d}, \qquad \text{and} \qquad
    \gamma\infty =\lim_{\tau\rightarrow \infty} \gamma\tau.
\end{align*}

\begin{defi}\label{defi-modu} \kern-0.1cm
{\rm (\kern-0.2cm\cite[Chapter~3]{Diamond-Shurman-2005})} Let $\mathbb{H}=\{\tau\in\mathbb{C}\colon \mathrm{Im}(\tau)>0\}$.
A meromorphic function $f\colon \mathbb{H}\rightarrow \mathbb{C}$ is
called a modular function of weight $k\in\mathbb{N}^+$ for $\Gamma_1(N)$, if it satisfies the
following two conditions:

$(1)$ for all $\gamma=\begin{pmatrix}a &b\\ c &d\end{pmatrix}\in\Gamma_1(N)$, $f(\gamma\tau)=(c\tau+d)^kf(\tau)$;

$(2)$ for  $\gamma=\begin{pmatrix}a &b\\ c &d\end{pmatrix}\in\Gamma$, $(c\tau+d)^{-k}f(\gamma\tau)$ has a Fourier expansion
of the form
\begin{align*}
(c\tau+d)^{-k}f(\gamma\tau)=\sum_{n=n_\gamma}^\infty a(n)q_{w_\gamma}^n,
\end{align*}
where $a(n_\gamma)\neq 0$, $q_{w_\gamma}=e^{2\pi i\tau/w_{\gamma}}$,
and $w_\gamma$ is the minimal positive integer $h$ such that
\begin{align*}
\begin{pmatrix}1 &h\\ 0 &1\end{pmatrix}\in\gamma^{-1}\Gamma_1(N)\gamma.
\end{align*}
\end{defi}

A modular function of weight $0$ for $\Gamma_1(N)$ is referred to a modular function for $\Gamma_1(N)$.

As shown in Theorem \ref{main1a}, we find that the generalized Nahm sums specified in Theorems  \ref{main1}, \ref{main2} and \ref{main3} can be expressed as generalized eta-quotients or sums of generalized eta-quotients. Generalized eta-quotients can be viewed as the quotients of generalized Dedekind eta-functions.

For a positive
integer $\delta$ and a residue class $g\pmod {\delta}$, the
generalized Dedekind eta-function $\eta_{\delta, g}(\tau)$ is defined
by
\begin{align*}
\eta_{\delta,g}(\tau)=q^{\frac{\delta}{2}
P_2{\left(\frac{g}{\delta}\right)}}
\prod_{\substack{n>0\\ n\equiv g\pmod{\delta}}}(1-q^n)
\prod_{\substack{n>0\\ n\equiv -g\pmod{\delta}}}(1-q^n),%\label{gen}
\end{align*}
where $q=e^{2\pi i\tau}$ and
\begin{align*}
P_2(t)=\{t\}^2-\{t\}+\frac{1}{6}
\end{align*}
is the second Bernoulli function and $\{t\}$ is the fractional part of
$t$; see, for example, \cite{Robins-1994, Sch74}.

%Note that
%\begin{align*}
%\eta_{\delta,0}(\tau)=\eta^2(\delta\tau)\qquad\text{and}\qquad
%\eta_{\delta,\frac{\delta}{2}}(\tau)
%=\dfrac{\eta^2(\frac{\delta}{2}\tau)}{\eta^2(\delta\tau)}.%\label{ge-e}
%\end{align*}

%It follows from the definitions of the eta-function and generalized eta-function that
%\begin{align}
%J_m=q^{-\frac{m}{24}}\eta_m(\tau) \text{ and }
%J_{a,m}=q^{-\frac{m}{24}-\frac{a}{2}P_2(\frac{a}{m})}\eta_m(\tau)\eta_{a,m}(\tau).
%\label{J-2-eta}
%\end{align}

The following criterion established by Robins \cite[Theorem~3]{Robins-1994} can be used to determine the modularity of generalized eta-quotients. In particular, the generalized eta-quotients linked to the generalized Nahm sums
 outlined in  Theorems  \ref{main1}, \ref{main2}, and  \ref{main3} also fall into this category.

\begin{lem}{\rm (\!\cite[Theorem~3]{Robins-1994})}\label{lem-mf}
For the generalized eta-quotient $f(\tau)$ of the following form:
\begin{align}
f(\tau)=\prod_{\delta|N\atop 0\leq g<\delta} \eta_{\delta,g}^{r_{\delta,g}}(\tau),\label{gen-quotient}
\end{align}
where
\begin{align*}
r_{\delta,g}\in
\begin{cases}
\frac{1}{2}\mathbb{Z}, &\text{if } g=0 \text{ or } \frac{\delta}{2},\\
\mathbb{Z}, & \text{otherwise},
\end{cases}
\end{align*}
we assume that $f(\tau)$ satisfies the following three conditions:
\begin{align*}
(1)\  & w(f)=\sum_{\delta|N} r_{\delta, 0}=0; \quad
(2)\ \mathrm{Ord}_\infty(f)=\sum_{\delta|N\atop 0\leq g<\delta}\delta P_2\Big(\frac{g}{\delta}\Big)r_{\delta,g}=\frac{m_1}{n_1};\\[5pt]
(3)\ &\mathrm{Ord}_0(f)=\sum_{\delta|N\atop 0\leq g<\delta} \frac{N}{\delta}P_2(0)r_{\delta,g}=\frac{m_2}{n_2},
\end{align*}
where $m_j, n_j\in\mathbb{Z}$, $n_j>0, j=1,2$.
%and $\gcd(m_j,n_j)=1$ for $j=1,2$.
Let $t$ and $N_0$ be the least positive integers such that both
$tm_1/n_1$ and $N_0m_2/n_2$ are even integers.
Then $f(t\tau)$ is a modular function for $\Gamma_1(tN_0N)$.
\end{lem}

% \begin{thm}\label{main1-gg} For the generalized Nahm sum
% $\widetilde{f}_{A,\boldsymbol{b}_j,c_j, \boldsymbol{d}}(q)$ of index $({2,\ldots, 2}, 1)$,

% {\rm (1)} if  $(A,\boldsymbol{b}_j,c_j)$ {\rm (}$0\leq j\leq r${\rm )}   are the  triples  specified   in Theorem  \ref{main1},  then $\widetilde{f}_{A,\boldsymbol{b}_j,c_j,\boldsymbol{d}}(q^{32r-8})$ are modular functions for    $\Gamma_1(128(4r-1)^2)$;

% {\rm (2)}  if $(A,\boldsymbol{b}_j,c_j)$ {\rm (}$0\leq j\leq r${\rm )}  are the  triples  given in Theorem  \ref{main2},  then $\widetilde{f}_{A,\boldsymbol{b}_j,c_j,\boldsymbol{d}}(q^{32r-24})$ are modular functions for $\Gamma_1(64(4r-3)^2)$.
% \end{thm}

We now prove Theorems \ref{main1}, \ref{main2} and \ref{main3} with the aid of Theorem \ref{main1a} and Lemma \ref{lem-mf}.

\begin{proof}[Proof of Theorem \ref{main1}]
Let $(A,\boldsymbol{b}_j,{c}_j)$  be the  triple of index $\boldsymbol{d}=({2,\ldots, 2}, 1)$ given in Theorem \ref{main1}.
It follows from Theorem \ref{main1a} (1) and (2) that, for $0\leq j\leq r$,
\begin{align*}
&\widetilde{f}_{A,\boldsymbol{b}_j,{c}_j, \boldsymbol{d}}({\tau})=\widetilde{f}_{A,\boldsymbol{b}_j,{c}_j, \boldsymbol{d}}(q)\\[3pt]
&=\frac{\eta_{16r-4,4k_j}(\tau)\eta_{16r-4,0}^{1/2}(\tau)}{\eta_{4,1}(\tau)\eta_{4,0}^{1/2}(\tau)}+\frac{\eta_{8r-2,2k_j}(\tau)\eta_{16r-4,8r-4k_j-2}(\tau)\eta_{8r-2,0}^{3/2}(\tau)}
{\eta_{8r-2,4r-2k_j-1}(\tau)\eta_{4,2}(\tau)\eta_{4,0}^{1/2}(\tau)\eta_{4r-1,0}^{1/2}(\tau)\eta_{16r-4,0}^{1/2}(\tau)},
\end{align*}
where $k_0=2r-1$ and $k_j=j$ for $1\leq j\leq r$.

Let
\begin{align*}
    f_1(\tau)&=\frac{\eta_{16r-4,4k_j}(\tau)\eta_{16r-4,0}^{1/2}(\tau)}{\eta_{4,1}(\tau)\eta_{4,0}^{1/2}(\tau)},\\[10pt]
    f_2(\tau)&=\frac{\eta_{8r-2,2k_j}(\tau)\eta_{16r-4,8r-4k_j-2}(\tau)\eta_{8r-2,0}^{3/2}(\tau)}
{\eta_{8r-2,4r-2k_j-1}(\tau)\eta_{4,2}(\tau)\eta_{4,0}^{1/2}(\tau)\eta_{4r-1,0}^{1/2}(\tau)\eta_{16r-4,0}^{1/2}(\tau)}.
\end{align*}
Then $\widetilde{f}_{A,\boldsymbol{b}_j,c_j, \boldsymbol{d}}({\tau})=f_1(\tau)+f_2(\tau)$.
Take $N=16r-4$. Then we have
\begin{align*}
    &\quad \quad w(f_1)=\frac{1}{2}-\frac{1}{2}=0, \quad
    \mathrm{Ord}_0(f_1)=-r+\frac{1}{2},\\
     &\mathrm{Ord}_\infty(f_1)=\frac{64r^2-(64k_j+36)r+16k_j^2+16k_j+5}{16r-4},
\end{align*}
and
\begin{align*}
    &\quad \quad \quad w(f_2)=\frac{3}{2}-\frac{1}{2}-\frac{1}{2}-\frac{1}{2}=0,\\
    &\mathrm{Ord}_\infty(f_2)=\frac{(4 r-4 k_j-1)^2}{16 r-4},\quad
    \mathrm{Ord}_0(f_2)=-r+\frac{1}{2}.
\end{align*}
Let $t=32r-8$ and $N_0=4$. Then
\begin{align*}
  t\cdot \mathrm{Ord}_\infty(f_1)&=2(64r^2-(64k_j+36)r+16k_j^2+16k_j+5),\\
  t\cdot\mathrm{Ord}_\infty(f_2)&=2(4 r-4 k_j-1)^2
\end{align*}
and
\begin{align*}
  N_0\cdot \mathrm{Ord}_0(f_1)=N_0\cdot \mathrm{Ord}_0(f_2)=2(-2r+1)
\end{align*}
are even.
By Lemma \ref{lem-mf},
we know that $f_1((32r-8)\tau)$ and $f_2((32r-8)\tau)$ are modular functions for $\Gamma_1(128(4r-1)^2)$.
Therefore,
the function $\widetilde{f}_{A,\boldsymbol{b}_j,{c}_j, \boldsymbol{d}}((32r-8)\tau)$
(i.e., $\widetilde{f}_{A,\boldsymbol{b}_j,{c}_j, \boldsymbol{d}}(q^{32r-8})$) is
a modular function for $\Gamma_1(128(4r-1)^2)$.
\end{proof}

\begin{proof}[Proof of Theorem \ref{main2}]
Let $({B},{\boldsymbol{b}_j},{{c}_j})$ be the triple
of index $\boldsymbol{d}=({2,\ldots, 2}, 1)$ specified in Theorem \ref{main2}.
It follows from Theorem \ref{main1a} (3) and (4) that, for $0\leq j\leq r$,  we have
\begin{align*}
&\widetilde{f}_{{B},\boldsymbol{b}_j,{c}_j, \boldsymbol{d}}({\tau})=\widetilde{f}_{{B},\boldsymbol{b}_j,{c}_j, \boldsymbol{d}}(q)\\[5pt]
&=\frac{\eta_{16r-12,4\tilde{k}_j}(\tau)\eta_{16r-12,0}^{1/2}(\tau)}
{\eta_{4,1}(\tau)\eta_{4,0}^{1/2}(\tau)}+
\frac{\eta_{8r-6,2\tilde{k}_j}(\tau)\eta_{16r-12,8r-4\tilde{k}_j-6}(\tau)\eta_{8r-6,0}^{3/2}(\tau)}
{\eta_{8r-6,4r-2\tilde{k}_j-3}(\tau)\eta_{4,2}(\tau)\eta_{4,0}^{1/2}(\tau)\eta_{4r-3,0}^{1/2}(\tau)\eta_{16r-12,0}^{1/2}(\tau)},
\end{align*}
where $\tilde{k}_0=2r-2$ and $\tilde{k}_j=j$ and $1\leq j\leq r$.

Let
\begin{align*}
    g_1(\tau)&
=\frac{\eta_{16r-12,4\tilde{k}_j}(\tau)\eta_{16r-12,0}^{1/2}(\tau)}
{\eta_{4,1}(\tau)\eta_{4,0}^{1/2}(\tau)},\\[10pt]
g_2(\tau)&=\frac{\eta_{8r-6,2\tilde{k}_j}(\tau)\eta_{16r-12,8r-4\tilde{k}_j-6}(\tau)\eta_{8r-6,0}^{3/2}(\tau)}
{\eta_{8r-6,4r-2\tilde{k}_j-3}(\tau)\eta_{4,2}(\tau)\eta_{4,0}^{1/2}(\tau)\eta_{4r-3,0}^{1/2}(\tau)\eta_{16r-12,0}^{1/2}(\tau)}.
\end{align*}
Then $\widetilde{f}_{{B},\boldsymbol{b}_j,{c}_j, \boldsymbol{d}}({\tau})=g_1(\tau)+g_2(\tau)$. Set $\tilde{N}=16r-12$, we have
\begin{align*}
    & w(g_1)=\frac{1}{2}-\frac{1}{2}=0, \quad \mathrm{Ord}_0(g_1)=1-r,\\[7pt]
    &\mathrm{Ord}_\infty(g_1)=\frac{64 r^2-(64 k_j+100) r+16 \tilde{k}_j^2+48 \tilde{k}_j+39}{4(4 r-3)},
\end{align*}
and
\begin{align*}
    &w(g_2)=\frac{3}{2}-\frac{1}{2}-\frac{1}{2}-\frac{1}{2}=0, \quad
    \mathrm{Ord}_0(g_2)=1-r,\\[7pt]
    &\mathrm{Ord}_\infty(g_2)=\frac{(4 r-4 \tilde{k}_j-3)^2}{4(4 r-3)}.
\end{align*}
Let $\tilde{t}=32r-24$ and $\tilde{N}_0=2$. Then
\begin{align*}
  \tilde{t}\cdot \mathrm{Ord}_\infty(g_1)&=2(64 r^2-(64 \tilde{k}_j+100) r+16 \tilde{k}_j^2+48 \tilde{k}_j+39),\\
  \tilde{t}\cdot\mathrm{Ord}_\infty(g_2)&=2(4 r-4 \tilde{k}_j-3)^2
\end{align*}
and
\begin{align*}
  \tilde{N}_0\cdot \mathrm{Ord}_0(g_1)=\tilde{N}_0\cdot \mathrm{Ord}_0(g_2)=2(1-r)
\end{align*}
are even.
By Lemma \ref{lem-mf}, we have $g_1((32r-24)\tau)$ and $g_2((32r-24)\tau)$
are modular functions for $\Gamma_1(64(4r-3)^2)$.
Hence, $\widetilde{f}_{{B},\boldsymbol{b}_j,{c}_j, \boldsymbol{d}}((32r-24)\tau)$
(i.e., $\widetilde{f}_{{B},\boldsymbol{b}_j,{c}_j, \boldsymbol{d}}(q^{32r-24})$)
is a modular function for $\Gamma_1(64(4r-3)^2)$.  This completes the proof.
\end{proof}

% \begin{proof}[Proof of  Theorem \ref{main1} and Theorem \ref{main2}]
% Theorem \ref{main1} and Theorem \ref{main2}
% follow immediately from Theorem \ref{main1-gg}.
% \end{proof}

% As shown in Theorem \ref{main1b}, the sums $\widetilde{f}_{A,b,c, \boldsymbol{d}}(q)$
% and $\widetilde{f'}_{A',b',c', d'}(q)$ are generalized eta-quotients,
% where $({A,b,c, \boldsymbol{d}})$ are the quadruples specified in Theorem \ref{main3}
% and $({A',b',c',d'})$ are the quadruples given in Theorem \ref{MT-3}.
% Their modularity follows from the following theorems.

% \begin{thm}\label{main2-gg}
% Let $(A,\boldsymbol{b}_j,c_j)$ {\rm (}$0\leq j\leq r${\rm )}  are the  triples  specified   in Theorem  \ref{main3}.
% Then for the generalized Nahm sums
% $\widetilde{f}_{A,\boldsymbol{b}_j,c_j, \boldsymbol{d}}(q)$ with the symmetrizer $D={\rm diag} ({1,\ldots, 1}, 2)_{r\times r}$,
% we have $\widetilde{f}_{A,\boldsymbol{b}_j,c_j, \boldsymbol{d}}(q^{32r-8})$ are modular functions for   $\Gamma_1(128(4r-1)^2)$.
% \end{thm}

\begin{proof}[Proof of Theorem \ref{main3}]
Let $(A^{\vee},\boldsymbol{b}_j,c_j)$ {\rm (}$0\leq j\leq r${\rm )}  be
the  triple of index $\boldsymbol{d}^{\vee}=(1,\ldots,1,2)$ specified  in Theorem  \ref{main3}.
From Theorem \ref{main1a} (5)  and (6), we can see that, for $0\leq j\leq r$,
\begin{align*}
\widetilde{f}_{A^{\vee},\boldsymbol{b}_j,{c}_j, \boldsymbol{d}^{\vee}}(\tau)=\widetilde{f}_{A^{\vee},\boldsymbol{b}_j,{c}_j, \boldsymbol{d}^{\vee}}(q)=\frac{\eta_{2r-\frac{1}{2},k_j}(\tau)\eta_{2r-\frac{1}{2},0}^{1/2}(\tau)}
{\eta_{2,\frac{1}{2}}(\tau)\eta_{2,0}^{1/2}(\tau)},
\end{align*}
where $k_0=\frac{1}{2}$ and $k_j=j$ for $1\leq j\leq r$.

Take $N=4r-1$. Then we have
\begin{align*}
    &\quad \quad w(\widetilde{f}_{A^{\vee},\boldsymbol{b}_j,{c}_j, \boldsymbol{d}^{\vee}}(\tau))=0, \quad
    \mathrm{Ord}_0(\widetilde{f}_{A^{\vee},\boldsymbol{b}_j,{c}_j, \boldsymbol{d}^{\vee}}(\tau))=\frac{5-4 r}{8} ,\\
     &\mathrm{Ord}_\infty(\widetilde{f}_{A^{\vee},\boldsymbol{b}_j,{c}_j, \boldsymbol{d}^{\vee}}(\tau))=\frac{8 r^2-2(8 k_j+3)r+8{k_j}^2+4 k_j+1}{16 r-4}.
\end{align*}
Let $t=32r-8$ and $N_0=16$. Then both
$$t\cdot \mathrm{Ord}_\infty(\widetilde{f}_{A^{\vee},\boldsymbol{b}_j,{c}_j, \boldsymbol{d}^{\vee}}(\tau)) \text{ and } N_0\cdot \mathrm{Ord}_0(\widetilde{f}_{A^{\vee},\boldsymbol{b}_j,{c}_j, \boldsymbol{d}^{\vee}}(\tau))$$ are even.
It follows from Lemma \ref{lem-mf} that
$\widetilde{f}_{A^{\vee},\boldsymbol{b}_j,{c}_j, \boldsymbol{d}^{\vee}}((32r-8)\tau)$ (that is, \\ $\widetilde{f}_{A^{\vee},\boldsymbol{b}_j,{c}_j, \boldsymbol{d}^{\vee}}(q^{32r-8})$) is a  modular function for $\Gamma_1(128(4r-1)^2)$.
\end{proof}

\section{Modularity for vector-valued functions}\label{sect:trans}

In this section, we will prove Theorem \ref{vmf-G-AF}
and Theorem \ref{vmf-H-AF}.
We first recall some necessary definitions and background in Subsection \ref{Subsect--MF-Pre}.
Then based on Theorems \ref{main1}, \ref{main2} and  \ref{main3},
we construct four vector-valued functions, divide them into two ``Langlands dual''
pairs,
and derive modular transformation formulas for each pair.
Using these transformation formulas, we prove Theorem \ref{vmf-G-AF}
and Theorem \ref{vmf-H-AF}
in Subsection \ref{Proof-thmG} and Subsection \ref{proof-thmH},
respectively.

\subsection{Preliminary}\label{Subsect--MF-Pre}

Let $\Gamma'$ be a discrete subgroup of
$$PSL_2(\mathbb{R})=\left\{\left(a~~b \atop c~~d\right)\colon  a,b,c,d\in\mathbb{R}, \text{ and } ad-bc>0\right\}$$
with $\rho\colon \Gamma'\rightarrow GL_n(\mathbb{C})$
a representation of $\Gamma'$.
A column vector of meromorphic functions $\boldsymbol{F}(\tau)=(f_1(\tau), f_2(\tau),$ $\ldots, f_n(\tau))^T$
on the upper half-plane
is called a vector-valued automorphic form of multiplier $\rho$
and integral weight $k$ with respect to $\Gamma'$, if it satisfies for  $\gamma\in\Gamma'$,
\begin{align*}
\boldsymbol{F}(\gamma\tau)=\rho(\gamma) (c\tau+d)^k \boldsymbol{F}(\tau).
\end{align*}
See \cite{Saber-Sebbar-2017}.

Furthermore, suppose   $\Gamma'$ is a finite index subgroup of  $\Gamma$
and let $\boldsymbol{F}(\tau)=(f_1(\tau), f_2(\tau),$ $\ldots, f_n(\tau))^T$  be a vector-valued automorphic form of multiplier $\rho$
and integral weight $k$ for $\Gamma'$. If, in addition, for  all $\gamma\in \Gamma$ and each component function $f_i(\tau)$,
there exists a positive integer $N_{\gamma,i}$ such that
\begin{align*}
f_i(\gamma\tau)=(c\tau+d)^k \sum_n a_n e^{2\pi in \tau/N_{\gamma,i}}
\end{align*}
with $a_n=0$ for  $n\ll 0$,
then $\boldsymbol{F}(\tau)$ is called a vector-valued modular function
of integral weight $k$ associated with $\rho$ for  $\Gamma'$;
see \cite{Knopp-Mason-2003}.

A vector-valued automorphic form (modular function)
of weight 0 associated with $\rho$ with respect to $\Gamma'$
is referred to a vector-valued automorphic form (modular function)
associated with $\rho$ for $\Gamma'$.

To prove Theorem \ref{vmf-G-AF}
and Theorem \ref{vmf-H-AF}, we first recall some properties of Weber's modular functions and theta series. The Dedekind eta function is defined as follows:
\begin{align*}
    \eta(\tau)=q^{1/24}(q;q)_\infty,
\end{align*}
and Weber's modular functions are defined as
\begin{align*}
    \mathfrak{f}(\tau)=q^{-1/48}(-q^{1/2};q)_\infty,\,
    \mathfrak{f}_1(\tau)=q^{-1/48}(q^{1/2};q)_\infty,\,
    \mathfrak{f}_2(\tau)=q^{1/24}(-q;q)_\infty,
\end{align*}
where $q=e^{2\pi i\tau}$ and $\tau\in\mathbb{H}$.
They have the following properties:
\begin{align}
& \eta(-1 / \tau)=\sqrt{-i \tau} \eta(\tau), \quad \eta(\tau+1)=e^{\pi i / 12} \eta(\tau),
\label{eta-trans}\\[5pt]
& \mathfrak{f}(-1 / \tau)=\mathfrak{f}(\tau), \quad \mathfrak{f}_2(-1 / \tau)=\frac{1}{\sqrt{2}} \mathfrak{f}_1(\tau), \quad \mathfrak{f}_1(-1 / \tau)=\sqrt{2} \mathfrak{f}_2(\tau),
\label{f-p-trans}\\[5pt]
& \mathfrak{f}(\tau+1)=e^{-\pi i / 24} \mathfrak{f}_1(\tau), \quad \mathfrak{f}_1(\tau+1)=e^{-\pi i / 24} \mathfrak{f}(\tau), \quad \mathfrak{f}_2(\tau+1)=e^{\pi i / 12} \mathfrak{f}_2(\tau).
\label{f-a-trans}
\end{align}
For $m\in\mathbb{Q}_{>0}$ and $j\in\mathbb{Q}$, define the theta series
\begin{align*}
    h_{j, m}(\tau)=\sum_{k \in \mathbb{Z}} q^{m\left(k+\frac{j}{2 m}\right)^2}, \quad
    g_{j, m}(\tau)=\sum_{k \in \mathbb{Z}}(-1)^k q^{m\left(k+\frac{j}{2 m}\right)^2}.
\end{align*}
It follows from the Jacobi triple product identity \eqref{eq-2.18}  that
\begin{align}
& h_{j, m}(\tau)=q^{\frac{j^2}{4 m}}\left(-q^{m-j},-q^{m+j}, q^{2 m}; q^{2 m}\right)_{\infty},
\label{def-h}\\[5pt]
& g_{j, m}(\tau)=q^{\frac{j^2}{4 m}}\left(q^{m+j}, q^{m-j}, q^{2 m}; q^{2 m}\right)_{\infty}.
\label{def-g}
\end{align}
Moreover, $h_{j, m}(\tau)$ and $g_{j, m}(\tau)$ satisfy the following properties:
\begin{align}
& h_{j, m}(\tau)=h_{-j, m}(\tau)=h_{2 m+j, m}(\tau), \quad g_{j, m}(\tau)=g_{-j, m}(\tau)=-g_{2 m+j, m}(\tau),\label{iden-h}\\[5pt]
& h_{j, m}(\tau)=h_{2 j, 4 m}(\tau)+h_{4 m-2 j, 4 m}(\tau),\label{iden-h-a}\\[5pt]
& g_{j, m}(\tau)=h_{2 j, 4 m}(\tau)-h_{4 m-2 j, 4 m}(\tau),\label{iden-g-h}\\[5pt]
& h_{j, m}(2 \tau)=h_{2 j, 2 m}(\tau), \quad g_{j, m}(2 \tau)=g_{2 j, 2 m}(\tau).\label{iden-h-2}
\end{align}

\begin{lem}\label{Wak-lem}{\rm (\!\!\cite[Theorem 4.5]{Wakimoto-2001})}
For $j\in\mathbb{Z}$ and $m\in\frac{1}{2}\mathbb{N}$,
\begin{align*}
g_{j, m}\left(-\frac{1}{\tau}\right)
=\frac{(-i \tau)^{\frac{1}{2}}}{\sqrt{2 m}} \sum_{\substack{0 \leq k \leq 4 m-1 \\ k \text { odd }}} e^{\frac{\pi i j k}{2 m}} h_{\frac{k}{2}, m}(\tau).
\end{align*}
\end{lem}

\begin{lem}\label{ww-lem}{\rm(\!\!\cite[Lemma 2.13]{Wang-Wang-2025-2})}
For any integer $m$ and an odd integer $1\leq j\leq m$,
\begin{align*}
& g_{j, m}\left(-\frac{\tau+1}{4 \tau}\right) \\[5pt]
&\quad  =\sqrt{\frac{-\tau}{m}} \epsilon_m \sum_{\substack{1 \leq \ell \leq m-1 \\
\ell \text { odd }}}\left(e^{\pi i \frac{1-(j+\ell)-(j+\ell-2)^2 \delta_m}{2 m}}+e^{\pi i \frac{1-(j-\ell)-(j-\ell-2)^2 \delta_m}{2 m}}\right) g_{\ell, m}\left(\frac{\tau+1}{4}\right),
\end{align*}
where
\begin{align*}
    \epsilon_m=
    \begin{cases}
        1, &\text{if } m\equiv 1\pmod{4},\\[5pt]
        i, &\text{if } m\equiv 3\pmod{4},
    \end{cases}
\end{align*}
and $\delta_m$ is any integer satisfying $4\delta_m\equiv 1\pmod{m}$.
\end{lem}

\begin{rem}
Notice that we can choose $\delta_{4r-1}=r$ and $\delta_{4r-3}=3r-2$.
\end{rem}

We also need the following lemma.

\begin{lem}\label{inv-S}
Let $N$ be an odd integer, $m=\frac{N-1}{2}$,
$S=(s_{j,k})_{m\times m}$,
where
$s_{j,k}=\sqrt{\frac{2}{N}}\cos\frac{(2j-1)(2k-1)\pi}{2N}$.
Then $2S^2=E_m$,
where $E_m$ is the $m\times m$ identity matrix.
\end{lem}

\begin{proof}
To prove $S^2=\frac{1}{2} E_m$, let $a_{j,k}$ denote the $(j,k)$-entry of $S^2$. By matrix multiplication,
\begin{align*}
a_{j,k}&=\sum_{\ell=1}^m s_{j,\ell} s_{\ell, k}\\[5pt]
&=\frac{2}{N}\sum_{\ell=1}^m \cos\frac{(2j-1)(2\ell-1)\pi}{2N}\cos\frac{(2k-1)(2\ell-1)\pi}{2N}\\[5pt]
&=\frac{1}{N}\left(\sum_{\ell=1}^m \cos\frac{(j+k-1)(2\ell-1)\pi}{N}+
\sum_{\ell=1}^m \cos\frac{(j-k)(2\ell-1)\pi}{N}\right).
\end{align*}
Next we show that
\begin{align*}
C(t):=\sum_{\ell=1}^m \cos\frac{t(2\ell-1)\pi}{N}=
\begin{cases}
m, &\text{if } t=0,\\
\frac{(-1)^{t+1}}{2}, & \text{if } t\neq 0 \text{ and } \sin\frac{t\pi}{N}\neq 0.
\end{cases}
\end{align*}
For $t=0$, it is easy to see that $C(t)=m$.
For  $t$ with $\sin\frac{t\pi}{N}\neq 0$,
\begin{align*}
C(t)&=\mathrm{Re} \sum_{\ell=1}^m e^{\frac{it(2\ell-1)\pi}{N}}
=\mathrm{Re} \left(e^{ti\pi /N}\frac{1-e^{2mti\pi/N}}{1-e^{2ti\pi/N}}\right)\\[5pt]
&=\mathrm{Re}\left(\frac{1-(-1)^t e^{-ti\pi/N}}{-2i\sin\frac{t\pi}{N}}\right)\\[5pt]
&=\mathrm{Re}\left(\frac{1-(-1)^t (\cos\frac{t\pi}{N}-i\sin\frac{t\pi}{N})}{-2i\sin\frac{t\pi}{N}}\right)
=\frac{(-1)^{t+1}}{2}.
\end{align*}
Thus, for  $1\leq j=k\leq m$,
\begin{align*}
a_{j,j}=\frac{1}{N}(C(2j-1)+C(0))=\frac{1}{N}\left(\frac{1}{2}+m\right)=\frac{1}{2},
\end{align*}
and for  $1\leq j,k\leq m$ with $j\neq k$, we have
$1\leq j+k-1\leq 2m-1<N$ and $1\leq |j-k|\leq m-1< N$.
So
\begin{align*}
a_{j,k}=\frac{1}{N}(C(j+k-1)+C(j-k))=\frac{1}{N}\left(\frac{(-1)^{j+k}}{2}+\frac{(-1)^{j-k+1}}{2}\right)=0.
\end{align*}
Therefore, $S^2=\frac{1}{2}E_m$.
\end{proof}

\subsection{Proof of Theorem \ref{vmf-G-AF}}\label{Proof-thmG}

Recall that
\begin{align*}
\boldsymbol{g}_{2r-1}(\tau)=(g_1(\tau), g_2(\tau), \ldots, g_{2r-1}(\tau))^T,
\quad
\boldsymbol{g}_{2r-1}^{\vee}(\tau)=(g^{\vee}_1(\tau), g^{\vee}_2(\tau), \ldots, g^{\vee}_{2r-1}(\tau))^T,
\end{align*}
where ${g}_j(\tau)$ and ${g}_j^{\vee}(\tau)$ are given in \eqref{gv-def} and \eqref{g-def}.

We obtain the following modular transformation formula between
 $\boldsymbol{g}_{2r-1}(\tau)$
and $\boldsymbol{g}_{2r-1}^{\vee}(\tau)$, which is a ``Langlands dual'' pair similar to that of \cite{Mizuno-2025}.

\begin{thm}\label{vmf-G}
For  $r\geq 2$,
\begin{align*}
\boldsymbol{g}_{2r-1}\Big(-\frac{1}{2\tau}\Big)=2S\, \boldsymbol{g}_{2r-1}^{\vee}(\tau),
\quad
\boldsymbol{g}_{2r-1}^{\vee}\bigg(-\frac{1}{2\tau}\bigg)=S\,\boldsymbol{g}_{2r-1}(\tau),
\end{align*}
where
$S=(s_{j,k})_{(2r-1)\times (2r-1)}$
and $s_{j,k}=\sqrt{\frac{2}{4r-1}}\cos\frac{(2j-1)(2k-1)\pi}{2(4r-1)}$.
\end{thm}

\begin{proof} By Lemma \ref{inv-S}, it suffices to prove
\begin{align}
\boldsymbol{g}_{2r-1}\Big(-\frac{1}{2\tau}\Big)=2S\; \boldsymbol{g}_{2r-1}^{\vee}(\tau).\label{trans-gv}
\end{align}

Let
\begin{align}
V_j(\tau)&=q^{\frac{16 j^2-16 j-4 r+5}{32r-8}}\frac{(q^{8r-4j},q^{8r+4j-4},q^{16r-4};q^{16r-4})_\infty}{(q,q^3,q^4;q^4)_\infty},\notag\\[5pt]
V^*_j(\tau)&
=q^{\frac{(4r-4j+1)^2}{32r-8}}\frac{(-q^{2j-1},q^{4r-2j},-q^{4r-1};-q^{4r-1})_\infty}
{(q^2,q^2,q^4;q^4)_\infty},\notag\\[5pt]
U_j(2\tau)&=q^{\frac{2j^2-2j-2r+1}{16r-4}}
\frac{(q^{2r-j},q^{2r+j-1},q^{4r-1};q^{4r-1})_\infty}{(q,q^3,q^4;q^4)_\infty},
\label{def-Uj}\\[5pt]
U^*_j(2\tau)&=(-1)^{\frac{j(j-1)}{2}}q^{\frac{2j^2-2j-2r+1}{16r-4}}\frac{((-q)^{2r-j},(-q)^{2r+j-1},-q^{4r-1};-q^{4r-1})_\infty}
{(-q,-q^3,q^4;q^4)_\infty}.\label{def-Ujs}
\end{align}
Then for  $1\leq j \leq 2r-1$,
\begin{align}
g_j(\tau)=V_j(\tau)+V^*_j(\tau),\quad
g_j^{\vee}(2\tau)=\frac{1}{2}(U_j(2\tau)+U_j^*(2\tau)).\label{g-dis}
\end{align}
Let
\[\boldsymbol{V}_{2r-1}(\tau)=(V_1(\tau), V_2(\tau), \ldots, V_{2r-1}(\tau))^T,\]

\[\boldsymbol{V}_{2r-1}^*(\tau)=(V^*_1(\tau), V^*_2(\tau), \ldots, V^*_{2r-1}(\tau))^T,\]

\[\boldsymbol{U}_{2r-1}(\tau)=(U_1(\tau), U_2(\tau), \ldots, U_{2r-1}(\tau))^T\]

and

\[\boldsymbol{U}_{2r-1}^*(\tau)=(U^*_1(\tau), U^*_2(\tau), \ldots, U^*_{2r-1}(\tau))^T.\]

We have
\begin{align*}
\boldsymbol{g}_{2r-1}(\tau)=\boldsymbol{V}_{2r-1}(\tau)+\boldsymbol{V}_{2r-1}^*(\tau),\quad
\boldsymbol{g}_{2r-1}^{\vee}(2\tau)=\frac{1}{2}(\boldsymbol{U}_{2r-1}(2\tau)+\boldsymbol{U}_{2r-1}^*(2\tau)).
\end{align*}
Hence, in order to prove \eqref{trans-gv},
we only need to prove
\begin{align}
\boldsymbol{V}_{2r-1}\bigg(-\frac{1}{2\tau}\bigg) =S\; \boldsymbol{U}_{2r-1}(\tau),\label{trans-V-U}
\end{align}
and
\begin{align}
\boldsymbol{V}_{2r-1}^*\Big(-\frac{1}{2\tau}\Big)=S\; \boldsymbol{U}_{2r-1}^*(\tau).\label{trans-Vs-Us}
\end{align}

Next we prove \eqref{trans-V-U} and \eqref{trans-Vs-Us}, respectively.

\textbf{Proof of  \eqref{trans-V-U}.}~
In view of \eqref{def-g} and the following identity
 \begin{align*}
    \frac{1}{(q,q^3,q^4;q^4)_\infty}=\frac{(-q;q^2)_\infty}{(q^2;q^2)_\infty},
\end{align*}
we see that
\begin{align*}
    V_j(\tau)
    &=q^{\frac{16 j^2-16 j-4 r+5}{32r-8}}\frac{(-q;q^2)_\infty(q^{8r-4j},q^{8r+4j-4},q^{16r-4};q^{16r-4})_\infty}{(q^2;q^2)_\infty}\\[5pt]
    &=\frac{\mathfrak{f}(2\tau)}{\eta(2\tau)}g_{2j-1,4r-1}(2\tau),
\end{align*}
and
\begin{align}
U_j(2\tau)&
=\frac{\mathfrak{f}(2\tau)}{\eta(2\tau)}g_{2j-1,4r-1}\bigg(\frac{\tau}{2}\bigg).
\label{Uj-equi}
\end{align}
Then
\begin{align*}
    V_j\bigg(-\frac{1}{4\tau}\bigg)=\frac{\mathfrak{f}(-\frac{1}{2\tau})}{\eta(-\frac{1}{2\tau})}g_{2j-1,4r-1}\left(-\frac{1}{2\tau}\right).
\end{align*}
By means of \eqref{eta-trans}, \eqref{f-p-trans} and Lemma \ref{Wak-lem}, we have
\begin{align*}
    V_j\bigg(-\frac{1}{4\tau}\bigg)&=\frac{1}{\sqrt{8r-2}}\frac{\mathfrak{f}(2\tau)}{\eta(2\tau)}
     \sum_{\substack{0 \leq k \leq 16r-5 \\ k \text { odd }}} e^{\frac{\pi i (2j-1) k}{2(4r-1)}} h_{\frac{k}{2}, 4r-1}(2\tau)\\[5pt]
     &=\frac{1}{\sqrt{8r-2}}\frac{\mathfrak{f}(2\tau)}{\eta(2\tau)}
     \sum_{k=0}^{8r-3} e^{\frac{\pi i (2j-1) (2k+1)}{2(4r-1)}} h_{k+\frac{1}{2}, 4r-1}(2\tau).
\end{align*}
In view of \eqref{iden-h}, we deduce that
\begin{align*}
    &V_j\bigg(-\frac{1}{4\tau}\bigg)\\[5pt]
    &=\frac{1}{\sqrt{8r-2}}\frac{\mathfrak{f}(2\tau)}{\eta(2\tau)}
    \sum_{k=0}^{4r-2}\left(e^{\frac{\pi i (2j-1)(2k+1)}{2(4r-1)}} h_{k+\frac{1}{2}, 4r-1}(2\tau)
    +e^{\frac{\pi i (2j-1)(-2k-1)}{2(4r-1)}} h_{k+\frac{1}{2}, 4r-1}(2\tau)\right)\\[5pt]
    &=\sqrt{\frac{2}{4r-1}}\frac{\mathfrak{f}(2\tau)}{\eta(2\tau)}\sum_{k=0}^{4r-2}
    \cos\frac{\pi(2j-1)(2k+1)}{2(4r-1)} h_{k+\frac{1}{2}, 4r-1}(2\tau)\\[5pt]
    &=\sqrt{\frac{2}{4r-1}}\frac{\mathfrak{f}(2\tau)}{\eta(2\tau)}\sum_{k=0}^{2r-2}
    \cos\frac{\pi(2j-1)(2k+1)}{2(4r-1)} (h_{k+\frac{1}{2}, 4r-1}(2\tau)-h_{4r-k-\frac{3}{2},4r-1}(2\tau)),
\end{align*}
which, together with \eqref{iden-h-2}, yields that
\begin{align*}
    &V_j\bigg(-\frac{1}{4\tau}\bigg)\\[5pt]
    &=\sqrt{\frac{2}{4r-1}}\frac{\mathfrak{f}(2\tau)}{\eta(2\tau)}\sum_{k=0}^{2r-2}
    \cos\frac{\pi(2j-1)(2k+1)}{2(4r-1)} \bigg(h_{4k+2, 16r-4}\Big(\frac{\tau}{2}\Big)-
    h_{16r-4k-6,16r-4}\Big(\frac{\tau}{2}\Big)\bigg).
\end{align*}
From \eqref{iden-g-h}, we have
\begin{align*}
    V_j\bigg(-\frac{1}{4\tau}\bigg)&=\sqrt{\frac{2}{4r-1}}\frac{\mathfrak{f}(2\tau)}{\eta(2\tau)}\sum_{k=0}^{2r-2}
    \cos\frac{\pi(2j-1)(2k+1)}{2(4r-1)} g_{2k+1, 4r-1}\Big(\frac{\tau}{2}\Big)\\[5pt]
    &=\sqrt{\frac{2}{4r-1}}\sum_{k=1}^{2r-1}
    \cos\frac{\pi(2j-1)(2k-1)}{2(4r-1)}
    \frac{\mathfrak{f}(2\tau)}{\eta(2\tau)}g_{2k-1, 4r-1}\Big(\frac{\tau}{2}\Big).
\end{align*}
Hence,
\begin{align*}
    V_j\bigg(-\frac{1}{2\tau}\bigg)&=\sqrt{\frac{2}{4r-1}}\sum_{k=1}^{2r-1}
    \cos\frac{\pi(2j-1)(2k-1)}{2(4r-1)}
    \frac{\mathfrak{f}(\tau)}{\eta(\tau)}g_{2k-1, 4r-1}\Big(\frac{\tau}{4}\Big)\\[5pt]
    &=\sqrt{\frac{2}{4r-1}}\sum_{k=1}^{2r-1}
    \cos\frac{\pi(2j-1)(2k-1)}{2(4r-1)} U_k(\tau)~(\text{by } \eqref{Uj-equi}),
\end{align*}
which gives \eqref{trans-V-U}.

\textbf{Proof of  \eqref{trans-Vs-Us}.}~
Since
\begin{align*}
    \frac{1}{(q^2,q^2,q^4;q^4)}=\frac{(-q^2;q^2)_\infty}{(q^2;q^2)_\infty},
\end{align*}
we have
\begin{align*}
    V_j^*(\tau+\frac{1}{2})&=e^{\pi i\frac{(4r-4j+1)^2}{32r-8}}
    q^{\frac{(4r-4j+1)^2}{32r-8}}\frac{(-q^2;q^2)_\infty(q^{2j-1},q^{4r-2j},q^{4r-1};q^{4r-1})_\infty}{(q^2;q^2)_\infty}\\[5pt]
    &=
    \begin{cases}
        e^{\pi i\frac{(4r-4j+1)^2}{32r-8}} \frac{\mathfrak{f}_2(2\tau)}{\eta(2\tau)}g_{4r-4j+1,4r-1}\big( \frac{\tau}{2}\big), & \text{if }
        1\leq j\leq r,\\[8pt]
        e^{\pi i\frac{(4r-4j+1)^2}{32r-8}} \frac{\mathfrak{f}_2(2\tau)}{\eta(2\tau)}g_{4j-4r-1,4r-1}\big( \frac{\tau}{2}\big), & \text{if }
        r+1\leq j\leq 2r-1.
    \end{cases}
\end{align*}
So
\begin{align*}
    V_j^*(\tau)=
        e^{\pi i\frac{(4r-4j+1)^2}{32r-8}} \frac{\mathfrak{f}_2(2\tau-1)}{\eta(2\tau-1)}g_{|4r-4j+1|,4r-1}\Big( \frac{2\tau-1}{4}\Big).
\end{align*}
By means of \eqref{eta-trans} and \eqref{f-a-trans}, we have
\begin{align*}
    V_j^*(\tau)=
        e^{\pi i\frac{(4r-4j+1)^2}{32r-8}} \frac{\mathfrak{f}_2(2\tau)}{\eta(2\tau)}g_{|4r-4j+1|,4r-1}\Big( \frac{2\tau-1}{4}\Big).
\end{align*}
For $1\leq j\leq 2r-1$, let
\begin{align}
    W_j(\tau)=\frac{\mathfrak{f}_2(2\tau)}{\eta(2\tau)}g_{2j-1,4r-1}\Big( \frac{2\tau-1}{4}\Big),\label{def-Wj}
\end{align}
and let $\boldsymbol{W}_{2r-1}(\tau)=(W_1(\tau), \ldots, W_{2r-1}(\tau))^T$.
Then
\begin{align}
    \boldsymbol{V}_{2r-1}^*(\tau)=P\; \boldsymbol{W}_{2r-1}(\tau),\label{iden-VW}
\end{align}
where
$P=(p_{j,k})_{(2r-1)\times (2r-1)}$, and
\begin{align*}
  p_{j,k}=
  \begin{cases}
     e^{\pi i\frac{(4r-4j+1)^2}{32r-8}}, & \text{if } 1\leq j\leq 2r-1,
     k=2r-2j+1 \text{ or } 2j-2r,\\[5pt]
     0, & \text{otherwise}.
  \end{cases}
\end{align*}
It follows from \eqref{def-Wj} that
\begin{align*}
    W_j\left(-\frac{1}{2\tau}\right) &=\frac{\mathfrak{f}_2(-1/\tau)}{\eta(-1/\tau)}
    g_{2j-1,4r-1}\Big(-\frac{\tau+1}{4\tau}\Big).
\end{align*}
Applying Lemma \ref{ww-lem} and using \eqref{eta-trans} and \eqref{f-p-trans}, we have
\begin{align}
    &W_j\left(-\frac{1}{2\tau}\right) \notag\\[5pt]
    &=\sqrt{\frac{i}{2(4r-1)}}
    \frac{\mathfrak{f}_1(\tau)}{\eta(\tau)}
    \sum_{\substack{1\leq \ell\leq 4r-2 \atop \ell \text{ odd}}}
    \left( e^{\pi i \frac{2-2j-\ell-(2j+\ell-3)^2r}{8r-2}}
    +e^{\pi i \frac{2-2j+\ell-(2j-\ell-3)^2r}{8r-2}}\right) g_{\ell,4r-1}
    \Big(\frac{\tau+1}{4}\Big)\notag\\[5pt]
    &=\sqrt{\frac{i}{2(4r-1)}}
    \sum_{\ell=1}^{2r-1}
    \frac{\mathfrak{f}_1(\tau)}{\eta(\tau)}
    \left( e^{-\pi i \frac{2j+2\ell-3+(2j+2\ell-4)^2r}{8r-2}}
    +e^{-\pi i \frac{2j-2\ell-1+(2j-2\ell-2)^2r}{8r-2}}\right) g_{2\ell-1,4r-1}
    \Big(\frac{\tau+1}{4}\Big).\label{trans-W}
\end{align}
By \eqref{Uj-equi}, we obtain
\begin{align}
    U_j(\tau+1)&=\frac{\mathfrak{f}(\tau+1)}{\eta(\tau+1)}g_{2j-1,4r-1}
    \Big(\frac{\tau+1}{4}\Big)\notag\\[5pt]
    &=e^{-\frac{\pi i}{8}}\frac{\mathfrak{f}_1(\tau)}{\eta(\tau)}
    g_{2j-1,4r-1}\Big(\frac{\tau+1}{4}\Big) (\text{by } \eqref{eta-trans} \text{ and } \eqref{f-a-trans}).\label{trans-U-a}
\end{align}
It follows from \eqref{trans-W} and \eqref{trans-U-a} that
\begin{align*}
W_j(-\frac{1}{2\tau})=
\frac{1}{\sqrt{8r-2}}e^{\frac{3\pi i}{8}}
\sum_{\ell=1}^{2r-1}
\left( e^{-\pi i \frac{2j+2\ell-3+(2j+2\ell-4)^2r}{8r-2}}
+e^{-\pi i \frac{2j-2\ell-1+(2j-2\ell-2)^2r}{8r-2}}\right) U_\ell(\tau+1),
\end{align*}
which leads to
\begin{align}
    \boldsymbol{W}_{2r-1}\Big(-\frac{1}{2\tau}\Big)=X\; \boldsymbol{U}_{2r-1}(\tau+1),\label{iden-WU}
\end{align}
where $X=(x_{j,\ell})_{(2r-1)\times (2r-1)}$,
and
\begin{align*}
    x_{j,\ell}=\frac{1}{\sqrt{8r-2}}e^{\frac{3\pi i}{8}}
    (e^{-\pi i \frac{2j+2\ell-3+(2j+2\ell-4)^2r}{8r-2}}
    +e^{-\pi i \frac{2j-2\ell-1+(2j-2\ell-2)^2r}{8r-2}}).
\end{align*}
In view of \eqref{def-Uj} and \eqref{def-Ujs}, we have
\begin{align*}
U^*_j(2\tau)=(-1)^{\frac{j(j-1)}{2}}e^{-\frac{2j^2-2j-2r+1}{16r-4}\pi i} U_j(2\tau+1).
\end{align*}
Then
\begin{align}
    \boldsymbol{U}_{2r-1}^*(\tau)=\Lambda\; \boldsymbol{U}_{2r-1}(\tau+1),\label{trans-U-Us-g}
\end{align}
where
$\Lambda=(\lambda_{j,k})_{(2r-1)\times (2r-1)}$ and
\begin{align*}
\lambda_{j,k}=
\begin{cases}
    0, &\text{if } j\neq k,\\
    (-1)^{\frac{j(j-1)}{2}}e^{-\frac{2j^2-2j-2r+1}{16r-4}\pi i},
    & \text{if } j=k.
\end{cases}
\end{align*}
Combining \eqref{iden-VW}, \eqref{iden-WU} and \eqref{trans-U-Us-g}
yields
\begin{align}
    \boldsymbol{V}_{2r-1}^*\Big(-\frac{1}{2\tau}\Big)=PX\Lambda^{-1}\; \boldsymbol{U}_{2r-1}^*(\tau).\label{trans-VUS}
\end{align}
Next we show that $PX\Lambda^{-1}=S$.
By the definitions of $P, X$ and $\Lambda$,
for $1\leq j \leq r$, the $(j,k)$-entry of $PX\Lambda^{-1}$
is equal to
\begin{align*}
    &\frac{1}{\sqrt{8r-2}}(-1)^{\frac{k(k-1)}{2}}
    e^{\frac{\pi i (4r-4j+1)^2}{32r-8}+\frac{3\pi i}{8}+\frac{\pi i (2k^2-2k-2r+1)}{16r-4}}\\[5pt]
   &\quad \cdot \Big(e^{-\pi i \frac{4r-4j+2k-1+4r(2r-2j+k-1)^2}{8r-2}}+e^{-\pi i
    \frac{4r-4j-2k+1+4r(2r-2j-k)^2}{8r-2}}\Big)\\[5pt]
    &=\sqrt{\frac{2}{4r-1}}\cos\frac{(2j-1)(2k-1)\pi}{2(4r-1)},
\end{align*}
and for $r+1\leq j \leq 2r-1$, the $(j,k)$-entry of $PX\Lambda^{-1}$
is equal to
\begin{align*}
    &\frac{1}{\sqrt{8r-2}}(-1)^{\frac{k(k-1)}{2}}
    e^{\frac{\pi i (4r-4j+1)^2}{32r-8}+\frac{3\pi i}{8}+\frac{\pi i (2k^2-2k-2r+1)}{16r-4}}\\[5pt]
   &\quad \cdot \Big(e^{-\pi i \frac{4j-4r+2k-3+4r(2r-2j-k+2)^2}{8r-2}}+e^{-\pi i
    \frac{4j-4r-2k-1+4r(2r-2j+k+1)^2}{8r-2}}\Big)\\[5pt]
    &=\sqrt{\frac{2}{4r-1}}\cos\frac{(2j-1)(2k-1)\pi}{2(4r-1)}.
\end{align*}
Therefore, $PX\Lambda^{-1}=S$,
so \eqref{trans-Vs-Us} holds. This completes the proof of Theorem \ref{vmf-G}.
\end{proof}

We also derived the following transformation formulas for
$\boldsymbol{g}_{2r-1}(\tau)$ and $\boldsymbol{g}_{2r-1}^{\vee}(\tau)$.

\begin{thm}\label{tran-transl-G}
For any integer $r$,
\begin{align}\label{trans-G2}
    \boldsymbol{g}_{2r-1}(\tau+2)=
    T\,
\boldsymbol{g}_{2r-1}(\tau) \text{ and }
\boldsymbol{g}_{2r-1}^{\vee}(\tau+1)=
    {T}^{\vee}\,
\boldsymbol{g}_{2r-1}^{\vee}(\tau)
\end{align}
where
$T=(t_{j,k})_{(2r-1)\times (2r-1)},
{T}^{\vee}=({t}^{\vee}_{j,k})_{(2r-1)\times (2r-1)}$, and
\begin{align*}
    t_{j,k}=
    \begin{cases}
    0, & \text{if } j\neq k,\\
    e^{\frac{(4r-4j+1)^2}{8r-2}\pi i}, & \text{if } j=k,
    \end{cases}\quad
    {t}^{\vee}_{j,k}=
    \begin{cases}
    0, & \text{if } j\neq k,\\
    (-1)^{\frac{j(j-1)}{2}}e^{\frac{2j^2-2j-2r+1}{16r-4}\pi i}, & \text{if } j=k.
    \end{cases}
\end{align*}
Moreover, if $r$ is odd, then
\begin{align}\label{trans-G1}
\boldsymbol{g}_{2r-1}(\tau+1)=
        \widehat{T}\; \boldsymbol{g}_{2r-1}(\tau),
\end{align}
where $\widehat{T}=(\hat{t}_{j,k})_{(2r-1)\times (2r-1)}$, and
\begin{align*}
    \hat{t}_{j,k}=
    \begin{cases}
    0, & \text{if } j\neq k,\\
    e^{\frac{(4r-4j+1)^2}{16r-4}\pi i}, & \text{if } j=k.
    \end{cases}
\end{align*}
\end{thm}

\begin{proof}
By the definitions of $\boldsymbol{g}_{2r-1}(\tau)$ and $\boldsymbol{g}_{2r-1}^{\vee}(\tau)$,
it is easy to see that for integer $r$ and $0\leq j\leq 2r-1$,
\begin{align*}
g_j(\tau+2)&= e^{\frac{(4r-4j+1)^2}{8r-2}\pi i} g_j(\tau),\\[5pt]
g_j^{\vee}(2\tau+1)&=(-1)^{\frac{j(j-1)}{2}}e^{\frac{2j^2-2j-2r+1}{16r-4}\pi i} g_j^{\vee}(2\tau),
\end{align*}
which implies \eqref{trans-G2}.

If $r$ is odd,
then
\begin{align*}
g_j(\tau+1)&= e^{\frac{(4r-4j+1)^2}{16r-4}\pi i} g_j(\tau),
\end{align*}
which implies \eqref{trans-G1}.
\end{proof}

Now we can give a proof of Theorem \ref{vmf-G-AF}.

\begin{proof}[Proof of  Theorem \ref{vmf-G-AF}]
Let
\begin{align*}
    M= \begin{pmatrix}
        \boldsymbol{O} &2S\\
        S & \boldsymbol{O}
    \end{pmatrix},\quad
    Q=\begin{pmatrix}
        T & \boldsymbol{O}\\
        \boldsymbol{O} & ({T}^{\vee})^2
    \end{pmatrix}, \quad
    \widehat{Q}=\begin{pmatrix}
        \widehat{T} & \boldsymbol{O}\\
        \boldsymbol{O} & {T}^{\vee}
    \end{pmatrix},
\end{align*}
where $S$ is defined in Theorem \ref{vmf-G},
$T, {T}^{\vee}$ and $\widehat{T}$ are given in Theorem \ref{tran-transl-G}. From Theorem \ref{vmf-G} and Theorem \ref{tran-transl-G},
we see that for  $r\geq 2$,
\begin{align*}
\boldsymbol{G}_{4r-2}(\tau+2)&=Q\; \boldsymbol{G}_{4r-2}(\tau),\\[5pt]
\boldsymbol{G}_{4r-2}\left(\frac{\tau}{4\tau+1}\right)&=\boldsymbol{G}_{4r-2}\left(-\frac{1}{-2(2+\frac{1}{2\tau})}\right)=M\; \boldsymbol{G}_{4r-2}\left(-2-\frac{1}{2\tau}\right)\\[5pt]
&=MQ^{-1}\; \boldsymbol{G}_{4r-2}\left(-\frac{1}{2\tau}\right)
=MQ^{-1}M\; \boldsymbol{G}_{4r-2}(\tau).
\end{align*}
Let $\Gamma'$ be the group generated by $\gamma_1=\left(1~~2\atop 0~~1 \right)$
and $\gamma_2=\left(1~~0\atop 4~~1\right)$, and let
the multiplicative function
$\rho\colon \Gamma'\rightarrow GL_{4r-2}(\mathbb{C})$
satisfy that $\rho(\gamma_1)=Q$ and $\rho(\gamma_2)=MQ^{-1}M$.
Hence
$\boldsymbol{G}_{4r-2}(\tau)$ is a vector-valued automorphic form
of multiplier $\rho$ for $\Gamma'$.

If $r$ is odd, then
\begin{align}
\boldsymbol{G}_{4r-2}(\tau+1)&=\widehat{Q}\ \boldsymbol{G}_{4r-2}(\tau),\label{G-1-tansl}\\[5pt]
\boldsymbol{G}_{4r-2}\left(\frac{\tau}{2\tau+1}\right)
&=\boldsymbol{G}_{4r-2}\left(-\frac{1}{-2(1+\frac{1}{2\tau})}\right)
=M\ \boldsymbol{G}_{4r-2}\left(-\frac{1}{2\tau}-1\right)\notag\\
&=M\widehat{Q}^{-1}\ \boldsymbol{G}_{4r-2}\left(-\frac{1}{2\tau}\right)
=M\widehat{Q}^{-1}M\ \boldsymbol{G}_{4r-2}(\tau).\label{trans-G-2g}
\end{align}
Notice that
\begin{align*}
\widehat{\gamma}_1=
\begin{pmatrix}
1& 1\\
0& 1
\end{pmatrix} \text{ and }
\widehat{\gamma}_2=
\begin{pmatrix}
1&0\\
2&1
\end{pmatrix}
\end{align*}
are all generators of $\Gamma_0(2)$.
Let
the multiplicative function
$\widehat{\rho}\colon \Gamma_0(2)\rightarrow GL_{4r-2}(\mathbb{C})$
satisfy that $\widehat{\rho}(\widehat{\gamma}_1)=\widehat{Q}$ and
$\widehat{\rho}(\widehat{\gamma}_2)=M\widehat{Q}^{-1}M$.
It follows from \eqref{G-1-tansl} and \eqref{trans-G-2g}
that for any $\gamma\in\Gamma_0(2)$,
\begin{align*}
\boldsymbol{G}_{4r-2}(\gamma\tau)=\widehat{\rho}(\gamma) \boldsymbol{G}_{4r-2}(\tau).
\end{align*}
Since for any $1\leq j\leq 2r-1$,
$g_j(\tau)$ and $g_j^{\vee}(\tau)$ can be expressed as sums of generalized
eta-quotients, by \cite[Lemma 2.6]{CDZ-2019},
for any $\gamma\in\Gamma$, there exists positive integers $N_{\gamma,j}$
and $N^{\vee}_{\gamma,j}$
such that
\[g_j(\gamma\tau)=\sum_{n}a_{n,j}q^{n/N_{\gamma,j}}\]
and
\[g_j^{\vee}(\gamma\tau)=\sum_{n}a^{\vee}_{n,j}q^{n/N^{\vee}_{\gamma,j}},\]
where $a_{n,j}=0$ and $a^{\vee}_{n,j}=0$ for $n \ll 0$.
Therefore,
$\boldsymbol{G}_{4r-2}(\tau)$ is a vector-valued modular function
associated with $\widehat{\rho}$ for $\Gamma_0(2)$.
\end{proof}

\subsection{Proof of Theorem \ref{vmf-H-AF}}\label{proof-thmH}

For  $r\geq 2$  and $1\leq j\leq 2r-2$, recall that
\begin{align*}
h_j(\tau)&=q^{\frac{(4 j-4 r+1)^2}{32 r-24}}\left(
q^{\frac{4 j-2 r-1}{4}}
\frac{(q^{4(2r-1-j)}, q^{8r+4j-8},q^{16r-12}; q^{16r-12})_\infty}{(q,q^3,q^4;q^4)_\infty}\right.\\[5pt]
&\quad \left.
+
\frac{(q^{2(2r-j-1)},-q^{2j-1},-q^{4r-3};-q^{4r-3})_\infty}{(q^2,q^2,q^4;q^4)_\infty}\right),\\
h^{\vee}_j(2\tau)&=\frac{1}{2}q^{\frac{j^2-j-r+1}{8r-6}}
\left(\frac{(q^{2r-j-1},q^{2r+j-2},q^{4r-3};q^{4r-3})_\infty}
{(q,q^3,q^4;q^4)_\infty}\right.\\[5pt]
&\quad \left.+
\frac{((-q)^{2r-j-1},(-q)^{2r+j-2},-q^{4r-3};-q^{4r-3})_\infty}
{(-q,-q^3,q^4;q^4)_\infty}\right).
\end{align*}

Let
\begin{align*}
\boldsymbol{h}_{2r-2}(\tau)=(h_1(\tau), \ldots, h_{2r-2}(\tau))^T,\quad
\boldsymbol{h}_{2r-2}^{\vee}(\tau)=(h_1^{\vee}(\tau), \ldots, h_{2r-2}^{\vee}(\tau))^T.
\end{align*}

We derive the following transformation formula on
the ``Langlands dual'' pair
$\boldsymbol{h}_{2r-2}(\tau)$
and
$\boldsymbol{h}_{2r-2}^{\vee}(\tau)$.

\begin{thm}\label{vmf-H}
For  $r\geq 2$,
\begin{align}
\boldsymbol{h}_{2r-2}\left(-\frac{1}{2\tau}\right)=2\widetilde{S}\; \boldsymbol{h}_{2r-2}^{\vee}(\tau),\quad
\boldsymbol{h}_{2r-2}^{\vee}\left(-\frac{1}{2\tau}\right)=\widetilde{S}\; \boldsymbol{h}_{2r-2}(\tau),
\label{trans-hs-hv}
\end{align}
where $\widetilde{S}=(\widetilde{s}_{j,k})_{(2r-2)\times (2r-2)}$,
and
$\widetilde{s}_{j,k} =\sqrt{\frac{2}{4r-3}}\cos\frac{(2j-1)(2k-1)}{8r-6}\pi$.
\end{thm}

\begin{proof}
Taking $N=4r-3$ in Lemma \ref{inv-S},
in order to prove \eqref{trans-hs-hv},
we only need to prove
\begin{align}
\boldsymbol{h}_{2r-2}\left(-\frac{1}{2\tau}\right)=
2\widetilde{S}\; \boldsymbol{h}_{2r-2}^{\vee}(\tau),\label{trans-hs-h}
\end{align}

For $1\leq j\leq 2r-2$,
let
\begin{align}
\widetilde{V}_j(\tau)&=q^{\frac{16j^2-16j-4r+7}{32r-24}}
\frac{(q^{4(2r-1-j)}, q^{8r+4j-8},q^{16r-12}; q^{16r-12})_\infty}{(q,q^3,q^4;q^4)_\infty},\label{H-V-j-1}\\
\widetilde{V}_j^*(\tau)&=q^{\frac{(4 j-4 r+1)^2}{32 r-24}}
\frac{(q^{2(2r-j-1)},-q^{2j-1},-q^{4r-3};-q^{4r-3})_\infty}{(q^2,q^2,q^4;q^4)_\infty},\label{H-V-j}\\
\widetilde{U}_j(2\tau)&=q^{\frac{j^2-j-r+1}{8r-6}}
\frac{(q^{2r-j-1},q^{2r+j-2},q^{4r-3};q^{4r-3})_\infty}
{(q,q^3,q^4;q^4)_\infty},\label{def-U-H}\\
\widetilde{U}_j^*(2\tau)&=q^{\frac{j^2-j-r+1}{8r-6}}
\frac{((-q)^{2r-j-1},(-q)^{2r+j-2},-q^{4r-3};-q^{4r-3})_\infty}
{(-q,-q^3,q^4;q^4)_\infty}.\label{def-U-s-H}
\end{align}
Then
\begin{align}
h_j(\tau)=\widetilde{V}_j(\tau)+\widetilde{V}_j^*(\tau),\quad
h^{\vee}_j(\tau)=\frac{1}{2}(\widetilde{U}_j(\tau)+\widetilde{U}_j^*(\tau)).\label{def-h-v-j}
\end{align}
Let
\begin{align*}
\boldsymbol{\widetilde{V}}_{2r-2}(\tau)&=(\widetilde{V}_1(\tau), \ldots, \widetilde{V}_{2r-2}(\tau))^T,\\[5pt]
\boldsymbol{\widetilde{V}}_{2r-2}^*(\tau)&=(\widetilde{V}^*_1(\tau), \ldots, \widetilde{V}^*_{2r-2}(\tau))^T,\\[5pt]
\boldsymbol{\widetilde{U}}_{2r-2}(\tau)&=(\widetilde{U}_1(\tau),\ldots, \widetilde{U}_{2r-2}(\tau))^T,
\end{align*}
and
$$\boldsymbol{\widetilde{U}}_{2r-2}^*(\tau)=(\widetilde{U}_1^*(\tau), \ldots, \widetilde{U}_{2r-2}^*(\tau))^T.$$
Then
\begin{align*}
\boldsymbol{h}_{2r-2}(\tau)=\boldsymbol{\widetilde{V}}_{2r-2}(\tau)+\boldsymbol{\widetilde{V}}_{2r-2}^*(\tau),\quad
\boldsymbol{h}_{2r-2}^{\vee}(\tau)=\frac{1}{2}(\boldsymbol{\widetilde{U}}_{2r-2}(\tau)+\boldsymbol{\widetilde{U}}_{2r-2}^*(\tau)).
\end{align*}
Hence, in order to prove \eqref{trans-hs-h},
it suffices to prove
\begin{align}\label{trans-V-h}
\boldsymbol{\widetilde{V}}_{2r-2}\left(-\frac{1}{2\tau}\right)=\widetilde{S}\; \boldsymbol{\widetilde{U}}_{2r-2}(\tau),
\end{align}
and
\begin{align}\label{trans-V-s-h}
\boldsymbol{\widetilde{V}}_{2r-2}^*\left(-\frac{1}{2\tau}\right)=\widetilde{S}\; \boldsymbol{\widetilde{U}}_{2r-2}^*(\tau).
\end{align}

\textbf{Proof of  \eqref{trans-V-h}.}~
It follows from \eqref{H-V-j-1} and \eqref{H-V-j} that
\begin{align}
\widetilde{V}_j(\tau)&=\frac{\mathfrak{f}(2\tau)}{\eta(2\tau)}
g_{2j-1,4r-3}(2\tau),\notag\\[5pt]
\widetilde{U}_j(2\tau)&=\frac{\mathfrak{f}(2\tau)}{\eta(2\tau)}
g_{2j-1,4r-3}\left(\frac{\tau}{2}\right).\label{def-U}
\end{align}
So,
\begin{align*}
\widetilde{V}_j\left(-\frac{1}{4\tau}\right)
=\frac{\mathfrak{f}\left(-\frac{1}{2\tau}\right)}{\eta\left(-\frac{1}{2\tau}\right)}
g_{2j-1,4r-3}\left(-\frac{1}{2\tau}\right).
\end{align*}
By means of \eqref{eta-trans}, \eqref{f-p-trans} and Lemma \ref{Wak-lem}, we have
\begin{align*}
&\quad \widetilde{V}_j\left(-\frac{1}{4\tau}\right)\\[5pt]
&=\frac{1}{\sqrt{8r-6}}
\frac{\mathfrak{f}(2\tau)}{\eta(2\tau)}
\sum_{k=0}^{8r-7}e^{\frac{(2j-1)(2k+1)}{8r-6}\pi i}h_{k+\frac{1}{2},4r-3}(2\tau)\\[5pt]
&=\frac{1}{\sqrt{8r-6}}
\frac{\mathfrak{f}(2\tau)}{\eta(2\tau)}
\sum_{k=0}^{4r-4}
\left(e^{\frac{(2j-1)(2k+1)}{8r-6}\pi i}h_{k+\frac{1}{2},4r-3}(2\tau)
+e^{-\frac{(2j-1)(2k+1)}{8r-6}\pi i}h_{k+\frac{1}{2},4r-3}(2\tau)\right)\\[5pt]
&=\sqrt{\frac{2}{4r-3}}\sum_{k=0}^{4r-4}
\cos\frac{(2j-1)(2k+1)\pi}{8r-6}
\frac{\mathfrak{f}(2\tau)}{\eta(2\tau)}
h_{k+\frac{1}{2},4r-3}(2\tau)\\[5pt]
&=\sqrt{\frac{2}{4r-3}}\sum_{k=0}^{2r-3}
\cos\frac{(2j-1)(2k+1)\pi}{8r-6}
\frac{\mathfrak{f}(2\tau)}{\eta(2\tau)}
(h_{k+\frac{1}{2},4r-3}(2\tau)-h_{4r-k-\frac{7}{2},4r-3}(2\tau)).
\end{align*}
In view of \eqref{iden-h-2}, the above identity can be rewritten as
\begin{align*}
&\quad \widetilde{V}_j\left(-\frac{1}{4\tau}\right)\\[5pt]
&=\sqrt{\frac{2}{4r-3}}\sum_{k=0}^{2r-3}
\cos\frac{(2j-1)(2k+1)\pi}{8r-6}
\frac{\mathfrak{f}(2\tau)}{\eta(2\tau)}
\left(h_{4k+2,16r-12}\left(\frac{\tau}{2}\right)
-h_{16r-4k-14,16r-12}\left(\frac{\tau}{2}\right)\right)\\[5pt]
&=\sqrt{\frac{2}{4r-3}}\sum_{k=0}^{2r-3}
\cos\frac{(2j-1)(2k+1)\pi}{8r-6}
\frac{\mathfrak{f}(2\tau)}{\eta(2\tau)}
g_{2k+1,4r-3}\left(\frac{\tau}{2}\right)\\[5pt]
&=\sqrt{\frac{2}{4r-3}}\sum_{k=1}^{2r-2}
\cos\frac{(2j-1)(2k-1)\pi}{8r-6}
\frac{\mathfrak{f}(2\tau)}{\eta(2\tau)}
g_{2k-1,4r-3}\left(\frac{\tau}{2}\right)\\[5pt]
&=\sqrt{\frac{2}{4r-3}}\sum_{k=1}^{2r-2}
\cos\frac{(2j-1)(2k-1)\pi}{8r-6}\widetilde{U}_k(2\tau)~(\text{by } \eqref{def-U}).
\end{align*}
Thus,
\begin{align}\label{trans-Vj}
\widetilde{V}_j\left(-\frac{1}{2\tau}\right)=
\sqrt{\frac{2}{4r-3}}\sum_{k=1}^{2r-2}
\cos\frac{(2j-1)(2k-1)\pi}{8r-6}\widetilde{U}_k(\tau),
\end{align}
which implies \eqref{trans-V-h}.
%$S=(s_{j,k})_{(2r-2)\times (2r-2)}$,
%and
%\begin{align*}
%s_{j,k} =\sqrt{\frac{2}{4r-3}}\cos\left(\frac{(2j-1)(2k-1)}{8r-6}\pi\right).
%\end{align*}

\textbf{Proof of  \eqref{trans-V-s-h}.}~
By \eqref{H-V-j}, it is easy to see that
\begin{align*}
\widetilde{V}_j^*\left(\tau+\frac{1}{2}\right)
&= e^{\frac{(4r-4j-1)^2}{32r-24}\pi i}
q^{\frac{(4r-4j-1)^2}{32r-24}}
\frac{(q^{2(2r-j-1)},q^{2j-1},q^{4r-3};q^{4r-3})_\infty}
{(q^2,q^2,q^4;q^4)_\infty}\\[5pt]
&=\begin{cases}
e^{\frac{(4r-4j-1)^2}{32r-24}\pi i}\frac{\mathfrak{f}_2(2\tau)}{\eta(2\tau)}
g_{4r-4j-1,4r-3}\left(\frac{\tau}{2}\right), & \text{if } 1\leq j\leq r-1,\\
e^{\frac{(4r-4j-1)^2}{32r-24}\pi i}\frac{\mathfrak{f}_2(2\tau)}{\eta(2\tau)}
g_{4j-4r+1,4r-3}\left(\frac{\tau}{2}\right), & \text{if } r\leq j\leq 2r-2.
\end{cases}
\end{align*}
So,
\begin{align*}
 \widetilde{V}_j^*(\tau)&=
 e^{\frac{(4r-4j-1)^2}{32r-24}\pi i}\frac{\mathfrak{f}_2(2\tau-1)}{\eta(2\tau-1)}
g_{|4r-4j-1|,4r-3}\left(\frac{2\tau-1}{4}\right)\\[5pt]
&=e^{\frac{(4r-4j-1)^2}{32r-24}\pi i}\frac{\mathfrak{f}_2(2\tau)}{\eta(2\tau)}
g_{|4r-4j-1|,4r-3}\left(\frac{2\tau-1}{4}\right)~(\text{by } \eqref{eta-trans}
\text{ and } \eqref{f-a-trans}).
\end{align*}
Let
\begin{align}
\widetilde{W}_j(\tau)=\frac{\mathfrak{f}_2(2\tau)}{\eta(2\tau)}
g_{2j-1,4r-3}\left(\frac{2\tau-1}{4}\right),\label{def-W-H}
\end{align}
and $\boldsymbol{\widetilde{W}}_{2r-2}(\tau)=(\widetilde{W}_1(\tau), \ldots, \widetilde{W}_{2r-2}(\tau))^T$.
Then
\begin{align}
\boldsymbol{\widetilde{V}}_{2r-2}^*(\tau)=\widetilde{P}\; \boldsymbol{\widetilde{W}}_{2r-2}(\tau),\label{trans-Vs-W}
\end{align}
where
$\widetilde{P}=(\widetilde{p}_{j,k})_{(2r-2)\times (2r-2)}$,
\begin{align*}
\widetilde{p}_{j,k}=\begin{cases}
e^{\frac{(4r-4j-1)^2}{32r-24}\pi i}, & \text{if } 1\leq j \leq 2r-2,
k=2r-2j, \text{or } k=2j-2r+1,\\[5pt]
0, & \text{otherwise}.
\end{cases}
\end{align*}
By means of \eqref{def-W-H}, we have
\begin{align*}
&\widetilde{W}_j\left(-\frac{1}{2\tau}\right)\\[5pt]
&=
\frac{\mathfrak{f}_2\left(-\frac{1}{\tau}\right)}{\eta\left(-\frac{1}{\tau}\right)}g_{2j-1, 4r-3}\left(-\frac{\tau+1}{4\tau}\right)\\[5pt]
&=\frac{1}{\sqrt{2i(4r-3)}}\frac{\mathfrak{f}_1(\tau)}{\eta(\tau)}\\[5pt]
&\quad \sum_{1\leq \ell \leq 4r-4\atop \ell \text{ odd}}
\left(e^{\frac{2-2j-\ell-(2j+\ell-3)^2(3r-2)}{8r-6}\pi i}
+e^{\frac{2-2j+\ell-(2j-\ell-3)^2(3r-2)}{8r-6}\pi i}\right)
g_{\ell, 4r-3}\left(\frac{\tau+1}{4}\right)\\[5pt]
&=\frac{1}{\sqrt{2i(4r-3)}}\frac{\mathfrak{f}_1(\tau)}{\eta(\tau)}\\[5pt]
&\quad\quad \sum_{\ell=1}^{2r-2}\left(e^{-\frac{2j+2\ell-3+(2j+2\ell-4)^2(3r-2)}{8r-6}\pi i}+ e^{-\frac{2j-2\ell-1+(2j-2\ell-2)^2(3r-2)}{8r-6}\pi i}\right)
g_{2\ell-1, 4r-3}\left(\frac{\tau+1}{4}\right)\\[5pt]
&=\frac{1}{\sqrt{2(4r-3)}}e^{-\frac{\pi i}{8}}
\sum_{\ell=1}^{2r-2}\left(e^{-\frac{2j+2\ell-3+(2j+2\ell-4)^2(3r-2)}{8r-6}\pi i}+ e^{-\frac{2j-2\ell-1+(2j-2\ell-2)^2(3r-2)}{8r-6}\pi i}\right)
\widetilde{U}_j(\tau+1),
\end{align*}
where the last identity follows from
\begin{align*}
\widetilde{U}_j(\tau+1)&=\frac{\mathfrak{f}(\tau+1)}{\eta(\tau+1)}
g_{2j-1,4r-3}\left(\frac{\tau+1}{4}\right)\\[5pt]
&=e^{-\frac{\pi i}{8}}
\frac{\mathfrak{f}_1(\tau)}{\eta(\tau)}
g_{2j-1,4r-3}\left(\frac{\tau+1}{4}\right)~(\text{by } \eqref{eta-trans}
\text{ and } \eqref{f-a-trans}).
\end{align*}
So
\begin{align}
\boldsymbol{\widetilde{W}}_{2r-2}\left(-\frac{1}{2\tau}\right)=\widetilde{X}\;  \boldsymbol{\widetilde{U}}_{2r-2}(\tau+1),\label{trans-W-U}
\end{align}
where
$\widetilde{X}=(\widetilde{x}_{j,\ell})_{(2r-2)\times (2r-2)}$ and
\begin{align*}
\widetilde{x}_{j, \ell}=\frac{1}{\sqrt{2(4r-3)}}e^{-\frac{\pi i}{8}}\left(e^{-\frac{2j+2\ell-3+(2j+2\ell-4)^2(3r-2)}{8r-6}\pi i}+ e^{-\frac{2j-2\ell-1+(2j-2\ell-2)^2(3r-2)}{8r-6}\pi i}\right).
\end{align*}
By \eqref{def-U-H} and \eqref{def-U-s-H}, we have
\begin{align*}
\widetilde{U}_j(2\tau+1)=e^{\frac{j^2-j-r+1}{8r-6}\pi i}\widetilde{U}_j^*(2\tau),
\end{align*}
and so
\begin{align*}
\widetilde{U}_j(\tau+1)=e^{\frac{j^2-j-r+1}{8r-6}\pi i}\widetilde{U}_j^*(\tau).
\end{align*}
Thus,
\begin{align}
\boldsymbol{\widetilde{U}}_{2r-2}(\tau+1)=\widetilde{\Lambda}\; \boldsymbol{\widetilde{U}}_{2r-2}^*(\tau),\label{trans-U-Us}
\end{align}
where $\widetilde{\Lambda}=(\widetilde{\lambda}_{j,k})_{(2r-2)\times (2r-2)}$,
and
\begin{align*}
\widetilde{\lambda}_{j,k}=\begin{cases}
e^{\frac{j^2-j-r+1}{8r-6}\pi i}, &\text{if } j=k,\\
0, &\text{otherwise}.
\end{cases}
\end{align*}
Therefore, from \eqref{trans-Vs-W}, \eqref{trans-W-U} and \eqref{trans-U-Us},
we have
\begin{align*}
\boldsymbol{\widetilde{V}}_{2r-2}^*\left(-\frac{1}{2\tau}\right)=\widetilde{P}\widetilde{X}\widetilde{\Lambda}\; \boldsymbol{\widetilde{U}}_{2r-2}^*(\tau).
\end{align*}
For $1\leq j\leq r-1$, the $(j,k)$-entry of $\widetilde{P}\widetilde{X}\widetilde{\Lambda}$ is equal to
\begin{align*}
&\frac{1}{\sqrt{2(4r-3)}}e^{-\frac{\pi i}{8}+\frac{(4r-4j-1)^2}{32r-24}\pi i+\frac{k^2-k-r+1}{8r-6}\pi i}
\left(e^{-\frac{2(2r-2j)+2k-3+(2(2r-2j)+2k-4)^2(3r-2)}{8r-6}\pi i}\right.\\[5pt]
&\quad\left.+ e^{-\frac{2(2r-2j)-2k-1+(2(2r-2j)-2k-2)^2(3r-2)}{8r-6}\pi i}\right)\\[5pt]
&=\sqrt{\frac{2}{4r-3}}\cos\frac{(2j-1)(2k-1)}{2(4r-3)}\pi,
\end{align*}
and
for $r\leq j \leq 2r-2$, the $(j,k)$-entry of $\widetilde{P}\widetilde{X}\widetilde{\Lambda}$ is equal to
\begin{align*}
&\frac{1}{\sqrt{2(4r-3)}}e^{-\frac{\pi i}{8}+\frac{(4r-4j-1)^2}{32r-24}\pi i+\frac{k^2-k-r+1}{8r-6}\pi i}
\left(e^{-\frac{2(2j-2r+1)+2k-3+(2(2j-2r+1)+2k-4)^2(3r-2)}{8r-6}\pi i}\right.\\[5pt]
&\quad\left.+ e^{-\frac{2(2j-2r+1)-2k-1+(2(2j-2r+1)-2k-2)^2(3r-2)}{8r-6}\pi i}\right)\\[5pt]
&=\sqrt{\frac{2}{4r-3}}\cos\frac{(2j-1)(2k-1)}{2(4r-3)}\pi,
\end{align*}
so $\widetilde{P}\widetilde{X}\widetilde{\Lambda}=\widetilde{S}$.
Thus, \eqref{trans-V-s-h} holds.
\end{proof}

\begin{thm}\label{thm-trans-transl-H}
For any $r\geq 2$,
\begin{align}\label{tran-tansl-H}
\boldsymbol{h}_{2r-2}(\tau+4)=\widetilde{T}\, \boldsymbol{h}_{2r-2}(\tau),\quad
\boldsymbol{h}_{2r-2}^\vee(\tau+1)=\widetilde{T}^{\vee}\, \boldsymbol{h}_{2r-2}^\vee(\tau),
\end{align}
where $\widetilde{T}=(\widetilde{t}_{j,k})_{(2r-2)\times (2r-2)}$, $\widetilde{T}^{\vee}=(\tilde{t}^{\vee}_{j,k})_{(2r-2)\times (2r-2)}$,
\begin{align*}
\widetilde{t}_{j,k}=\begin{cases}
e^{\frac{(4j-4r+1)^2}{4r-3}\pi i}, &\text{if } j=k,\\
0, &\text{otherwise},
\end{cases}\quad
\tilde{t}^{\vee}_{j,k}=\begin{cases}
e^{\frac{j^2-j-r+1}{8r-6}\pi i}, &\text{if } j=k,\\
0, &\text{otherwise}.
\end{cases}
\end{align*}
\end{thm}

\begin{proof}
 It is obvious that for  $1\leq j\leq 2r-2$,
\begin{align*}
h_j(\tau+4)&=e^{\frac{(4j-4r+1)^2}{4r-3}\pi i}h_j(\tau),\\[5pt]
h_j^\vee(2\tau+1)&=e^{\frac{j^2-j-r+1}{8r-6}\pi i}h_j^\vee(2\tau).
\end{align*}
The the proof is complete.
\end{proof}

\begin{proof}[Proof of  Theorem \ref{vmf-H-AF}]
Let
\begin{align*}
\widetilde{M}=\begin{pmatrix}
\boldsymbol{O} & 2 \widetilde{S}\\
\widetilde{S} & \boldsymbol{O}
\end{pmatrix},\quad
\widetilde{Q}=\begin{pmatrix}
    \widetilde{T}& \boldsymbol{O}\\
    \boldsymbol{O} &(\widetilde{T}^{\vee})^4
\end{pmatrix},
\end{align*}
where $\widetilde{S}$ is given in
Theorem \ref{vmf-H},
$\widetilde{T}$ and $\widetilde{T}^{\vee}$ are defined in
Theorem \ref{thm-trans-transl-H}.
By Theorem \ref{vmf-H} and Theorem \ref{thm-trans-transl-H}, we have
\begin{align*}
\boldsymbol{H}_{4r-4}(\tau+4)&=\widetilde{Q}\ \boldsymbol{H}_{4r-4}(\tau),\\[5pt]
\boldsymbol{H}_{4r-4}\left(\frac{\tau}{8\tau+1}\right)&
=\boldsymbol{H}_{4r-4}\left(-\frac{1}{-2(4+\frac{1}{2\tau})}\right)
=\widetilde{M}\; \boldsymbol{H}_{4r-4}\left(-4-\frac{1}{2\tau}\right)\\[5pt]
&=\widetilde{M}\widetilde{Q}^{-1}\; \boldsymbol{H}_{4r-4}\left(-\frac{1}{2\tau}\right)
=\widetilde{M}\widetilde{Q}^{-1}\widetilde{M}\; \boldsymbol{H}_{4r-4}(\tau).
\end{align*}
Let $\Gamma'$ be the group generated by $\gamma_1=\left(1~~4\atop 0~~1\right)$
and $\gamma_2=\left(1~~0\atop 8~~1\right)$ and
the multiplicative function $\widetilde{\rho}\colon \Gamma'\rightarrow GL_{4r-4}(\mathbb{C})$
satisfy that $\widetilde{\rho}(\gamma_1)=\widetilde{Q}$
and $\widetilde{\rho}(\gamma_2)=\widetilde{M}\widetilde{Q}^{-1}\widetilde{M}$.
Then for any $\gamma\in\Gamma'$,
\begin{align*}
\boldsymbol{H}_{4r-4}(\gamma\tau)=\widetilde{\rho}(\gamma) \boldsymbol{H}_{4r-4}(\tau),
\end{align*}
which
implies that $\boldsymbol{H}_{4r-4}(\tau)$ is a vector-valued automorphic form
of the multiplier $\widetilde{\rho}$
for $\Gamma'$.
\end{proof}

\vskip 0.2cm
\noindent{\bf Acknowledgment.} We thank Shashank Kanade for bringing references \cite{Kanade-Russell-2019} and \cite{Kursungoz-2019} to our intention. This work
was supported by the National Natural Science Foundation of China and the Hebei Natural Science Foundation (A2024205012).

\end{document}